\documentclass{amsart}

\usepackage{amsmath,amsfonts,amsthm,amscd,amssymb,graphicx}
\numberwithin{equation}{section}

\usepackage{hyperref}

\usepackage{color}

\newtheorem{theorem}{Theorem}[section]

\newtheorem{lemma}[theorem]{Lemma}
\newtheorem{lem}[theorem]{Lemma}
\newtheorem{proposition}[theorem]{Proposition}
 
\newtheorem{corollary}[theorem]{Corollary}

\newtheorem{remark}[theorem]{Remark}
\newtheorem{definition}[theorem]{Definition}
\newtheorem{rem}[theorem]{Remark}

\newcommand{\T}{{\mathbb T}}
\newcommand{\N}{{\mathbb N}}
\newcommand{\Z}{{\mathbb Z}}

\newcommand{\R}{{\mathbb R}}

\newcommand{\pa}{{\partial}}
\newcommand{\na}{{\nabla}}

\newcommand{\Hc}{\mathcal{H}}

\renewcommand{\a}{\alpha}

\renewcommand{\b}{\beta}

\renewcommand{\r}{\tilde{r}}

\newcommand{\F}{\mathcal{F}}

\newcommand{\Tc}{\mathcal{T}}

\begin{document}

\title[Propagation of higher space regularity for Vlasov equations]{On propagation of higher space regularity for non-linear Vlasov equations}

\author{Daniel Han-Kwan}
  \address{CMLS \\ \'Ecole polytechnique, CNRS, Universit\'e Paris-Saclay, 91128 Palaiseau Cedex, France}
  \email{daniel.han-kwan@polytechnique.edu}

\date{\today}

\maketitle

\renewcommand{\thefootnote}{\fnsymbol{footnote}}

\begin{abstract}
This work is concerned with the broad question of propagation of regularity for smooth solutions to non-linear Vlasov equations.
For a class of equations (that includes Vlasov-Poisson and relativistic Vlasov-Maxwell), we prove that higher regularity in space is propagated, locally in time, into higher regularity for the moments in velocity of the solution. This in turn can be translated into some anisotropic Sobolev higher regularity for the solution itself, which can be interpreted as a kind of weak propagation of space regularity.
To this end, we adapt the methods introduced in the context of the quasineutral limit of the Vlasov-Poisson system in [D. Han-Kwan and F. Rousset,  \emph{Ann. Sci. \'Ecole Norm. Sup.}, 2016].

\end{abstract}

\tableofcontents

\section{Introduction}
This paper is concerned with the broad question of propagation of regularity for smooth solutions to Vlasov equations of the general form
\begin{equation}
\label{Vlasov}
\pa_t f + a(v)\cdot \na_x f + F(t,x,v) \cdot \na_v f =0,
\end{equation}
set  in the phase space $\T^d \times \R^d$ (with $\T^d= \R^d/\Z^d$ endowed with normalized Lebesgue measure), where $F: \R^+ \times \T^d \times \R^d \to \R^d$ is a force field satisfying $\na_v \cdot F=0$ and $a: \R^d \to \R^d$ is an advection field satisfying suitable assumptions, $a(v) =v$ being the main example to be considered. 
The (scalar) function $f(t,x,v)$ may be understood as the distribution function of a family of particles, which can be, depending on the physical context, e.g. electrons, ions in plasma physics, or stars in galactic dynamics. The choice of the periodic torus $\T^d$ is made for simplicity.

The two precise examples of equations we specifically have in mind are the Vlasov equations arising from a coupling with Poisson or Maxwell equations, in which case the  resulting coupled system is called
the Vlasov-Poisson or the relativistic Vlasov-Maxwell system (we will discuss as well several other models).

\noindent$\bullet$ The \emph{Vlasov-Poisson} system -- either the repulsive or the attractive version, the sign of the interaction here does not matter here -- reads
\begin{equation}\label{VP} 
\left \{ \begin{aligned}
&\partial_t f + {v} \cdot \nabla_x f \pm E\cdot \nabla_v f =0,\\
&E(t,x)= -\na_x \phi(t,x), \\
&-\Delta_x \phi = \int_{\R^d} f \, dv -   \int_{\T^d \times \R^d} f \, dv dx, \\
&f|_{t=0} = f_0.
\end{aligned}
\right.
\end{equation}

In the repulsive version (that is with the sign $+$ in the Vlasov equation), this system describes the dynamics of charged particles in a non-relativistic plasma, with a self-induced electric field.

In the attractive version (that is with the sign $-$ in the Vlasov equation), it describes the dynamics of stars or planets with gravitational interaction.

\noindent $\bullet$ The \emph{relativistic Vlasov-Maxwell} system, reads, in dimension $d=3$,
\begin{equation}\label{VM} 
\left \{ \begin{aligned}
&\partial_t f + \hat{v} \cdot \nabla_x f + F\cdot \nabla_v f =0,\\
&\hat{v} := \frac{v}{\sqrt{1+ |v|^2/c^2}}, \qquad F(t,x,v):= E(t,x) + \frac{1}{c} \hat{v} \times B(t,x), \\
&\frac{1}{c} \partial_t B + \nabla_x \times E = 0, \qquad \nabla_x \cdot E  = \int_{\R^3} f \, dv  - \int_{\T^3 \times \R^3} f \, dv  dx,
\\
&- \frac{1}{c} \partial_t E + \nabla_x \times B = \frac{1}{c} \int_{\R^3} \hat{v} f   \, dv  , \qquad \nabla_x \cdot B  =0, \\
&f|_{t=0} = f_0, \qquad (E,B)|_{t=0}=(E_0,B_0),
\end{aligned}
\right.
\end{equation}
in which the parameter $c$ is the speed of light. There are also  related versions of~\eqref{VM} in lower dimensions. This system  describes the dynamics of charged particles in a relativistic plasma, with a self-induced electro-magnetic field. We recall that the (repulsive) Vlasov-Poisson can be derived from~\eqref{VM} in the non-relativistic regime, that is to say in the limit $c \to + \infty$, as studied in \cite{AU,Deg,SCH}.

In this paper, we will consider weighted Sobolev norms and associated weighted Sobolev spaces (based on $L^2$), defined, for  $k \in \N, r \in \R$, as
\begin{equation}
\| f \|_{\Hc^{k}_{r}} := \left(\sum_{|\a| + |\b| \leq k} \int_{\T^d} \int_{\R^d} (1+ |v|^2)^{r} |\pa^\a_x \pa^\b_v f|^2 \, dv dx \right)^{1/2},
\end{equation}
where  for $\a = (\a_1, \cdots, \a_n), \b = (\b_1, \cdots, \b_n) \in \{1,\cdots, d\}^n$, we write
$$
|\a|= n, \quad |\b|=n,
$$
and
$$
\pa^\a_x := \pa_{x_{\a_1}} \cdots \pa_{x_{\a_n}},\quad \pa^\b_v := \pa_{v_{\b_1}} \cdots \pa_{v_{\b_n}}.
$$
As usual the notation $H^s$ will stand for the standard Sobolev spaces, without weight.

It will be also useful to introduce the following weighted $W^{k,\infty}$ space, whose norm is defined,
for  $k \in \N, r \in \R$, as
\begin{equation}
\| f \|_{\mathcal{W}^{k,\infty}_{r}} := \sum_{|\a| + |\b| \leq k} \|(1+|v|^2)^{r/2} \pa^\a_x \pa^\b_v f\|_{L^{\infty}_{x,v}} \end{equation}
For the Vlasov-Poisson or Vlasov-Maxwell couplings, given an initial condition $f_0$ satisfying
$$
f_0 \in \Hc^n_r
$$
for $n, r>0$ large enough (and with a smooth enough  initial force $F(0)$),
it is standard that there exists a unique local solution
$f(t) \in {C}(0,T; \Hc^n_r)$. 
Under fairly general assumptions on the advection field $a$ and the force $F$, the same result can also be shown for~\eqref{Vlasov}, as we will soon see.

Let us now present the precise problem we tackle in this work.
Assuming  some higher \emph{space} regularity such as
\begin{equation}
\label{extra}
\pa_x^{n+1} f_0 \in \Hc^0_r, \qquad \text{(or } \pa_x^{p} f_0 \in \Hc^0_r, \text{  for } p \geq n+1 \text{)}
\end{equation}
the question we ask is the following: is there also propagation of any higher regularity for the solution $f(t)$?
A first remark to be made is that there is no hope of proving that this sole additional assumption entails that the solution $f(t)$ also satisfies $\pa_x^{n+1} f(t) \in \Hc^0_r$ even for small values of $t$. Indeed,  regularity in $x$ and $v$ is intricately intertwined for solutions of the Vlasov equation, as can be seen from the representation of the solution using the method of characteristics.

For $s,t \geq 0$ and  $(x,v) \in \T^d \times \R^d$, we define as usual the characteristic curves $(X(s,t,x,v),V(s,t,x,v))$  as the solutions to the system of ODEs
\begin{equation}
\label{charac}
\left \{ 
\begin{aligned}
&\frac{d}{ds} X(s,t,x,v) = a(V(s,t,x,v)), \qquad &X(t,t,x,v) =x,\\
&\frac{d}{ds} V(s,t,x,v) = F(s,X(s,t,x,v) ,V(s,t,x,v)), \qquad &V(t,t,x,v) =v.
\end{aligned}
\right.
\end{equation}
The existence and uniqueness of such curves are a consequence of the Cauchy-Lipschitz theorem (assuming we deal with smooth forces).
The method of characteristics asserts that one can represent the solution of~\eqref{Vlasov} as
\begin{equation}
f(t,x,v) = f_0 (X(0,t,x,v),V(0,t,x,v)).
\end{equation}
Therefore we see (except maybe in trivial cases such as $F\equiv 0$) that derivatives in $x$ of $f(t)$ involve derivatives in $x$ and in $v$ of $f_0$, so that regularity in $x$ only of $f_0$ can not in general be propagated for $f(t)$.
However given some smooth test function $\psi(v)$ (the case $\psi =1$ is already interesting), we can also wonder about the higher regularity of the moment $m_\psi(t,x) := \int_{\R^d} f(t,x,v) \psi(v) \,dv$. Such moments, which can be interpreted as hydrodynamic quantities, are important objects in kinetic theory.
We have the representation formula
$$
m_\psi(t,x) = \int_{\R^d}  f_0 (X(0,t,x,v),V(0,t,x,v)) \psi(v) \,dv.
$$
We note that for $t$ small enough, the map $v \mapsto V(s,t,x,v)$ is a diffeomorphism for all $s \in [0,t]$. Indeed for $s=t$ this map is the identity and integrating with respect to $s$ the equation satisfied by $V(s,t,x,v)$, we note that for $t$ small enough and $s \in [0,t]$, the map $v \mapsto V(s,t,x,v)$ is a small perturbation of the identity, hence our claim that it is a diffeomorphism. In particular the map $v \mapsto V(0,t,x,v)$ is a diffeomorphism  and we denote by $V^{-1}(t,x,v)$ its inverse. Using this diffeomorphism as a change of variables (in $v$) we get, for $t$ small enough,
\begin{multline*}
m_\psi(t,x) =\\
 \int_{\R^d}  f_0 (X(0,t,x,V^{-1}(t,x,v)), v) \psi(V^{-1}(t,x,v)) |\det D_v V(0,t,x,v)|^{-1} \,dv.
\end{multline*}
Thanks to this formula, at least formally,  Leibniz rule allows to ensure that  derivatives in $x$ of the moment $m_\psi$ only involve  derivatives in $x$ of $f_0$. Recalling the extra higher regularity~\eqref{extra}, it seems maybe natural to expect that the moment $m_\psi$ belongs to the Sobolev space $H^{n+1}$ in $x$. In the case where $F$ is a \emph{fixed} external force, assumed to be very smooth, say $C^\infty$ with respect to all variables, then the fact that, $t$ being fixed,
$m_\psi(t,\cdot)$ belongs to $H^{n+1}_x$ follows indeed from Leibniz formula, using the fact the characteristic curves $(X,V)$ inherit the $C^\infty$ regularity of $F$.

However this argument seems to break down in the case where the case $F$ depends on the solution $f(t)$ itself, as the regularity of $F$ is then tightly linked to that of $f$.
Let us discuss for instance the Poisson case -- the Maxwell case is actually worse in the sense that in the Vlasov-Poisson coupling, $F$ gains, loosely speaking, one derivative in $x$ compared to $f$.
As already mentioned, the local Cauchy theory yields $f(t) \in {C}(0,T; \Hc^n_r)$, and we have $F \in {C}(0,T; H^{n+1}_x)$. 
Note then that when applying $(n+1)$ derivatives in $x$ on $m_\psi$, one needs to apply $(n+1)$ derivatives in $x$ on  $|\det D_v V(0,t,x,v)|^{-1}$, which amounts to applying in total $(n+2)$ derivatives to $V(0,t,x,v)$. However, by~\eqref{charac}, we observe that $(X,V)$ inherits the same order of regularity as $F$, and therefore it does not seem licit to take as many as $(n+2)$ derivatives. 

The goal of this work is to show that despite this apparent shortcoming, it is indeed possible to show for a fairly wide class of non-linear Vlasov equations (including the Vlasov-Poisson and Vlasov-Maxwell system) a result of propagation of  regularity in $x$  for the moments, assuming higher order space regularity for the initial condition.  This in turn can be translated into some anisotropic Sobolev higher regularity  for the solution itself, which can be interpreted as a kind of weak propagation of space regularity.

It turns out that the \emph{lagrangian approach}, that is to say the approach that we have just underlined, based on  representation formulas using characteristics , is not adapted to answer this question. 
Instead we shall rely on an \emph{eulerian approach}, that is  based to a larger extent on the PDE itself, inspired by the recent work of the author in collaboration with F. Rousset on the quasineutral limit of the Vlasov-Poisson system \cite{HKR,HKR2}.
The quasineutral limit is a singular limit which loosely consists in a penalization of the Laplacian in the Poisson equation. The small parameter is the scaled Debye length, which appears to be very small in several usual plasma settings.
The limit leads to \emph{singular} Vlasov equations, which display a loss of regularity of the force field compared to that of  the distribution function. As a consequence, these equations are in general ill-posed in the sense of Hadamard, see \cite{BarNou} and \cite{HKN2}. This problem might therefore look quite different from the one considered here; the similarity comes from the fact that the justification of the quasineutral limit ultimately loosely comes down to the proof of an \emph{uniform}\footnote{(with respect to the scaled Debye length)} propagation of one order of higher regularity for  moments of solutions of the Vlasov-Poisson equation. Note though that the analysis of \cite{HKR,HKR2} requires the introduction of pointwise Penrose stability conditions, and also relies on pseudo-differential tools, which will not be the case in this paper.
As a matter of fact, the singular Vlasov equations which can be formally derived in the quasineutral limit will not enter the class of Vlasov equations we will deal with in this work, precisely because of the aforementioned loss of derivative.

The methodology of  \cite{HKR} was also used in the context of large time estimates for data close to stable equilibria for the Vlasov-Maxwell system in the non-relativistic regime, in a  recent work in collaboration with T. Nguyen and F. Rousset \cite{HKNR}.

 As a matter of fact, the approach can be considered as \emph{semi-lagrangian}, in the sense that at some point we still rely on characteristics as in the lagrangian approach but at the level at the PDEs that arise after applying derivatives on the Vlasov equation, whereas in the lagrangian approach, derivatives are taken after using the representation of the solution by characteristics.

\section{Main results}

\subsection{The abstract framework} Let us now describe precisely the class of Vlasov equations we deal with. We consider in this work the abstract equation
\begin{equation}
\label{abstract}
\pa_t f + a(v) \cdot \na_x f + F \cdot \na_v f =0,
\end{equation}
with the following structural assumptions. Among all these assumptions, we highlight that the force depends on the distribution function itself, but only through some of its moments in velocity.

\noindent $\bullet$ \emph{\large Assumptions on the advection field.}  The map $a: \R^d \to \R^d$ is a  one to one $C^\infty$ function
such that 
\begin{equation}
\label{a}
|a(v) | \leq C(1+ |v|), \quad \forall v \in \R^d, 
\end{equation} 
\begin{equation}
\label{deriv-a}
\|\pa^\alpha_v a \|_{L^\infty} \leq C_\a, \quad \forall |\alpha| \neq 0, 
\end{equation} and its inverse $a^{-1}$ (defined on $a(\R^d)$) satisfies, for some $\lambda>0$,
\begin{equation}
\label{growth-a-1}
|\pa^\alpha_v a^{-1}(w) | \leq C_\alpha (1+ |a^{-1} (w)| )^{1+ \lambda |\alpha|}, \qquad \forall w \in a(\R^d), \, \forall \alpha
.
\end{equation}

\bigskip

\noindent $\bullet$ \emph{\large Assumptions on the force field.} The vector field $F$ is divergence-free in $v$ (i.e. satisfies $\na_v \cdot F =0$) and 
we have the following decomposition for some $\ell \in \N^*$:
\begin{equation}
F(t,x,v)= \sum_{j=1}^\ell A_j (v) F^j(t,x).
\end{equation}
We assume that for all $j \in \{1,\cdots, \ell\}$, $A_j$ is a $C^\infty$ scalar function satisfying
\begin{equation}
\label{deriv-A}
\| \pa^\alpha_v A_j \|_{L^\infty} \leq C_\alpha, \quad \forall \alpha.
\end{equation}
Furthermore, there exist $C^\infty$ functions $\psi_1(v), \cdots, \psi_r(v)$ with at most polynomial growth,
i.e. there is $r_0>0$ such that
\begin{equation}
\label{eq-growthpsi}
\|\psi_i(v)\| _{\mathcal{W}^{k,\infty}_{-r_0}}\leq C_{i,k}, \quad\forall k \in \N,
\end{equation}
such that, 
denoting
$$
m_{\psi_i} (t,x)= \int_{\R^d} f (t,x,v)\psi_i(v) \,dv,
$$
 for all $j=1, \cdots, \ell$, 
 the vector field $F^j$ is \emph{uniquely determined by these moments and the initial conditions}, through a map
 \begin{equation}
\Big((m_{\psi_i} )_{i=1,\cdots,r},  (F^j(0))_{j=1,\cdots,\ell}\Big) \mapsto F^j
\end{equation}
 and 
  for all large enough $n > 1+ d $, and all $t >0$, we have
\begin{equation}
\label{Fj}
\begin{aligned}
&\| F^j \|_{L^2(0,t; H^n_x)} \\
& \leq \Gamma^{(j)}_{n} \left(t,  \| m_{\psi_1} \|_{L^2(0,t;H^n_x)}, \cdots, \| m_{\psi_r} \|_{L^2(0,t;H^n_x)}, \sum_{j=1}^{\ell} \|F^{j} (0) \|_{H^n_x} \right),
 \end{aligned}
\end{equation}
and
\begin{equation}
\label{Fj-inft}
\begin{aligned}
&\| F^j \|_{L^\infty(0,t; H^n_x)} \\
& \leq \Gamma^{(j)'}_{n} \left(t,  \| m_{\psi_1} \|_{L^\infty(0,t;H^n_x)}, \cdots, \| m_{\psi_r} \|_{L^\infty(0,t;H^n_x)}, \sum_{j=1}^{\ell} \|F^{j} (0) \|_{H^n_x} \right),
 \end{aligned}
\end{equation}
where $\Gamma^{(j)}_{n},\Gamma^{(j)'}_{n}$ is a polynomial function that is non-increasing with respect to each of its arguments (the others being fixed non-negative numbers). 

Finally, the force field satisfies the following stability property. Let $f$ and $g$ be two solutions of~\eqref{abstract}, and denote by $F[f]$ and $F[g]$ their associated force field. Assume that the initial conditions $ (F^j(0))_{j=1,\cdots,\ell}$ are the same. Then, we have for all $j = 1, \cdots, \ell$,
\begin{equation}
\label{stab}
\begin{aligned}
&\| F^j[f] - F^j[g]  \|_{L^2(0,t; H^n_x)} \\ 
&\leq \Gamma^{(j)\sharp}_{n} \left(t, \left\|\int (f - g)\psi_i(v) \,dv \right\|_{L^2(0,t;H^n_x)}, \cdots,  \left\|\int (f - g)\psi_r(v) \,dv \right\|_{L^2(0,t;H^n_x)}\right),
 \end{aligned}
\end{equation}
where $\Gamma^{(j)\sharp}_{n} $ is a polynomial function that is non-increasing with respect to each of its arguments and such that $\Gamma^{(j)'}_{n} (0,\cdot)=0$.

We shall explain later why both Vlasov-Poisson and relativistic Vlasov-Maxwell enter the abstract framework.


\subsection{Statement of the main results}

The regularity and integrability indices that will be useful to handle such equations will depend on the dimension $d$, the maximal growth of the moments that intervene in the definition of $F$, that is $r_0$, and the parameter of growth of the inverse of $a$, that is $\lambda$; let us set
\begin{equation}
\label{def-NR}
N := \frac{3}{2} d +4, \qquad R :=  \max\left( \frac{d}{2} + 2 (1+\lambda)(1+d) +  r_0\right).
\end{equation}
We  use in the following statement the notation $\lfloor \cdot \rfloor$ for the floor function.

The main result proved in this paper is the following theorem.
\begin{theorem}
\label{thm}
Let $n\geq N$ and $r>R$.
Let $n'>n$ be an integer such that
$n> \lfloor \frac{n'}{2} \rfloor+1$. Assume that $f_0 \in \Hc^{n}_{r}$ and $F^j(0) \in H^{n'}_x$ for all $j \in \{1,\cdots, \ell\}$.
Assume furthermore that the initial data $f_0$ satisfy the following higher anisotropic regularity:
\begin{equation}
\label{eq-thm}
\begin{aligned}
\pa_x^{2(n-\lfloor \frac{n'}{2} \rfloor+k)} \pa^\alpha_x \pa^\beta_v f_0 \in \Hc^0_r, \quad\forall |\alpha| + |\beta| =n'-n-k, \quad \forall k \in \left\{1,\cdots, 2\left\lfloor \frac{n'}{2} \right\rfloor-n \right\}.
\end{aligned}
\end{equation}
Then there is $T>0$ such that the following holds. There exists a unique solution $(f(t),F(t))$ with initial data $(f_0, F(0))$ to~\eqref{abstract} such that $f(t) \in C(0,T; \Hc^{n}_r)$.

Moreover, for all test functions $\psi \in L^\infty(0,T; \mathcal{W}^{n',\infty}_{-r_0})$, we have
\begin{equation}
\label{conc-thm}
\left \|\int f \psi \,dv  \right\|_{L^2(0,T; H^{n'}_x)} \leq \Lambda_\psi (T,M),
\end{equation}
where $\Lambda_\psi$ is a polynomial function and 
$$
M  = \| f_0 \|_{\Hc^n_r} +\sum_{j=1}^{\ell} \|F^j(0)\| _{H^{n'}_x} + \sum_{k=1}^{2\left\lfloor \frac{n'}{2} \right\rfloor-n } \sum_{|\alpha| + |\beta| =n'-n-k} \|\pa_x^{2(n-\lfloor \frac{n'}{2} \rfloor+k)} \pa^\alpha_x \pa^\beta_v f_0 \|_{\Hc^0_r}.
$$
\end{theorem}

%
 
 Thanks to~\eqref{Fj}, we immediatly deduce from~\eqref{conc-thm} that the force field satisfies as well the higher regularity 
 $$F^j \in L^2(0,T; H^{n'}_x).$$
Another consequence concerns the flow $(X,V)=(X(t,0,x,v),V(t,0,x,v))$ as defined
in \eqref{charac}, for which we also obtain  a higher regularity property.
\begin{corollary}
\label{coro}
 For some $T'\leq T$, we have
$$  \pa^\gamma_{x,v} (X-x-tv,V-v) \in {L^\infty(0,T'; L^\infty_v L^2_{x})}, \qquad \forall |\gamma|\leq n'.$$
\end{corollary}

\begin{remark}
Some remarks about Theorem~\ref{thm} are in order.
\begin{itemize}

\item In the case where $n = 2m-1$ and $n'=n+1 =2m$, the assumption~\eqref{eq-thm} simply reads $\pa^{n+1}_x f_0 \in \Hc^0_r$ and we obtain the $L^2_t H^{n+1}_x$ smoothness of the moments: in other words this gives an answer to the question raised in the beginning of the introduction. Note though that the regularity result we prove is not pointwise in $t$.

\item Observe that it is required that the higher regularity index $n'$ is not too large compared to $n$ (i.e., $n> \lfloor \frac{n'}{2} \rfloor+1$); such a restriction is somehow reminiscent of a similar one appearing in the celebrated result of Bony concerning the propagation of Sobolev microlocal regularity at characteristic points for general non-linear PDEs, see \cite[Th\'eor\`eme 6.1]{Bony}. We remark however that the class of PDEs considered in this work does not enter the framework of \cite{Bony}, in particular because of the ``non-locality'' in velocity. We refer to Section~\ref{sec-rk} for some remarks and (counter-)examples in this direction.

\item As a matter of fact, our result can  be somehow interpreted as a kinetic (and non-local) analogue of Bony's aforementioned theorem.

\item If it is ensured that the  solution $(f(t), F(t))$ to \eqref{abstract}  is \emph{global} (e.g. for Vlasov-Poisson in dimension $d\leq 3$, see \cite{LP,Pfa,SchVP,BR,Hor}), we do not know if the higher propagation of regularity for the moments may or may not be global.

\item Let us mention that in a somewhat different direction, a vector field method was devised in \cite{Smu} (see also \cite{FJM}) in order to prove time decay of moments for Vlasov equations set in unbounded spaces.

\end{itemize}
\end{remark}

In the case where the force is one-derivative smoother than the distribution function $f$ itself (that is to say when estimates~\eqref{Fj} hold with $n-1$ instead of $n$ in the right-hand side), the statement of Theorem~\ref{thm} may be strengthened, insofar as one may ask only for derivatives in $x$ in the regularity assumption~\eqref{eq-thm}. We refer to such as case as the transport/elliptic case, which includes in particular Vlasov-Poisson, see Theorem~\ref{thm2} in Section~\ref{sec-ell}.

\bigskip

As already mentioned in the introduction, the higher regularity for moments as obtained in Theorem~\ref{thm} actually yields regularity for the solution itself (see \cite{Gerard} for a microlocal version of this fact, in the context of averaging lemmas) in anisotropic Sobolev spaces (as defined by H\"ormander in \cite[Chapter II, Section 2.5]{Horm}), that we first introduce.

\begin{definition}Let $m, n \in \R$. The anisotropic Sobolev space $H^{m,n}_{x,v}$ is defined as 
$$
H^{m,n}_{x,v}: = \Bigg\{ g \in \mathcal{S}'(\T^d \times \R^d), \, (1+|k|^2)^{m/2} (1+|\eta|^2)^{n/2} \widehat{g}(k,\eta) \in L^2( \Z^d \times \R^d) \Bigg\},
$$
where $\widehat{g}$ stands for the Fourier transform\footnote{where $\widehat{g}(k,\eta)= 1/(2\pi)^d \int_{\T^d \times \R^d}  g(x,v) e^{-i  x \cdot k} e^{-i v \cdot \eta} \, \, dx dv$, although the convention that is chosen for the writing of the Fourier transform does not matter here.} of $g$.
We also denote 
$$H^{m,-\infty}_{x,v} := \bigcup_{p \in \R} H^{m,p}_{x,v}.$$

\end{definition}

We have

\begin{corollary}
\label{cor}
Consider the same assumptions and notations as in Theorem~\ref{thm}. 
We have 
$$f(t,x,v) \in L^2(0,T; H^{n',-\infty}_{x,v}).$$
\end{corollary}

Corollary~\ref{cor} is a direct consequence of some estimates obtained in the proof of Theorem~\ref{thm}; we will provide a proof of this fact in Section~\ref{sec-end}. It is actually possible to give an estimate of a value of $p <0$ such that $f \in L^2(0,T; H^{n',p}_{x,v})$.

\subsection{Overview of the proof}

{
We discuss in this section the ingredients, inspired from~\cite{HKR}, leading to the higher propagation of regularity for the moments
(the local well-posedness theory is fairly standard, see Section~\ref{sec-local}). We shall discuss here the case $n = 2k-1$ and $n'=n+1 =2k$. To ease readability, we assume here that the dimension is $d=1$ (in higher dimension, the algebra is more involved but the basic principle is the same).


\bigskip

\noindent {\bf A. Taking derivatives.}
Since we intend to propagate regularity in space, 
the first step consists in understanding how to appropriately apply derivatives in $x$ to the Vlasov equation~\eqref{abstract}.

We note that applying the operator $\pa^\alpha_x$ does not seem relevant, as it does not commute well with the operator $F  \pa_v$: as a result it is not possible to obtain a closed equation bearing on $\pa^\alpha_x f$ without appealing to $\pa^\beta_x \pa^\gamma_v f$ for $\gamma \neq 0$, and therefore such an approach would require a control of derivatives in $v$ which we do not have at initial time (this is of course reminiscent of the mixing in $x$ and $v$ that we have evoked in the introduction).

The idea is to look for more appropriate differential operators, with non-constant coefficients, satisfying the following three key properties:
\begin{itemize}
\item at initial time, they involve only derivatives in $x$;
\item they enjoy good commutation properties with the transport operator, so that it is eventually possible to obtain closed systems involving these differential operators alone;
\item they allow a good control of the Sobolev norm of the moments.
\end{itemize}
It turns out that second order differential operators in $x$ and $v$, with coefficients depending on the solution itself will be appropriate. 
More precisely, we consider  the operator
$$
    L := \pa^2_{x} + \varphi (t,x)\pa_{x} \pa_{v} + \psi(t,x)  \pa^2_{v},
$$
whose coefficients $\varphi $ and $\psi$ will depend on the force field $F$.
Setting $\Tc := \pa_t + a(v) \pa_x + F \pa_v$ the transport operator, we ask that the coefficients $\varphi,\psi$ solve a semi-linear system of the form
$$
\left\{
\begin{aligned}
&\Tc  \phi = 2 \pa_x F  +  G_1 \left(\phi,  \psi,  \pa_{x,v} F\right)
\\  
&\Tc \psi = G_2 \left(\phi,  \psi, \pa_{x,v} F\right), \\
&\varphi|_{t=0} =0, \, \psi|_{t=0} =0.
\end{aligned}
\right.
$$
where $G_1,G_2$ are polynomial functions of degree greater or equal to $2$; this corresponds to zero-order coupling terms.
Note  in  particular that by definition, $  L = \pa^2_{x}$ at time $t=0$. 
The semi-linear system is precisely chosen in order to cancel bad terms in the commutation between  $L$ and $\Tc$, so that for any function $g$, 
$$
L \Tc(g) = \Tc L (g) +  (L F)  \pa_v g + (La) \pa_x g+   (\pa_v a)   \varphi L g.
$$
Applying this identity to the solution $f$ of the Vlasov equation~\eqref{Vlasov}, this yields
$$
 \Tc L (f) =-  (L F)  \pa_v f  -(La) \pa_x f - (\pa_v a)   \varphi L f.
$$
   This  formula will play a key role in the analysis. The main term (in terms of regularity issues) is the term $-\pa_{x}^2 F   \pa_v f$, the others  involving either more regular quantities (we recall indeed that $F$ and $a$ are assumed to be smooth with respect to $v$), or the quantity $L f$, which paves the way for a closed system involving only compositions of $L$ applied to $f$. 
As a consequence, the operators obtained as compositions of  $L$
appear to be relevant for applying  higher order derivatives in $x$, since by construction 
\begin{itemize}
\item they require only a control of  space regularity at initial time;
\item denoting by ${L_k}$ the composition of $k$ operators $L$, one can obtain that $L_k f$ satisfies an equation of the form
\begin{equation}
\label{ab}
\Tc (L_k f) =  A( L_k f) -  (\pa^{2k}_x  F) \pa_v f +   G( (\pa^\alpha_{x,v} f)_{|\alpha| \leq 2k-1}),
\end{equation}
where $A,G$ are bounded linear operators. We note that this equation involves derivatives in $v$ of the solution, but only of order $2k-1=n$, which we control thanks to the local well-posedness theory.  This can therefore be seen as a closed equation for $L_k f$.
\item One can  show that for any smooth test function $\psi$,
$$
\int_{\R}  (L_k  f) \, \psi(t,x,v) \,dv =  \int_{\R} (\pa^{2k}_x f) \, \psi(t,x,v) \,dv + \text{``controlled terms''}.
$$
In the controlled terms, the overload of derivatives in $v$ falling on $f$ is transferred to $\psi$ by an integration by parts argument.
\end{itemize}

All in all, this eventually shows that the $L_k$ are indeed well-suited to study the regularity of moments.
This step is fully developed in Section~\ref{sec-DO}. There are two separate difficulties in order to complete this task: obtaining the right algebra as discussed here, and proving  Sobolev estimates for all the involved objects.

(In the case where $n'>n+1$, we need to set up  an induction argument, and this leads the study of  successive systems of coupled kinetic transport equations, which build on the general equation~\eqref{ab}.)

\bigskip


\noindent {\bf B. Propagation of regularity on moments.}
We then turn to the study of moments of the solutions to~\eqref{ab}.  
This step is partly inspired from (and thus related to) the treatment of linear Landau damping by Mouhot and Villani \cite{MV}.

We first use the method of characteristics to invert the operator $\Tc - A$.
It is convenient at this stage to use
changes of variables in velocity (introduced and studied in Section~\ref{sec-Burg}),
in order to straighten characteristics and eventually, roughly speaking, come down from $\Tc$ to the free transport operator $\pa_t + a(v) \cdot \na_x$. To this end, it turns out to be efficient to introduce the change of variables $v \mapsto \Phi$ where $\Phi$ solves the Burgers equation
$$
\pa_t \Phi + a(\Phi) \cdot \na_x \Phi = F(\cdot,\Phi), \qquad \Phi|_{t=0}=v,
$$
where we can prove  that $\Phi$ remains close to $v$ in small time (in terms of Sobolev norms). 
The problem comes down to the understanding of the contribution of the term $ -  (\pa^{2k}_x  F) \pa_v f $, and eventually roughly reduces to the study of an equation of the type
$$
\begin{aligned}
H_1(t,x) = &\int_{0}^t \int_{\R} (\pa_x H_2) (s,  x - (t-s) a(v))   
      U(t,s,x,v)\, dv ds  \\
      &+ \text{``controlled terms''},
      \end{aligned}
$$
where we know only that $H_2$ is controlled in $L^2(0,T; L^2_x)$ and $U$ is smooth, and we seek for a bound of $H_1$  in $L^2(0,T; L^2_x)$ (such an estimate corresponds to a control on the moments of $L_k f$). 
The integral in time is due to the use of Duhamel's formula, and the integral in $v$ to the fact that we study moments in $v$.
We observe that the operator in the right-hand side seems to feature a loss of derivative in $x$.
However, we use a smoothing effect to overcome this apparent loss, which was proved in \cite{HKR}. The outcome is the estimate
\begin{multline*}
\left\| \int_{0}^t \int_{\R^d} (\na_x H_2) (s,  x - (t-s) a(v))  \cdot 
      U(t,s,x,v)\, dv ds  \right\|_{L^2(0,t;L^2_x)} \\
       \lesssim \| H_2\|_{L^2(0,t; L^2_x)} \sup_{0\leq t,s \leq T} \| U(t,s, \cdot) \|,
\end{multline*}
where $\|\cdot \|$ stands for a high order weighted Sobolev norm (in $x$ and $v$) which we will make precise later.
As noted in~\cite{HKR}, this is reminiscent of (but different from) classical \emph{kinetic averaging lemmas}, as it loosely speaking involves the gain of one full derivative; we refer to Section~\ref{sec-K} for a thorough discussion. 


}
\subsection{Content of the end of the paper}

The paper is then organised as follows: the proofs of Corollaries~\ref{coro} and~\ref{cor} are provided in the end of Section~\ref{sec-end}. In Section~\ref{sec-Vla}, we check the general assumptions for the Vlasov-Poisson and relativistic Vlasov-Maxwell equations, and discuss some extensions as well.
As already mentioned, Section~\ref{sec-ell} is devoted to the particular case of the transport/elliptic case, for which Theorem~\ref{thm} can be improved.
We end the paper  with the study of two examples that we cook-up in order to discuss  the  regularity assumptions of Theorem~\ref{thm}.

\bigskip

We will prove Theorem~\ref{thm}  when $n$ is odd, of the form $n = 2m-1$, 
and the higher regularity index $n'$ is even,  of the form $n' =2(m+p)$. The other cases follow by the same arguments.
The requirement on $n$ and $n'$ is $m> p +2$.
 The assumption~\eqref{eq-thm} reads in this case
\begin{equation}
\label{eq-thm-proof}
\begin{aligned}
\pa_x^{2(m-p+k)} \pa^\alpha_x \pa^\beta_v f_0 \in \Hc^0_r, \qquad \forall |\alpha| + |\beta| =2p-k,  \qquad \forall k=0,\cdots, 2p.
\end{aligned}
\end{equation}


\section{Local well-posedness}
\label{sec-local}
We prove in this section a basic local Sobolev well-posedness result for~\eqref{abstract}.
We start by recalling useful product estimates in weighted Sobolev spaces, taken from \cite{HKR}.
\begin{lemma}
\label{lem-calculus}
  Let $s$ be a non-negative integer. Consider a smooth nonnegative  function $\chi= \chi(v)$ that satisfies $|\partial^\alpha \chi| \leq C_{\alpha} \chi$ for every multi-index $\alpha$ such that $|\alpha| \leq s$.
     \begin{itemize}
     
            \item Consider two functions $f=f(x,v)$, $g= g(x,v)$, then we have for $k \geq s/2$
     \begin{equation}
     \label{com1}
     \| \chi f g \|_{H^s_{x,v}} \lesssim \| f\|_{W^{k, \infty}_{x,v}} \|\chi g\|_{H^s_{x,v}} +   \| g\|_{W^{k, \infty}_{x,v}} \|\chi f\|_{H^s_{x,v}}.
        \end{equation}
     
     \item Consider a  function $E= E(x)$ and a  function $F(x,v)$,  then we have  for any $s_{0}>  {d }$  that
      \begin{equation}
      \label{com3}
       \left \| \chi  E F  \right\|_{H^s_{x,v}} \lesssim  \|E\|_{H^{s_{0} }_{x} } \| \chi F\|_{H^{s}_{x,v}} +  \|E\|_{H^s_{x}} \|\chi F \|_{H^{s}_{x,v}}.
      \end{equation}

       \item   Consider a vector field $E= E(x)$, a function $A(v)$, and a  function $f=f(x,v)$,  then we have  for any $s_{0}> 1+ {d }$ and for any multi-indices $\alpha$, $ \beta$ such that  $|\alpha |+ |\beta |= s \geq 1$ that
      \begin{equation}
      \label{com2}
       \left \| \chi \left[ \partial_{x}^\alpha \partial^\beta_{v}, A(v) E(x)\cdot \nabla_{v}\right]  f\right\|_{L^2_{x,v}} \lesssim  \| A\|_{W^{s,\infty}_v}(  \|E\|_{H^{s_{0} }_{x} } \| \chi f\|_{H^{s}_{x,v}} +  \|E\|_{H^s_x} \|\chi f \|_{H^{s}_{x,v}}).
      \end{equation}
\item  Consider two  functions $f=f(x,v)$, $g= g(x,v)$, then we have for multi-indices $\alpha, \beta $
       with $| \alpha | + | \beta | \leq s$ that 
    \begin{equation}
    \label{comdual}
     \| \partial_{x,v}^\alpha f\, \partial_{x,v}^{\beta} g\|_{L^2} \lesssim  \| { 1 \over \chi} f\|_{L^\infty_{x,v}} \| \chi g \|_{H^s_{x,v}} +  \| \chi g \|_{L^\infty_{x,v}} \| {1 \over \chi} f \|_{H^s_{x,v}}.
    \end{equation}

 \end{itemize}

\end{lemma}

\begin{proposition}
\label{lem-local}
Let $n>d+1$ and $r>r_0+d/2$.
Assume that $f_0 \in \Hc^{n}_r$ and $F^j(0) \in H^{n}_x$.
Then there exists $T>0$ such that there is a unique solution $(f(t),F(t))$ with initial data $(f_0, F(0))$ to~\eqref{abstract} such that $f(t) \in C(0,T; \Hc^{n}_r)$ and $F^j (t) \in L^\infty(0,T; H^{n}_x)$.
\end{proposition}

\begin{proof}[Proof of Proposition~\ref{lem-local}]
The existence part follows from a standard iterative construction. We define recursively a sequence of distribution functions $(f_{(k)})_{k \in \N}$, denoting by $F_{(k)}$ the force field associated to $f_{(k)}$ and the initial condition $(F^j(0))$. 
Let us denote 
$$
R_0 := \| f_0 \|_{\Hc^{n}_r} +  \sum_{j=1}^{\ell} \|F^{j} (0) \|_{H^n_x}
$$
We set $f_{(0)} := f_0$ and assume that $f_{(k)}$ is already constructed (with associated force field $F_{(k)}$), and is such that for some $T_k>0$, $f_{(k)} \in C(0,T_k; \Hc^n_r)$, and 
\begin{equation}
\|f_{(k)}\|_{L^\infty(0,T_k; \Hc^n_r)} \leq 2R_0.
\end{equation}
We define $f_{(k+1)}$ as the unique solution on $[0, T_k)$  to the equation
\begin{equation}
\label{iteration}
\pa_t f_{(k+1)} + a(v) \cdot \na_x f_{(k+1)} + F_{(k)} \cdot \na_v f_{(k+1)} =0,  \quad f_{(k+1)}|_{t=0} = f_0,
\end{equation}
obtained by  the method of characteristics.

Applying the operator $\pa^\a_x \pa^\b_v$ to~\eqref{iteration} for $|\a| +|\b| \leq n$ yields
$$
(\pa_t  + a(v) \cdot \na_x  + F_{(k)} \cdot \na_v) (\pa^\a_x \pa^\b_v f_{(k+1)}) + \left [ \partial_{x}^\alpha \partial_{v}^\beta, a(v) \cdot \na_x +  F_{(k)} \cdot \nabla_{v} \right] f_{(k+1)}   =0.
$$
 We then take  the $L^2$ scalar product with $(1+ |v|^2)^{r} \pa^\a_x \pa^\b_v f_{(k+1)}$ and sum for all $|\a| + |\b| \leq n$.  By using~\eqref{deriv-a}, we have
 $$
\sum_{|\a| + |\b| \leq n} \int |\left [ \partial_{x}^\alpha \partial_{v}^\beta, a(v) \cdot \na_x  \right] f_{(k+1)} \pa^\a_x \pa^\b_v f_{(k+1)}|  (1+ |v|^2)^{r} \,dv dx 
\leq \| f_{(k+1)} \|_{\Hc^{n}_r}^2.
 $$
Thanks to \eqref{deriv-A} and estimate \eqref{com2}  in Lemma~\ref{lem-calculus} with  $s = n$, $\chi (v)= (1+ |v|^2)^{r/2}$ and $s_{0}= n \,$  (recall that $n>d+1)$, we have for all $j\in \{1,\cdots,\ell\}$,
$$
       \left \| \chi \left[ \partial_{x}^\alpha \partial^{\beta}_{v}, A_j(v) F^j_{(k)}(x)\cdot \nabla_{v}\right]  f_{(k+1)}\right\|_{L^2_{x,v}} \lesssim  \| F^j_{(k)}\|_{H^{n}_x} \|f_{(k+1)} \|_{\Hc^{n}_{r} }.
      $$
      Therefore by Cauchy-Schwarz, we get 
       $$
       \begin{aligned}
\sum_{|\a| + |\b| \leq n} \int  |\left [ \partial_{x}^\alpha \partial_{v}^\beta, F_{(k)} \cdot \na_v  \right] f_{(k+1)} \pa^\a_x \pa^\b_v f_{(k+1)} | (1+ |v|^2)^{r} \,dv dx  \\
\lesssim \| F^j_{(k)}\|_{H^{n}_x}  \| f_{(k+1)} \|_{\Hc^{n}_r}^2.
\end{aligned}
 $$
Recalling that $\na_v \cdot F = 0$, we deduce  that for all $t \in (0,T_k)$, 
$$
\frac{d}{dt} \| f_{(k+1)}(t) \|_{\Hc^{n}_r}\lesssim \left(1+ \sum_{j=1}^\ell \| F^j_{(k)}\|_{H^n_x}\right)\| f_{(k+1)}(t) \|_{\Hc^{n}_r} \\
$$
so that
\begin{equation}
 \| f_{(k+1)}(t) \|_{\Hc^{n}_r} \lesssim  \| f_0 \|_{\Hc^{n}_r} \exp \left[ C\int_0^t\left(1+ \sum_{j=1}^\ell \| F^j_{(k)}(s)\|_{H^n_x}\right)  ds  \right].
\end{equation}
We set 
$$
m_{\psi_i, (k)} (t,x)= \int_{\R^d} f_{(k)} (t,x,v)\psi_i(v) \,dv,
$$
and get, for $r'> d/2$ such that $r\geq r_0 +r'$ (which is possible thanks to the assumption $r> r_0 + d/2$), by Cauchy-Schwarz and~\eqref{eq-growthpsi}, that 
\begin{align*}
&\| m_{\psi_i, (k)} \|_{L^2(0,t;H^n)}  = \sum_{|\a| \leq n} \left\| \left( \int_{\R^d} \pa^\a_x f_{(k)} {\psi_i} \,dv  \right)^2 \right \|_{L^2(0,t;L^1_x)}^{1/2}\\
&\qquad\lesssim \sum_{|\a| \leq n} \left\| \left( \int_{\R^d} |\pa^\a_x f_{(k)}|^2(1+ |v|^2)^{r_0 + r'} \,dv  \right)  \left(\int_{\R^d} \frac{ |\psi_i|^2 \, dv}{(1+ |v|^2)^{r_0 +r'}}  \right) \right \|_{L^2(0,t;L^1_x)}^{1/2} \\
&\qquad\lesssim \sum_{|\a| \leq n} \left\| \left( \int_{\R^d} |\pa^\a_x f_{(k)}|^2(1+ |v|^2)^{r_0 + r'} \,dv  \right)  \left(\int_{\R^d} \frac{dv}{(1+ |v|^2)^{r'}}  \right) \right \|_{L^2(0,t;L^1_x)}^{1/2} \\
&\qquad\lesssim \| f_{(k)}\|_{L^2(0,t;\Hc^{n}_r)}.
\end{align*}
Therefore, by~\eqref{Fj}, denoting by $C>0$ a generic constant that does not depend on $t$ or $k$, we obtain
\begin{align*}
& \| f_{(k+1)}(t) \|_{\Hc^{n}_r} \\
  &\lesssim  \| f_0 \|_{\Hc^{n}_r} \exp \left[C t + C \sqrt{t} \sum_{j=1}^\ell  \| F^j_{(k)} \|_{L^2(0,t; H^n_x)} \ \right] \\
 &\lesssim  \| f_0 \|_{\Hc^{n}_r} \exp \Bigg[C t  \\
 &\qquad\qquad +C\sqrt{t} \sum_{j=1}^\ell \Gamma^{(j)}_{n} \left(t, ( \sqrt{t} \| m_{\psi_i, (k)} \|_{L^\infty(0,t;H^n_x)})_{i=1,\cdots,r}, \sum_{j=1}^{\ell} \|F^{j} (0) \|_{ H^n_x} \right)\Bigg]  \\ 
&\lesssim  \| f_0 \|_{\Hc^{n}_r} \exp \left[ C t + C\sqrt{t} \sum_{j=1}^\ell \Gamma^{(j)}_{n} \left(t, \sqrt{t}   \| f_{(k)}\|_{L^\infty (0,t;\Hc^{n}_r)}
, \sum_{j=1}^{\ell} \|F^{j} (0) \|_{H^n_x} \right) \right].
\end{align*}
We now observe that if we choose $T>0$ small enough so that
\begin{equation}
R_0 \exp \left[ CT +C\sqrt{T}  \sum_{j=1}^\ell \Gamma^{(j)}_{n}(T, 2 \sqrt{T} R_0,  R_0) \ \right] < 2 R_0,
\end{equation}
and  $T_k\geq T$,
then, 
\begin{equation}
\| f_{(k+1)}(t) \|_{L^\infty(0,T; \Hc^{n}_r)} \leq 2 R_0. 
\end{equation}
Therefore, by induction, we  obtain that for all $k \in \N$, $f_{(k)} \in C(0,T; \Hc^n_r)$, and 
\begin{equation}
\label{unifk}
\|f_{(k)}\|_{L^\infty(0,T; \Hc^n_r)} \leq 2R_0.
\end{equation}
%
For $k \in \N\setminus\{0\}$, we set $h_k := f_{(k+1)}-f_{(k)}$, which satisfies the equation
\begin{equation}
\pa_t h_k+ a(v) \cdot \na_x h_k + F[f_{k}] \cdot \na_v h_k  + (F[f_{(k)}]-F[f_{(k-1)}])\cdot \na_v f_{k}=0.
\end{equation}
By weighted $L^2$ estimates, proceeding as before, we get
\begin{multline*}
\frac{d}{dt} \| h_k(t) \|^2_{\Hc^{0}_r}\lesssim \left(1+ \sum_{j=1}^\ell \| F^j[f_{(k)}] \|_{H^n_x}\right)\| h_k(t) \|_{\Hc^{0}_r}^2  \\
+ \| f_{(k)}\|_{\Hc^n_r} \sum_{j=1}^\ell \| F^j[f_{(k)}]-F^j[f_{(k-1)}] \|_{L^2_x} \| h_k(t) \|_{\Hc^{0}_r}.
\end{multline*}
Let $t \in (0,T)$. Integrating in time, applying Cauchy-Schwarz and using the stability property~\eqref{stab} and the uniform estimates~\eqref{unifk} for $(f_{(k)})$, we obtain
$$
\begin{aligned}
&\| h_k\|_{L^\infty(0,t;\Hc^{0}_r)}  \\
&\lesssim \int_0^t  \left(1+ \sum_{j=1}^\ell \| F^j[f_{(k)}] \|_{H^n_x}\right)\| h_k(s) \|_{\Hc^{0}_r} \, ds + \sum_{j=1}^\ell \int_0^t \| F^j[f_{(k)}]-F^j[f_{(k-1)}] \|_{L^2_x}\, ds\\
&\lesssim \sqrt{t} \left[(\sqrt{t}+ \| F^j[f_{(k)}] \|_{L^2(0,t;H^n_x)})\| h_k \|_{L^\infty (0,t;\Hc^{0}_r)}+ \sum_{j=1}^\ell \| F^j[f_{(k)}]-F^j[f_{(k-1)}] \|_{L^2(0,t; L^2_x)}\right]\\ 
&\lesssim \sqrt{t}\Bigg[\| h_k \|_{L^\infty(0,t;\Hc^{0}_r)} +  \sqrt{t} \sum_{i=1}^r \Gamma^{(j)\sharp}_{n} \left(t, \left( \sqrt{t}\left\|\int (f_{(k)} - f_{(k-1)})\psi_i(v) \,dv \right\|_{L^\infty(0,t;L^2_x)}\right)_{i=1,\cdots,r} \right)\Bigg] \\
&\lesssim \sqrt{t}\left[\| h_k \|_{L^\infty(0,t;\Hc^{0}_r)} + \sqrt{t} \sum_{i=1}^r \Gamma^{(j)\sharp}_{n} \left(t, \sqrt{t}\| h_{k-1}\|_{L^\infty(0,t; \Hc^0_r)}\right)\right] . 
\end{aligned}
$$
We can thus pick a small enough time $T'>0$, independently of $k$ such that for all $k \in \N\setminus\{0\}$,
$$
\| f_{(k+1)}- f_{(k)}\|_{L^\infty(0,T';\Hc^{0}_r)}  \leq \frac{1}{2} \| f_{(k)}- f_{(k-1)}\|_{L^\infty(0,T';\Hc^{0}_r)}.
$$
We can therefore pass to the limit in~\eqref{iteration} and
 find that the limit $(f, F[f])$ satisfies (in the sense of distributions) 
\begin{equation}
\label{eqeq}
\pa_t f + a(v) \cdot \na_x f + F[f] \cdot \na_v f =0,
\end{equation}
with the initial conditions $(f_0, F^j(0))$. 
We deduce from~\eqref{eqeq} that $f \in C^0(0,T'; \Hc^n_r)$ and $\pa_t f \in L^2(0,T'; \Hc^{n-1}_{r-1})$. Also, thanks to~\eqref{Fj-inft}, we deduce  $F^j \in L^\infty(0,T'; H^{n}_x)$.
That the equation is satisfied in a classical way follows from the smoothness of $(f,F[f)])$. Uniqueness is also a consequence of the contraction estimate.

%

\end{proof}

The main matter is now to obtain the higher regularity statement for the moments.
To this end, we will focus only on the task of obtaining \emph{a priori} estimates for  \emph{smooth} solutions of~\eqref{abstract}; setting 
\begin{equation}\label{defM}
M := \| f_0 \|_{\Hc^{2m-1}_r} +\sum_{k=0}^{2p} \sum_{|\alpha| + |\beta| =2p-k} \|\pa_x^{2(m-p+k)} \pa^\alpha_x \pa^\beta_v f_0 \|_{\Hc^0_r}+\sum_{j=1}^\ell \|F^j(0)\| _{H^{2(m+p)}_x},
\end{equation}
 we look for some  time  $T_0>0$ depending only on $M$ such that given a smooth  test function $\psi \in L^\infty(0,T_0; \mathcal{W}^{2(m+p),\infty}_{-r_0})$,
the following estimate holds:
\begin{equation}
\label{apriori}
\left\| \int f \psi(v) \,dv  \right\|_{L^2(0,T_0; H^{2(m+p)}_x)} \leq C_\psi \Lambda(T_0,M).
\end{equation}
where $\Lambda$ is a polynomial function which is non-decreasing with respect to each of its arguments, once the others are fixed non-negative numbers. In what follows, the function $\Lambda$  may change from line to line but will always refer to such a function.

Once  \emph{a priori} estimates such as \eqref{apriori} as are obtained, we apply them to the sequence of solutions built in the iteration scheme proving the existence of solutions in the proof of Lemma~\ref{lem-local}. Passing to the limit yields the higher regularity for the moments of the solution $f(t)$.

\section{Differential operators}
\label{sec-DO}

{
In this section, we introduce and study the second order differential operators (with coefficients depending on $t$ and $x$) that we use in order to apply derivatives in $x$ on the Vlasov equation~\eqref{abstract}.

 The basic operators are defined in~\eqref{def-L} and the definition of the coefficients is provided in Lemma~\ref{lem-Lij}. 
 By definition these operators involve only derivatives in $x$ at initial time. The key algebraic result reflecting the good commutation properties of these operators with the transport operator is stated in Lemma~\ref{Lij}.
 
 The composition of these operators is then studied. 
\begin{itemize}
\item In Lemma~\ref{lem-momentsLij}, it is shown that they are indeed well-suited to study the regularity of moments, as after integration in $v$, they act like derivations in $x$ only (plus remainders that we can control).
The proof is quite technical as one needs to be careful of the limited available smoothness on the coefficients of the differential operators. Note that in  the statement, one does assume some (limited) higher order smoothness for the moments: this is in prevision of a forthcoming induction argument.

\item In Lemmas~\ref{even} and~\ref{odd}, the equations satisfied by the functions obtained after composition of these operators is established. This is where the key algebraic Lemma~\ref{Lij} appears to be crucial. Whereas the formal computation is  straightforward, here again, the proof appears to be quite technical in order to justify that remainders are indeed well controlled.
One also needs to be careful in order to get some Sobolev regularity for the coefficients involved in the equations.

\item As the systems of equations in Lemmas~\ref{even} and~\ref{odd} are not closed, this invites to study the system satisfied by a larger set of appropriate functions; this is the purpose of Lemmas~\ref{even'} and~\ref{odd'} (whose proof is similar to that of 
Lemmas~\ref{even} and~\ref{odd}).
\end{itemize}
}

\subsection{Second order operators}

As in the introduction, we set to ease readability
$\Tc := \pa_t + a(v) \cdot \na_x + F \cdot \na_v$ the transport operator.

\begin{lemma}
\label{lem-Lij}
Let $n>d+1$. Assume that $(F^j) \in  L^2(0,T';  H^n_x)$ with norm bounded by $\Lambda(T',M)$.
There is $T\in (0,T')$ such that there exists a unique smooth solution  $(\varphi_{k,l}^{i,j}, \psi_{k,l}^{i,j})_{i,j,k,l \in \{1,\cdots,d\}}$ on $[0,T]$ 
of the system:
      \begin{equation}
      \label{eq-constraint}
      \left\{
\begin{aligned}
&\Tc  \varphi_{k,l}^{i,j} = \sum_{k'} \pa_{v_{k'}} a(v)_k \psi_{k',l}^{i,j} +\sum_{k'} \pa_{v_{k'}} a(v)_k   \psi_{l,k'}^{i,j}  -\sum_{k',l',m} \pa_{v_{l'}}a(v)_m \varphi_{k',l'}^{i,j}   \varphi_{k,l}^{k',m}  \\
&\qquad+ \delta_{k,j} \pa_{x_i}F_l + \delta_{k,i} \pa_{x_j}F_l  + \sum_{l'} \varphi_{k,l'}^{i,j}  \pa_{v_{l'}}  F_l,  \\
&\Tc \psi_{k,l}^{i,j} =-   \sum_{k',l',m} \pa_{v_{l'}}a(v)_m \varphi_{k',l'}^{i,j}  \psi_{k,l}^{k',m}  +  \varphi_{k,l}^{i,j} \pa_{x_k} F_k +   \sum_{k'}  \psi_{k',l}^{i,j}   \pa_{v_{k'}} F_k \\
&\qquad  + \sum_{l'}  \psi_{k,l'}^{i,j}   \pa_{v_{l'}} F_l, \\
& \varphi_{k,l}^{i,j}\vert_{t=0} =  \psi_{k,l}^{i,j}\vert_{t=0} =0,
\end{aligned}
\right.
    \end{equation}
    where $\delta$ denotes the Kronecker function and $a(v)_k$ (resp. $F_k$) stands for the $k$-th coordinate of $a(v)$ (resp. $F$). Moreover we have the following estimates:
    \begin{equation}
    \label{est-Lij}
    \begin{aligned}
   &\sup_{[0,T]} \sup_{i,j,k,l} \| (\varphi_{k,l}^{i,j}, \psi_{k,l}^{i,j})\|_{W^{p,\infty}_{x,v}} \lesssim \Lambda(T,M), \qquad \forall p< n-1 - d/2, \\
      &\sup_{[0,T]} \sup_{i,j,k,l} \| (\varphi_{k,l}^{i,j}, \psi_{k,l}^{i,j})\|_{\Hc^{n-1}_{-\r}}\lesssim \Lambda(T,M), \qquad\forall \r > d/2.
      \end{aligned}
    \end{equation}
\end{lemma}
We will not reproduce the proof of Lemma~\ref{lem-Lij}, since it follows, \emph{mutatis mutandis}, that of Lemma 4.2 of \cite{HKR}: System~\eqref{eq-constraint} is solved as a semi-linear system of coupled kinetic transport equations. Note that we use the assumptions~\eqref{deriv-a} on $a$ and~\eqref{deriv-A} on $A$ to control the contribution of the additional linear and semi-linear terms that appear compared to Lemma 4.2 of \cite{HKR}.


  Introduce now the second order operators
    \begin{equation}
    \label{def-L}
    L_{i, j} := \pa^2_{x_i, x_j} + \sum_{1\leq  k,l \leq d} \left( \varphi_{k,l}^{i,j}\pa_{x_k} \pa_{v_l} + \psi_{k,l}^{i,j} \pa^2_{v_k, v_l} \right), \quad \forall i,j \in \{1,\cdots d\}.
    \end{equation}
We observe that by uniqueness of the solution of~\eqref{eq-constraint} and a symmetry argument, $L_{i,j} =L_{j,i}$.

One of the interests of the operators $L_{i,j}$ comes from the following lemma.

  \begin{lemma}
    \label{Lij}
     For all smooth functions $f$, we have the formula
        \begin{equation}
        \label{LT}
        \begin{aligned}
L_{i,j} \Tc(f) &= \Tc L_{i,j} (f) +  \left(\pa_{x_i,x_j}^2 F +  \sum_{  k,l } \varphi_{k,l}^{i,j}  \pa_{x_k} \pa_{v_l} F +  \psi_{k,l}^{i,j} \pa_{v_k,v_l}^2 F\right)\cdot \na_v f \\
&+ \sum_{  k,l } \psi_{k,l}^{i,j}  \pa_{v_k, v_l}^2 a(v)  \cdot \na_x f 
+ \sum_{k,l,m}\pa_{v_l}a(v)_m   \varphi^{i,j}_{k,l}  L_{k,m} f.
\end{aligned}
    \end{equation}
    
    \end{lemma}
\begin{rem}
Formula~\eqref{LT} can also be written in a more synthetic form:
$$
L_{i,j} \Tc(f) = \Tc L_{i,j} (f) +  (L_{i,j} F) \cdot \na_v f  +(L_{i,j} a)  \cdot \na_x f 
+ \sum_{k,l,m}\pa_{v_l}a(v)_m   \varphi^{i,j}_{k,l}  L_{k,m} f.
$$
\end{rem}

 \begin{proof}[Proof of Lemma~\ref{Lij}]
    We have by direct computations
    \begin{align*}
    \pa_{x_i x_j}^2 (\Tc f) &= \Tc (  \pa_{x_i, x_j}^2 f) + \pa_{x_i, x_j}^2 F \cdot \na_v f 
    + \pa_{x_i} F \cdot \na_v \pa_{x_j} f 
    + \pa_{x_j} F \cdot \na_v \pa_{x_i} f, \\
    \varphi_{k,l}^{i,j}\pa_{x_k} \pa_{v_l} (\Tc f) &= \Tc ( \varphi_{k,l}^{i,j}\pa_{x_k} \pa_{v_l} f)      -\Tc  (\varphi_{k,l}^{i,j}) \pa_{x_k} \pa_{v_l} f \\
     &+ \varphi_{k,l}^{i,j} \Big(\pa_{v_l} a(v) \cdot \na_x \pa_{x_k }f \\
     &\qquad \qquad \qquad \qquad
     + \pa_{x_k} F  \cdot \na_v \pa_{v_l} f 
     + \pa_{v_l} F \cdot \na_v \pa_{x_k} f
     + \pa_{x_k} \pa_{v_l} F \cdot \na_v f\Big), \\
      \psi_{k,l}^{i,j} \pa_{v_k, v_l}^2 (\Tc f) &= \Tc ( \psi_{k,l}^{i,j} \pa_{v_k, v_l}^2 f)       -  \Tc (\psi_{k,l}^{i,j}) \pa_{v_k, v_l}^2 f \\
      &+  \psi_{k,l}^{i,j}\Big( \pa_{v_l} a(v) \cdot \na_x \pa_{v_k} f 
      + \pa_{v_k} a(v) \cdot \na_x  \pa_{v_l}  f + \pa_{v_k,v_l}^2 a(v) \cdot \na_v f\\
      &\qquad \qquad \qquad \qquad+ \pa_{v_k} F \cdot \na_v \pa_{v_l} f + \pa_{v_l} F \cdot \na_v \pa_{v_k} f + \pa_{v_k, v_l}^2 F \cdot \na_v  f \Big).
    \end{align*}
    We can rewrite
    \begin{multline*}
     \varphi_{k,l}^{i,j} \pa_{v_l} a(v) \cdot \na_x  \pa_{x_k}f  =   \varphi_{k,l}^{i,j} \sum_{m}  \pa_{v_l} a(v)_m  \pa_{x_m} \pa_{x_k}f  \\
     =    \varphi_{k,l}^{i,j}  \sum_{m} \pa_{v_l}a(v)_m  \left(  L_{k,m} f - \sum_{k',l'} \left( \varphi_{k',l'}^{k,m}\pa_{x_{k'}} \pa_{v_{l'}} + \psi_{k',l'}^{k,m} \pa_{v_{k'},v_{l'}}^2 \right)f \right),
    \end{multline*}
    which entails that 
    \begin{align*}
   L_{i,j} \Tc (f) &= \Tc  L_{i,j} (f) +  \pa_{x_i x_j} F \cdot \na_v f \\
   +&  \sum_{  k,l } \left( \varphi_{k,l}^{i,j}  \pa_{x_k} \pa_{v_l} F \cdot \na_v f + \psi_{k,l}^{i,j} \pa_{v_k,v_l}^2 F \cdot \na_v f + \psi_{k,l}^{i,j} \pa_{v_k,v_l}^2 a(v) \cdot \na_x f  \right) \\
   +& \sum_{k,l,m} \pa_{v_l}a(v)_m  \varphi^{i,j}_{k,l}  L_{k,m} f \\ 
   + &\sum_{k,l} \pa_{x_k} \pa_{v_l} f \Bigg[-\Tc  \varphi_{k,l}^{i,j} + \sum_{k'} \pa_{v_{k'}} a(v)_k \psi_{k',l}^{i,j} +\sum_{k'} \pa_{v_{k'}} a(v)_k   \psi_{l,k'}^{i,j}  \\
   &\qquad -\sum_{k',l',m} \pa_{v_{l'}}a(v)_m \varphi_{k',l'}^{i,j}   \varphi_{k,l}^{k',m}  + \delta_{k,j} \pa_{x_i}F_l + \delta_{k,i} \pa_{x_j}F_l  + \sum_{l'} \varphi_{k,l'}^{i,j}  \pa_{v_{l'}}  F_l  \Bigg] \\
   + &\sum_{k,l} \pa_{v_k, v_l}^2 f \Bigg[ -\Tc  \psi_{k,l}^{i,j} -   \sum_{k',l',m} \pa_{v_{l'}}a(v)_m \varphi_{k',l'}^{i,j}  \psi_{k,l}^{k',m}  +  \varphi_{k,l}^{i,j} \pa_{x_k} F_k \\
   &\qquad +   \sum_{k'}  \psi_{k',l}^{i,j}   \pa_{v_{k'}} F_k + \sum_{l'}  \psi_{k,l'}^{i,j}   \pa_{v_{l'}} F_l   \Bigg].
   \end{align*} 
   We therefore deduce \eqref{LT}, because $(\varphi_{k,l}^{i,j}, \psi_{k,l}^{i,j})$ solves~\eqref{eq-constraint}.
    \end{proof}

\subsection{Composition of the second order operators} Relying on Lemma~\ref{Lij}, we shall use the $L_{i,j}$ operators in order to apply derivatives to the solution $f$ of the Vlasov equation~\eqref{abstract}.

Set for $I,J \in \{1, \cdots, d\}^k$,
\begin{equation}
L^{I,J} := L_{i_1, j_1} \cdots L_{i_k, j_k}.
\end{equation}
Let us also introduce the following useful notation. Given $I=(i_1,\cdots,i_k)$ and $J=(j_1,\cdots,j_k)$, we set
\begin{equation}
\alpha(I,J) := (i_1, j_1, \cdots, i_k,j_k).
\end{equation}
and 
\begin{equation}
\pa^{\alpha(I,J)}_x = \pa_{x_{i_1}} \pa_{x_{j_1}} \cdots \pa_{x_{i_k}} \pa_{x_{j_k}}.
\end{equation}
Note that by construction, 
$$
L^{I,J}|_{t=0} = \pa^{\alpha(I,J)}_x.
$$
In what follows, $f$ will systematically stand for the solution of~\eqref{abstract}, starting from $f_0$ verifying the assumptions of Theorem~\ref{thm}.

\subsection{Moments in $v$}
\label{subsec-comp}
We study in this section the moments in $v$ of the $L^{I,J} f$.
Until the end of the section, the times $T>0$ will be such that
the solution to~\eqref{abstract} satisfies
$$
\| f\|_{L^\infty(0,T;\Hc^{2m-1}_r)} \leq 2 R_0,
$$
thanks to Proposition~\ref{lem-local}.
\begin{lemma}
\label{lem-momentsLij}
\begin{itemize}
\item 
Let $k=0,\cdots,p$ and $I,J \in \{1, \cdots, d\}^{m+k}$.
Assume that the force field satisfies $F^j \in L^2(0,T; H^{2(m+k)-1}_x)$ with norm bounded by $\Lambda(T,M)$.
Assume that for all $n=2m, \cdots, 2(m+k)-1$, for all $\varphi  \in {L^\infty(0,T; \mathcal{W}^{d+2+n-2m,\infty}_{-r_0})}$ such that $\|\varphi \|_{L^\infty(0,T; \mathcal{W}^{d+2+n-2m,\infty}_{-r_0})}\leq \Lambda(T,M)$, and all $|\a|=n$, we have
\begin{equation}
\label{assump-moments}
\left\| \int_{\R^d}  (\pa^\a_x f) \varphi(t,x,v) \,dv \right\|_{L^2(0,T;L^2_x)} \leq \Lambda(T,M).
\end{equation}
Let $\psi \in L^\infty(0,T; \mathcal{W}^{d+2+2k,\infty}_{-r_0})$ satisfying $\|\psi\|_{L^\infty(0,T; \mathcal{W}^{d+2+2k,\infty}_{-r_0})} \leq \Lambda(T,M)$.
We have
\begin{equation}
\label{eq-mom1}
\int_{\R^d}  L^{I,J} f \, \psi(t,x,v) \,dv =  \int_{\R^d}\pa^{\alpha(I,J)}_x f \, \psi(t,x,v) \,dv + \mathfrak{R}_{I,J,\psi},
\end{equation}
where $\mathfrak{R}_{I,J,\psi}$ is a remainder satisfying the estimate
\begin{equation}
\label{eq-mom2}
\| \mathfrak{R}_{I,J,\psi}\|_{L^2(0,T;L^2_x)} \leq \Lambda(T,M).
\end{equation}
\item  Let $k=0,\cdots,p-1$ and $I,J \in \{1, \cdots, d\}^{m+k}$.
Assume that the force field satisfies  $F^j \in L^2(0,T; H^{2(m+k)}_x)$ with norm bounded by $\Lambda(T,M)$.
Assume that for all $n=2m, \cdots, 2(m+k)$, for all $\varphi  \in {L^\infty(0,T; \mathcal{W}^{d+2+n-2m,\infty}_{-r_0})}$ such that $\|\varphi \|_{L^\infty(0,T; \mathcal{W}^{d+2+n-2m,\infty}_{-r_0})}\leq \Lambda(T,M)$, and all $|\a|=n$, we have
\begin{equation}
\label{assump-moments2}
\left\| \int_{\R^d}  (\pa^\a_x f) \varphi(t,x,v) \,dv \right\|_{L^2(0,T;L^2_x)} \leq \Lambda(T,M).
\end{equation}
Let $\psi \in L^\infty(0,T; \mathcal{W}^{d+3+2k,\infty}_{-r_0})$ satisfying $\|\psi\|_{L^\infty(0,T; \mathcal{W}^{d+3+2k,\infty}_{-r_0})} \leq \Lambda(T,M)$.
Let $\partial =\partial_{x_i}$ or $\partial_{v_i}$ for some $i \in \{1,\cdots,d\}$. 
We have
\begin{equation}
\int_{\R^d} \pa  L^{I,J} f \, \psi(t,x,v) \,dv =  \int_{\R^d}  \pa^{\alpha(I,J)}_x \pa f \, \psi(t,x,v) \,dv + \mathfrak{R}_{I,J,\psi},
\end{equation}
where $\mathfrak{R}_{I,J,\psi}$ is a remainder satisfying the estimate
\begin{equation}
\| \mathfrak{R}_{I,J,\psi}\|_{L^2(0,T;L^2_x)} \leq \Lambda(T,M).
\end{equation}
\end{itemize}
\end{lemma}

This result will allow to set up an induction argument: indeed, with the assumption \eqref{assump-moments} (resp. \eqref{assump-moments2}) that corresponds to regularity of the moments up to order $2(m+k)-1$ (resp. $2(m+k)$), this will imply that controlling the moments of the $(L^{I,J} f)$ gives information on the regularity of the moments up to order $2(m+k)$  (resp. $2(m+k)+1$).

\begin{proof}[Proof of Lemma~\ref{lem-momentsLij}]
Let us focus only on the first item (the proof of the second one is completely similar).
Let $\psi \in L^\infty(0,T; \mathcal{W}^{d+2+2k,\infty}_{-r_0})$.
The beginning of the proof closely follows that of Lemma 4.3 of \cite{HKR}.
  At first, we can expand $f_{I,J}=  L^{I,J} f$ in a more tractable form.
        Let us  set for readability 
        $$U:= ( \varphi^{i_{\alpha}, j_{\beta}}_{k',l}, \psi^{i_{\alpha},j_{\beta}}_{k',l})_{1\leq k',l\leq d, \, 1 \leq \alpha, \beta \leq  m+k}.$$
     Then, by induction, we obtain
     \begin{align} f_{I,J}
     \label{Lmexp}
     & = \partial_{x}^{\alpha (I,J)} f + \sum_{s=0}^{2(m+k)-2} \sum_{e, \, \alpha, \, k_{0},  \cdots, k_{s}}  P_{s, e,\alpha}^{k_{0}}(U) P_{s,e, \alpha}^{k_{1}}(\partial U) \cdots P_{s,e,\alpha}^{k_{s}}(\partial^{s} U) \partial_{v}^e \partial^{\alpha} f 
     \\ \nonumber &  =:  \partial_{x}^{\alpha (I,J)} f  + 
       \sum_{s=0}^{2(m+k)-2} \sum_{e, \, \alpha, \, k_{0},\cdots, k_{s}} \mathcal{R}_{s,e, \alpha}^{k_{0},\cdots, k_{s}},
       \end{align}
     where  the sum is  taken on indices such that
      \begin{equation}
      \label{indices}
       |e|= 1, \, | \alpha |= 2(m+k) - 1-s, \, k_{0} + k_{1}+ \cdots k_{s} \leq m+k , \, k_{0} \geq 1, \, k_{1} + 2 k_{2}+ \cdots s k_{s}= s,
      \end{equation}
       and  for all $0\leq p \leq s$,  $ P_{s,e, \alpha}^{k_{p}}(X)$ is a  polynomial of degree smaller than   $k_{p}$ (we denote by $\partial^{k} U$ the vector made of all the partial derivatives of length $k$ of all components of $U$). 
                We can set
         $$ \mathfrak{R}_{I,J,\psi}=\int_{\mathbb{R}^d} \psi(\cdot, v)  \sum_{s=0}^{2(m+k)-2} \sum_{e, \, \alpha, \, k_{0},\cdots, k_{s}} \mathcal{R}_{s,e, \alpha}^{k_{0},\cdots, k_{s}}\, dv,$$
       so that we  have to estimate $\int_{\mathbb{R}^d}  \psi \mathcal{R}_{s,e, \alpha}^{ k_{0},\cdots, k_{s}} \, dv$.
        All the following estimates are uniform in time  for $t \in [0,T]$ 
        and we shall dismiss 
         the time parameter to ease readability.
        
               We begin by estimating the  terms for which  $s \geq 2k+1$. Note that for all these  terms the total number of derivatives applied to $f$ is at most  $2m-1$.

\medskip
            
\noindent $\bullet$  When $ s< 2(m+k)-{ d \over 2} -2$, we can use estimate~\eqref{est-Lij} in Lemma~\ref{lem-Lij}  to obtain that
             $$ \|P_{s, e,\alpha}^{k_{0}}(U) P_{s,e, \alpha}^{k_{1}}(\partial U) \cdots P_{s,e,\alpha}^{k_{s}}(\partial^{s} U)\|_{L^\infty_{x,v}} \leq \Lambda (T,M)$$
              and hence using that 
              $$
              \sup_v | (1+ |v|^2)^{-r_0/2} \psi (\cdot , v) | \leq \Lambda(T,M) 
              $$
              we obtain by Cauchy-Schwarz that since $r> r_0 + r'$, for some $r'>d/2$, we have
              $$  
                \begin{aligned}\left\| \int  \mathcal{R}_{e,s, \alpha}^{k_{0},  \cdots , k_{s}} \, dv\right\|_{L^2_{x}} &\leq  \left \| \|{ (1 + |v|^2)^{-r_0 - r' \over 2 } } \psi  \|_{L^2_{v}} \|  (1 + |v|^2)^{(r_0+r')\over 2}\partial_{v}^e \partial^\alpha f \|_{L^2_v} \right\| \\
                &\leq  \Lambda (T, M) \left(  \int_{\R^d} \frac{dv}{(1+|v|^2)^{r'}} \right)^{1/2} \|f\|_{\Hc^{2m-1}_{r}}  \\&\leq  \Lambda (T, M).
                \end{aligned}$$
          
     \medskip
            
\noindent $\bullet$  Let us now consider  $s \geq 2(m+k)- 2 - {d \over 2}.$
             We start with  the case  where in the sequence $(k_{1}, \cdots, k_{s})$, the  largest index $l$ such that $k_{l} \neq 0$ and $k_{p}$ = 0 for every $p >l$
              is such that $l >s/2$. In this case, since $l k_l \leq s$ has to hold, we necessarily have $k_{l}= 1$. Moreover, for the indices $p<l$ such that $k_{p} \neq 0$, we must have $p\leq p k_p <s/2$.
                Thus, we can use estimate~\eqref{est-Lij} in Lemma~\ref{lem-Lij}  to bound $\| \partial^{p} U \|_{L^\infty_{x,v}}$ provided
                 $s/2 \leq 2(m+k) - {d \over 2} - 2$. Since $s\leq 2(m+k)-2$, this is verified thanks 
                 to the assumption that $ 2 m >2+d$.
               We thus obtain   that
               $$   \left\| \int  \mathcal{R}_{e,s, \alpha}^{k_{0}, \cdots k_{s}} \, dv\right\|_{L^2_{x}} \leq \Lambda (T,M) \left  \| \int \psi \partial^{{l}} U
                \partial_{v}^e  \partial^\alpha f \, dv \right \|_{L^2_{x}}.
                $$
                 Next, we can use the fact that 
               \begin{align*}
                 \left  \| \int  \psi \partial^{{l}} U
                \partial_{v}^e  \partial^\alpha f \right \|_{L^2_{x}}  & \lesssim \Lambda(T,M)  \left \| \|{   (1 + |v|^2)^{-r \over 2 } } \partial^{{l}} U \|_{L^2_{v}} \|  (1 + |v|^2)^{(r_0+r)\over 2}\partial_{v}^e \partial^\alpha f \|_{L^2_{v}} \right \|_{L^2_{x}} \\
                &   \lesssim \Lambda(T,M)  \|  U \|_{\Hc_{-r}^{2m-2}} \sup_{x} \|   (1 + |v|^2)^{r \over 2} \partial^e_v \partial^\alpha f \|_{L^2_{v}}.
                \end{align*}
                  By Sobolev embedding in $x$, we have
                  $$   \sup_{x} \|  (1 + |v|^2)^{r \over 2}\partial^e_v \partial^\alpha f \|_{L^2_{v}} \lesssim \|f\|_{\Hc^{2m-1}_{r}}$$
                   as soon as  $2m-1> 1 + | \alpha | + {d \over 2}=  1 + 2(m+k)-1 - s + {d \over 2}$ which  is equivalent to 
                    $s > 1 +2k + {d \over 2}$. Since we are in the case where $s    \geq 2(m+k)- 2 - {d \over 2}$, the condition is matched, thanks to the assumption $ 2 m >3+d$.
                     Consequently, by using estimate~\eqref{est-Lij} in Lemma~\ref{lem-Lij}, we obtain again  that
               $$  \left\| \int  \mathcal{R}_{e,s, \alpha}^{k_{0}, \cdots k_{s}} \, dv\right\|_{L^2_{x}} \lesssim \Lambda(T, M).$$
                Finally, it remains to handle the case where $k_{l}=0$ for every $l>s/2$. 
                 As above, we necessarily have  ${s / 2} < 2 (m+k) - {d \over 2} - 2$ 
                  and hence by using again estimate~\eqref{est-Lij} in Lemma~\ref{lem-Lij}, we  find
                  $$ \|\partial^l U \|_{L^\infty_{x,v}} \leq \Lambda (T, M), \quad l \leq s/2.$$
                   We deduce
                   $$   \left\| \int  \mathcal{R}_{e,s, \alpha}^{k_{0}, \cdots k_{s}} \, dv\right\|_{L^2_{x}} 
                    \leq \Lambda (T, M)  \|f \|_{\Hc^{2m-1}_{r}} \leq \Lambda(T, M).$$

\medskip

    It remains to treat the cases corresponding to $s\leq 2k$, that is to say
          $\mathcal{R}_{e,s, \alpha}^{ k_{0},\cdots k_{s}}$ contains the maximal number of derivatives applied to $f$. This means that
           $|\alpha |= 2m-1, \cdots, 2(m+k)-1$ so that at least $2m$ derivatives of $f$ are involved.
          We denote for readability the associated coefficient 
          $$\Gamma := \psi P_{s, e,\alpha}^{k_{0}}(U) P_{s,e, \alpha}^{k_{1}}(\partial U) \cdots P_{s,e,\alpha}^{k_{s}}(\partial^{s} U),$$
          and we have to study the $L^2_x$ norm of $\int \Gamma \pa^e_v \pa^\a f \, dv$. 
         
       First, assume that $|\alpha| \leq 2(m+k)-2$ (which corresponds to $s\geq1$). We note that for all $s'=0, \cdots, 2(m+k)-1-|\a|$, 
       we have by Lemma~\ref{lem-Lij} that 
       $$
      \| \pa^{s'} U \|_{W^{k, \infty}_{x,v}} \leq \Lambda(T,M), \quad \forall k< 2(m+k)-2 -d/2-s'.
       $$
       Since $s'\leq 2(m+k)-1-|\a|$, $2(m+k)-2 -d/2-s' \geq |\a| -d/2-1 > d+2+|\a|+1-2m$ since 
       $2m > 3d/2 +4$. Therefore 
       $$\| \pa^{s'} U \|_{W^{d+2+|\a|+1-2m, \infty}_{x,v}}  \leq \Lambda(T,M)$$ 
       and  
       $$\| \Gamma \|_{ \mathcal{W}^{d+2+|\a|+1-2m,\infty}_{-r_0}} \leq \Lambda(T,M).$$
       We can thus use the assumption~\eqref{assump-moments} to obtain the bound
       \begin{equation}
       \label{eq-Gamma}
       \left\| \int \Gamma \pa^e_v \pa^\a f \, dv \right\|_{L^2_x} \leq \Lambda(T,M).
       \end{equation}
       
                    Assume finally that $|\alpha| = 2(m+k)-1$ (which corresponds to $s=0$), that is to say $2(m+k)$ derivatives of $f$ are involved.  We can write, by integration by parts in $v$ (relying on the fast decay of $f$ and its derivatives at infinity)
\begin{align*}
\int_{\R^d} \Gamma \pa^e_v \pa^\a  f \, dv=- \int_{\R^d} \pa^{e}_v \Gamma  \pa^{\a}   f \, dv.
\end{align*}
We have that 
       $$\| \pa^{e}_v\Gamma \|_{ \mathcal{W}^{d+1+2k,\infty}_{-r_0}} \leq \Lambda(T,M),$$
 and we can use again~\eqref{assump-moments} to obtain       \begin{equation*}
       \left\| \int \pa^{e}_v \Gamma  \pa^\a f \, dv \right\|_{L^2_x} \leq \Lambda(T,M).
       \end{equation*}
       In summary, we have proved that
       $$
       \left\| \mathfrak{R}_{I,J,\psi} \right\|_{L^2_x} \leq \Lambda(T,M).
       $$
  This ends the proof of the lemma.

\end{proof}

\subsection{The equation satisfied by $L^{I,J}f$}
Using the algebraic identities of Lemma \ref{Lij}, we obtain
\begin{lemma}
\label{even}
For all $k=0, \cdots, p$, the following holds. Assume that $(F^j) \in L^2(0,T; H^{2(m+k)-1}_x)$ with norm bounded by $\Lambda(T,M)$. For
  all $I,J \in \{1, \cdots, d\}^{m+k}$, we have
\begin{equation}
\label{even-eq}
\begin{aligned}
&\Tc ( L^{I,J} f) +   \pa^{\alpha(I,J)}_x F \cdot \na_v f =
\\ 
&\qquad \sum_{r=m-k}^{m+k}\sum_{K, L \in \{1,\cdots,d\}^r} \sum_{|\alpha| +|\beta| = m + k - r} \gamma^{I,J}_{K,L,\alpha,\beta} L^{K,L} \pa^\alpha_x \pa^\beta_v  f + R_{I,J}
\end{aligned}
\end{equation}
where 
\begin{itemize}
\item $\gamma^{I,J}_{K,L,\alpha,\beta}$ are coefficients satisfying
\begin{equation}
\label{coef1}
\|  \gamma^{I,J}_{K,L,\alpha,\beta} \|_{L^2(0,T;W^{d+2, \infty}_{x,v})}  \lesssim  \Lambda(T,M),
\end{equation}
\item $R_{I,J}$ is a remainder satisfying 
$$\| R_{I,J} \|_{L^\infty(0,T; \Hc^{0}_{\r})} \lesssim \Lambda(T,M), \qquad \forall \r\leq r-d/2.$$
\end{itemize}
\end{lemma}
We mention that a version of this lemma was proved in \cite{HKR} in the case $k=0$. 

Lemma~\ref{even} will be useful in the induction argument to treat the case of  even integers.
For what concerns odd integers, we have the following result.
\begin{lemma}
\label{odd}
For all $k=0, \cdots, p-1$, the following holds.
Assume that $(F^j) \in L^2(0,T; H^{2(m+k)}_x)$ with norm bounded by $\Lambda(T,M)$. For all $I,J \in \{1, \cdots, d\}^{m+k}$, and $i=1,\cdots,d$,
 we have
\begin{equation}
\begin{aligned}
&\Tc ( L^{I,J}\pa_{x_i} f) +   \pa_{x_i} \pa^{\alpha(I,J)}_x   F \cdot \na_v f =
\\ 
&\qquad \sum_{r=m-k-1}^{m+k}\sum_{K, L \in \{1,\cdots,d\}^r} \sum_{|\alpha| +|\beta| = m + k +1 - r} \gamma^{x_i,I,J}_{K,L,\alpha,\beta} L^{K,L} \pa^\alpha_x \pa^\beta_v  f + R_{x_i, I,J}
\end{aligned}
\end{equation}
and 
\begin{equation}
\begin{aligned}
&\Tc (L^{I,J}  \pa_{v_i}  f) +    \pa_{v_i} \pa^{\alpha(I,J)}_x  F \cdot \na_v f =
\\ 
&\qquad \sum_{r=m-k-1}^{m+k}\sum_{K, L \in \{1,\cdots,d\}^r} \sum_{|\alpha| +|\beta| = m + k +1 - r} \gamma^{v_i,I,J}_{K,L,\alpha,\beta} L^{K,L} \pa^\alpha_x \pa^\beta_v  f + R_{v_i, I,J}
\end{aligned}
\end{equation}
where 
\begin{itemize}
\item $\gamma^{x_i I,J}_{K,L,\alpha,\beta}$, $\gamma^{v_i I,J}_{K,L,\alpha,\beta}$ are coefficients satisfying
\begin{equation}
\label{coef2}
\|  \gamma^{x_i,I,J}_{K,L,\alpha,\beta},\gamma^{v_i,I,J}_{K,L,\alpha,\beta} \|_{L^2(0,T; W^{d+2, \infty}_{x,v})} \lesssim \Lambda(T,M), 
\end{equation}
\item $R_{x_i, I,J}$, $R_{v_i, I,J}$ are remainders satisfying 
$$\| R_{x_i, I,J} \|_{L^\infty(0,T; \Hc^0_{\r})}  + \| R_{v_i, I,J} \|_{L^\infty(0,T; \Hc^0_{\r})} \lesssim \Lambda(T,M), \qquad \forall \r\leq r-d/2.$$
\end{itemize}
\end{lemma}

\subsection{The equation satisfied by $L^{I,J} \pa^\alpha_x \pa^\beta_v  f$}
Lemma~\ref{even} invites to seek for a closed equation on $L^{I,J} \pa^\alpha_x \pa^\beta_v  f$, for $k \in \{0,\cdots p\}$, $r \in \{m-k,\cdots, m+k\}$, $I,J \in \{1, \cdots, d\}^{r}$ and all $|\a| + |\b| = m+k-r$ (and similarly for what concerns Lemma~\ref{odd}).
This is the purpose of the next two lemmas.

\begin{lemma}
\label{even'}
Let $k=0, \cdots, p$.
Let $r =m-k, \cdots, m+k$.
Assume that $(F^j) \in L^2(0,T; H^{2(m+k)-1}_x)$ with norm bounded by $\Lambda(T,M)$. For
  all $I,J \in \{1, \cdots, d\}^{r}$ and all $|\a| + |\b| = m+k-r$, we have
\begin{equation}
\label{even-eq-general}
\begin{aligned}
&\Tc ( L^{I,J}\pa^\alpha_x \pa^\beta_v  f) +   \pa^\alpha_x \pa^\beta_v\pa^{\alpha(I,J)}_x F \cdot \na_v f =
\\ 
&\qquad \sum_{r'=m-k}^{r}\sum_{K, L \in \{1,\cdots,d\}^{r'}} \sum_{|\alpha'| +|\beta'| = m+k- r'} \gamma^{I,J,\a,\b}_{K,L,\alpha',\beta'} L^{K,L} \pa^{\alpha'}_x \pa^{\beta'}_v  f + R_{I,J,\a,\b}
\end{aligned}
\end{equation}
where 
\begin{itemize}
\item $\gamma^{I,J,\a,\b}_{K,L,\alpha',\beta'}$ are coefficients satisfying
\begin{equation}
\label{coef3}
\|  \gamma^{x_i,I,J}_{K,L,\alpha,\beta},\gamma^{v_i,I,J}_{K,L,\alpha,\beta} \|_{L^2(0,T;W^{d+2, \infty}_{x,v})} \lesssim  \Lambda(T,M),
\end{equation}
\item $R_{I,J,\a,\b}$ is a remainder satisfying 
$$\| R_{I,J,\a,\b} \|_{L^\infty(0,T; \Hc^0_{\r})} \lesssim \Lambda(T,M), \qquad \forall \r\leq r-d/2.$$
\end{itemize}
\end{lemma}

\begin{lemma}
\label{odd'}
Let $k=0, \cdots, p-1$.
Let $r =m-k-1, \cdots, m+k$.
Assume that $(F^j) \in L^2(0,T; H^{2(m+k)}_x)$ with norm bounded by $\Lambda(T,M)$. For all $I,J \in \{1, \cdots, d\}^{r}$,  and all $|\a| + |\b| = m+k+1-r$, we have
\begin{equation}
\label{odd-eq-general}
\begin{aligned}
&\Tc ( L^{I,J}\pa^\alpha_x \pa^\beta_v  f) +   \pa^\alpha_x\pa^\beta_v  \pa^{\alpha(I,J)}_x F \cdot \na_v f =
\\ 
&\qquad \sum_{r'=m-k-1}^{r}\sum_{K, L \in \{1,\cdots,d\}^{r'}} \sum_{|\alpha'| +|\beta'| = m + k +1 - r'} \gamma^{I,J,\a,\b}_{K,L,\alpha',\beta'} L^{K,L} \pa^{\alpha'}_x \pa^{\beta'}_v  f + R_{I,J,\a,\b}
\end{aligned}
\end{equation}
where 
\begin{itemize}
\item $\gamma^{I,J,\a,\b}_{K,L,\alpha',\beta'}$ are coefficients satisfying
\begin{equation}
\label{coef4}
\|  \gamma^{x_i,I,J}_{K,L,\alpha,\beta},\gamma^{v_i,I,J}_{K,L,\alpha,\beta} \|_{L^2(0,T; W^{d+2, \infty}_{x,v})} \lesssim  \Lambda(T,M),
\end{equation}
\item $R_{I,J,\a,\b}$ is a remainder satisfying 
$$\| R_{I,J,\a,\b} \|_{L^\infty(0,T; \Hc^0_{\r})} \lesssim \Lambda(T,M), \qquad \forall \r\leq r-d/2.$$
\end{itemize}
\end{lemma}

We observe that as wanted, Lemmas~\ref{even'} and~\ref{odd'} provide \emph{closed} systems of equations.

\bigskip

To conclude this section, we shall  give the proofs of Lemmas~\ref{even} and~\ref{even'} (the proofs of the remaining Lemmas~\ref{odd} and~\ref{odd'} being very similar).

\subsection{Proof of Lemmas~\ref{even} and~\ref{even'}}

\begin{proof}[Proof of Lemma~\ref{even}]
Let $\r< r-d/2$. Since $r>d$, we can assume, without loss of generality, that $\r > d/2$.
 We can write, by an induction argument relying on Lemma~\ref{Lij}, that
 $$
 \Tc ( L^{I,J} f)  =   F_{I,J}
$$
 with the source term $F_{I,J}$  is given by
$F_{I,J}= -  \sum_{i=1}^4 F_{i}$, 
   where 
  \begin{align}
  \label{F1def}
 F_{1}&= \sum_{\ell=1}^{m+k-1} L_{i_{1}, j_{1} } \cdots L_{i_{m+k-\ell}, j_{m+k-\ell} } \\&\qquad \qquad \times \Big(( \partial^2_{x_{i_{m+k-\ell+1} }, x_{j_{m+k-\ell+1} } } F)\cdot \nabla_{v} L_{i_{m+k-\ell+2},j_{m+k-\ell+2} } \cdots L_{i_{m+k}, j_{m+k} } f \Big), \nonumber\\
  \label{F2def}
 F_{2}&= \sum_{\ell=1}^{m+k-1} L_{i_{1}, j_{1} } \cdots L_{i_{m+k-\ell}, j_{m+k-\ell} }  \Big( \Big[ \sum_{  k,l } \varphi_{k,l}^{{i_{m+k-\ell+1}},{j_{m+k-\ell+1}}}  \pa_{x_{i_{m+k-\ell+1}}} \pa_{v_{j_{m+k-\ell+1} } } F  \\
  &+  \psi_{k,l}^{{i_{m+k-\ell+1}},{j_{m+k-\ell+1}}} \pa^2_{v_{i_{m+k-\ell+1}},v_{j_{m+k-\ell+1} }}F\Big]    \cdot \nabla_{v} L_{i_{m+k-\ell+2},j_{m+k-\ell+2} } \cdots L_{i_{m+k}, j_{m+k} } f \Big), \nonumber\\
    \label{F3def}
       F_{3}&= \sum_{\ell=1}^{m+k-1} L_{i_{1}, j_{1} } \cdots L_{i_{m+k-\ell}, j_{m+k-\ell} }  \Big( \Big[ \sum_{  k,l } \psi_{k,l}^{{i_{m+k-\ell+1}},{j_{m+k-\ell+1}}} \pa^2_{v_{i_{m+k-\ell+1}},v_{j_{m+k-\ell+1} }}a\Big] \\
        &\qquad\qquad\qquad \cdot \nabla_{x} L_{i_{m+k-\ell+2},j_{m+k-\ell+2} } \cdots L_{i_{m+k}, j_{m+k} } f \Big), \nonumber\\
        \label{F4def}
       F_{4}&=  \sum_{\ell=1}^{m+k-1} L_{i_{1}, j_{1} } \cdots L_{i_{m+k-\ell},  j_{m+k-\ell}} \\
       &\times\sum_{k',l',m'} \pa_{v_{l'}}a(v)_{m'}   \varphi^{i_{m+k-\ell+1}, j_{m+k-\ell+1} }_{k',l'} L_{k',m'} L_{i_{m+k-\ell+2},j_{m+k-\ell+2} } \cdots L_{i_{m+k}, j_{m+k} } f .
       \nonumber
   \end{align}

   \bigskip
   
We shall focus on the contribution of $F_{1}$. We have to estimate terms under the form
\begin{equation}
\label{termsF1}
 F_{1, \ell}= L^{m+k-\ell} G_{\ell}, \quad G_{\ell}= \partial^2 E \cdot \nabla_{v} L^{\ell-1}
 \end{equation}
 where we use the notation $L^{n}$ for the composition of $n$ operators of type  $L_{i, j}$  (the exact combination of the operators involved in the composition does not matter here).
  Note that as in \eqref{Lmexp}, we can  develop  $L^n$ under the form
  \begin{equation}
  \label{Lpexp}
  L^n= \partial_{x}^{\alpha_{n}} +  \sum_{s=0}^{2n-2} \sum_{e, \, \alpha, \, k_{0}  \cdots k_{s}}  P_{s, e,\alpha}^{k_{0}}(U) P_{s,e, \alpha}^{k_{1}}(\partial U) \cdots P_{s,e,\alpha}^{k_{s}}(\partial^{s} U) \partial_{v}^e \partial^{\alpha}, 
    \end{equation}
     where for all $0\leq p \leq s$,  $ P_{s,e, \alpha}^{k_{p}}(X)$ is a polynomial of degree smaller than   $k_{p}$, the multi-index $\alpha_{n}$ has length $2n$ and  the sum is  taken on indices such that
      \begin{equation}
      \label{indices2}
       |e|= 1, \, | \alpha |= 2n - 1-s, \, k_{0} + k_{1}+ \cdots k_{s} \leq n , \, k_{0} \geq 1, \, k_{1} + 2 k_{2}+ \cdots s k_{s}= s.
      \end{equation}

       Let us first establish a general  estimate, adapted from \cite{HKR}.
        We set for  any fonction $G(x,v)$,
         \begin{equation}
         \label{Jp}
         J_{p}(G)(x,v) =  \sum_{s, \, \beta,  K \in E} J_{p, s, \beta, K}(G)
         \end{equation}
          where $K= (k_{0}, \cdots, k_{s})$ and  
        \begin{equation}
         \label{Jbeta}J_{p, s, \beta, K} (G) (x,v)=  P^{k_{0}}_{s, \beta}(U) P^{k_{1}}_{s, \beta}(\partial U) \cdots P^{k_{s}}_{s, \beta}(\partial^{s} U)  \partial^{\beta}  G
             \end{equation}
        where  for all $0\leq r \leq s$, $ P^{k_{r}}_{s, \beta}(X)$ is a  polynomial of degree smaller than   $k_{r}$ and the sum is taken over indices belonging to the set $E$ defined by 
\begin{equation}
\label{indices1}  \, | \beta |= p -s, \, k_{0} + k_{1}+ \cdots k_{s} \leq p/2 , \, k_{1} + 2 k_{2}+ \cdots s k_{s}= s, \, 0 \leq s \leq p-2.
\end{equation}
 \begin{lemma}
 \label{JseK}
  For $2(m+k)-1 \geq p  \geq 1$,  $ 2m >d + 3$, $\r>d/2$ and $s, \, p, \, K$ satisfying  \eqref{indices1},  we have the estimate
  \begin{equation}
  \label{estJseK}
  \|  J_{p} (G) \|_{\Hc^0_{\r}} \leq \Lambda(T, M) \Bigl(  \| G \|_{\Hc^{p}_{\r}} +  \sum_{ \begin{array}{ll} {\scriptstyle l \geq 2(m+k) - {d \over 2} - 2,} \\ {\scriptstyle  l+ | \alpha | \leq p, \, | \alpha | \geq 2}  \end{array}} \| \partial^l U   \partial^\alpha G \|_{\Hc^0_{\r}} \Bigr).
  \end{equation}
 \end{lemma}
 \begin{proof}[Proof of Lemma~\ref{JseK}]
 
  For the terms in the sum such that $s  < 2 (m+k) - { d \over 2} - 2$, we can use estimate~\eqref{est-Lij} in Lemma~\ref{lem-Lij}  to obtain that
 $$ \|  J_{p,s, \beta, K} (G) \|_{\Hc^0_{\r}} \leq \Lambda(T, M) \| G \|_{\Hc^{p}_{\r}}.$$
  When $s \geq  2 (m+k) - { d \over 2} - 2$, we first consider the terms for which in the sequence $(k_{1}, \cdots, k_{s})$ the largest index $l$ for which 
  $ k_{l} \neq 0$ is such that $l  < 2(m+k)-{d \over 2} - 2$. Then again thanks to  estimate~\eqref{est-Lij} in Lemma~\ref{lem-Lij}, we obtain that
  $$ \|  J_{p,s, \beta, K} (G) \|_{\Hc^0_{\r}} \leq \Lambda(T, M) \| G \|_{\Hc^{p}_{\r}}.$$
   When $l \geq 2(m+k)-{d\over 2} - 2$, we first observe that we necessarily have $k_{l}= 1$. Indeed if $k_{l} \geq 2$,  because of  \eqref{indices1}, we must have $ l \leq {s \over 2 }$. This is possible only if $  2(m+k)-{d\over 2} - 2 \leq   {p - 2 \over 2}  \leq  {2(m+k)-3 \over 2}$, which corresponds to $m+k \leq  {d \over 2} + 1$  and hence this is impossible. Consequently $k_{l} = 1$.
    Moreover we note that for the other indices $\tilde l$ for which  $k_{\tilde l} \neq 0$, because of \eqref{indices1}, we must have $ \tilde l k_{ \tilde l} \leq s -l k_l$, so that
    $$ \tilde l  \leq s- l \leq s - 2(m+k) + {d \over 2} + 2
     \leq {d \over 2} - 1$$ 
     and we observe that ${d \over 2} - 1 < 2m - {d \over 2} - 2$. Consequently, by another use of estimate~\eqref{est-Lij} in Lemma~\ref{lem-Lij}, we obtain that 
     $$  \|  J_{p,s, \beta, K} (G) \|_{\Hc^0_{\r}} \leq \Lambda(T, M)   \sum_{ \begin{array}{ll} \scriptstyle{ l\geq  2(m+k) - {d \over 2} - 2,} \\ \scriptstyle{ l+ | \alpha | \leq p, \, | \alpha | \geq 2}  \end{array}} \| \partial^l U   \partial^\alpha G \|_{\Hc^0_{\r}}.$$
      The fact that $| \alpha | \geq 2$ comes from \eqref{indices1}.
      This ends the proof of Lemma \ref{JseK}.
 \end{proof}
 
  We shall now estimate $F_{1, \ell}$. 
   Looking at the expansion of $L^{m+k-\ell}$ given by \eqref{Lpexp}, we have to estimate terms under the form $ J_{p}(G_{\ell})$ with $p \leq 2(m+k-\ell)$. 
 Using~\eqref{Lpexp}, we decompose $G_{\ell}$ in the following way:
 $$
 \begin{aligned}
 G_{\ell} &= \pa^2 F \cdot \na_v L^{\ell-1} f \\
 &= \pa^2 F \cdot \na_v \left( H_{\ell,+}+ H_{\ell,-}\right) \\
 &=:  G_{\ell,+} +  G_{\ell,-}.
 \end{aligned}
 $$
 where 
 \begin{itemize}
 \item in $H_{\ell+}$, we gather all  terms of the form~\eqref{Jbeta}, with $G=f$, such that $2k+1 + |\beta| \geq 2\ell$. These terms may contribute to terms with at least $2m$ derivatives on $f$.
\item On the other hand in $H_{\ell,-}$, the terms  that arise correspond to $2k+1 + |\beta| < 2\ell$, which involve at most $2m-1$ derivatives on $f$. 
 \end{itemize}
 \bigskip
 
  We first focus on the contribution of $G_{\ell,-}$; we denote  
  $$F_{1,\ell,-} := L^{m+k-\ell} G_{\ell,-}.$$ 
  Let us start with the case   $\ell \geq (m+k)/2$. We can  use the decomposition~\eqref{Lpexp}, which entails that we have to estimate terms of the form 
  $J_p(G_{\ell,-})$ with $p\leq 2(m+k-\ell)\leq 2(m+k)-1$, and apply Lemma~\ref{JseK} to get  
   \begin{multline}  
   \label{estJseK-ap}
    \|F_{1,\ell,-}\|_{L^2(0, T; \Hc^0_{\r})} \leq   \Lambda(T,M) \\
       \Bigl(  \|G_{\ell,-} \|_{L^2 (0, T; \Hc^{2(m+k-\ell)}_{\r})} + 
   \sum_{ \begin{array}{ll} \scriptstyle{ l \geq  2(m+k) - {d \over 2} - 2,} \\ \scriptstyle{ l+ | \alpha | \leq 2(m+k-\ell), \, | \alpha | \geq 2}  \end{array}} \| \partial^l U   \partial^\alpha G_{\ell,-} \|_{L^2(0, T; \Hc^0_{\r})}
     \Bigr).
     \end{multline}
  We observe that in the right hand side of \eqref{estJseK-ap}, we  have that 
  $l \leq  2(m+k-\ell) - 2 \leq m+k - 2$; consequently, since $2m-1 > d -1$, we have $ l <  2(m+k) - {d \over 2 } - 2$ and hence we can estimate
   $\| \partial^l U \|_{L^\infty}$ by using estimate~\eqref{est-Lij} in Lemma~\ref{lem-Lij}.
     This yields
    $$ \| F_{1,\ell,-} \|_{L^2(0, T; \Hc^0_{\r})} \leq \Lambda(T,M) \| G_{\ell,-}\|_{L^\infty(0, T; \Hc^{2(m+k-\ell) }_{\r})}, \quad \ell \geq (m+k)/2.$$
Then we  use  estimate  \eqref{com3}  in Lemma~\ref{lem-calculus} with $s=2(m+k-\ell)$ and $s_0 =d+1$, and the definition of $G_{\ell,-}$ to estimate $\| G_{\ell,-}\|_{\Hc^{2(m+k-\ell) }_{\r}}$.
  Since $ d+2 < 2 m-1$  and $ 2(m+k-\ell) + 2 \leq 2(m+k)-1$ (since $\ell  \geq (m+k)/2 \geq 2$), we obtain
  \begin{equation}
   \label{F1kinterm}
 \begin{aligned}
  &\|F_{1,\ell,-}\|_{L^2(0, T; \Hc^0_{\r})} \\
  & \leq   \Lambda(T,M) \big( \sup_j  \| F^j \|_{L^2(0, T; H^{d+1})} \|\nabla_vH_{\ell,-} \|_{L^\infty(0, T; \Hc^{2(m+k-\ell)}_{\r}) } \\
  &\qquad \qquad \qquad+ \sup_j \| F^j \|_{L^2(0, T; H^{2(m+k-\ell) + 2})} \| \nabla_vH_{\ell,-} \|_{L^\infty(0, T; \Hc^{2(m+k-\ell)}_{\r}) }  \big) \\
 &   \leq    \Lambda(T,M) \sup_j\| F^j \|_{L^2(0, T; H^{2(m+k)-1})} \|\nabla_v  H_{\ell,-}\|_{L^\infty(0, T; \Hc^{2(m+k-\ell)}_{\r})} .
 \end{aligned}
 \end{equation}
  By using  the regularity assumption on $F^j$, this yields
  $$   \|F_{1,\ell,-}\|_{L^2 (0, T; \Hc^0_{\r})}  \leq   \Lambda(T,M)  \|\nabla_v  H_{\ell,-} \|_{L^\infty(0, T; \Hc^{2(m+k-\ell)}_{\r})}.$$
   To estimate the above right hand side, we need to estimate
   $ \partial_{x,v}^\gamma H_{\ell,-}$ with  $| \gamma | \leq 2(m + k-\ell) + 1$.
  Recalling the definition of $H_{\ell,-}$,  by  taking derivatives using  the expression  \eqref{Lpexp},  we see that we have  to estimate terms under the form
     $ J_{p}(f)$  with $p \leq 2m-1$. Using Lemma~\ref{JseK} one more time, we thus obtain that
   \begin{multline*}  
    \|F_{1,\ell,-}\|_{L^2(0, T; \Hc^0_{\r})} \\
     \leq   \Lambda(T,M) \Bigl(  \| f \|_{L^\infty (0, T; \Hc^{2 m-1}_{\r})} + 
   \sum_{ \begin{array}{ll} \scriptstyle{ l \geq  2(m+k) - {d \over 2} - 2,} \\ \scriptstyle{ l+ | \alpha | \leq 2m-1, \, | \alpha | \geq 2}  \end{array}} \| \partial^l U   \partial^\alpha f \|_{L^\infty(0, T; \Hc^0_{\r})}
     \Bigr).
     \end{multline*}
   To estimate the right hand side, we argue as follows. 
   Let $r'>d/2$ such that $\r+r' \leq r$. 
        Since $| \alpha | \geq 2$ and $|\alpha| - 2 + l \leq 2m-3$, we can use estimate \eqref{comdual} in Lemma~\ref{lem-calculus} (taking $\chi(v) = ( 1 + |v|^2)^{r'\over 2}$), to  obtain that 
    \begin{equation}
\label{before}
 \| \partial^l U  ( 1 + |v|^2)^{\r\over 2} \partial^\alpha f\|_{L^2_{x,v}} \lesssim \|U\|_{\Hc^{2m-3}_{-r'}} \|  ( 1 + |v|^2)^{r}\partial^2 f \|_{L^\infty} +  
          \|U \|_{L^\infty} \|f \|_{\Hc^{2m-1}_{r}}.
\end{equation}
          By using again estimate~\eqref{est-Lij} in Lemma~\ref{lem-Lij}  and the Sobolev embedding, we finally obtain that
          \begin{equation}
          \label{F1k1} \|F_{1,\ell,-}\|_{L^2(0, T; \Hc^0_{\r})}  \leq   \Lambda(T,M) \|f\|_{L^\infty(0, T; \Hc^{2m-1}_{r})} \leq \Lambda(T, M), \quad \ell \geq (m+k)/2.
          \end{equation}
          
          It remains to handle the case $\ell \leq (m+k)/2$. Note that necessarily, for these cases to be meaningful, we must have $2k+1 < 2\ell$. Assume first $\ell \geq 2$. We obtain again~\eqref{estJseK-ap}. 
%
   We first need  to estimate $\|\pa^2 F \cdot \na_v H_{\ell,-} \|_{L^2 (0, T; \Hc^{2(m+k-\ell)}_{\r})}$.
     We thus have to study
     $$
     \|\pa^\beta \pa^2 F \cdot \na_v \pa^\gamma H_{\ell,-} \|_{L^2 (0, T; \Hc^{0}_{\r})}
     $$
     with $|\beta| + |\gamma| \leq 2(m+k-\ell)$. Since $\ell \geq 2$, we have $|\b|+2 \leq 2(m+k-1)$.
     If $|\b| +2 < 2(m+k)-1 -d/2$, then we get by Sobolev embedding the bound
       $$
       \begin{aligned}
     \|\pa^\beta \pa^2 F \cdot \na_v \pa^\gamma H_{\ell,-} \|_{L^2 (0, T; \Hc^{0}_{\r})} &\leq   
   \sup_j    \|\pa^\beta \pa^2 F^j  \|_{L^2 (0, T; L^\infty_x)}  
     \|\na_v \pa^\gamma H_{\ell,-} \|_{L^2 (0, T; \Hc^{0}_{\r})} \\
     &\leq      \sup_j \| F^j  \|_{L^2 (0, T;H^{2(m+k)-1}_x)}  \|   f \|_{L^2 (0, T; \Hc^{2m-1}_{\r})} \\
     &\leq \Lambda(T,M),
     \end{aligned}
     $$
 recalling the definiton of $H_{\ell,-}$. If $|\b|  \geq 2(m+k)-3 -d/2$, then $|\gamma| \leq 2(m+k-\ell) -2(m+k)+3 +d/2$ and thus the term $\na_v \pa^\gamma H_{\ell,-}$ involves at most $d/2 +2$ derivatives. 
 Since $2m-1 > 3d/2 +2$, we have this time
   $$
       \begin{aligned}
     \|\pa^\beta \pa^2 F \cdot \na_v \pa^\gamma &H_{\ell,-} \|_{L^2 (0, T; \Hc^{0}_{\r})} \\
     &\leq   
    \sup_j   \|\pa^\beta \pa^2 F^j   \|_{L^2 (0, T; L^2_x)}  
     \| H_{\ell,-} \|_{L^2 (0, T; \mathcal{W}^{d/2 +2,\infty}_{\r})} \\
     &\leq     \sup_j  \| F^j  \|_{L^2 (0, T;H^{2(m+k)-1}_x}  \|   f \|_{L^2 (0, T; \Hc^{2m-1}_{r})} \\
     &\leq \Lambda(T,M).
     \end{aligned}
     $$
     We also have to estimate terms in~\eqref{estJseK-ap} under the form 
     $$ \|   \partial^l U  \partial^\beta \partial^2 F\,  \partial^\gamma \nabla_{v} H_{\ell,-} \|_{\Hc^0_{\r}}$$
      with $  l \geq  2(m+k) - {d \over 2} - 2$ and $    l+ | \beta  | + | \gamma |  \leq 2(m + k-\ell).$
       Note that this implies that  $| \beta | \leq    2(m+k-\ell)-l \leq {d\over 2} + 2  - 2\ell \leq {d \over 2}$
        since we have $\ell \geq 1$. In particular this yields $|\beta | + 2 < 2m-1- {d \over 2} $ since $2m> 3 +{d\over 2}$, and thus 
       by using the  Sobolev embedding  and~\eqref{Fj}, we obtain that  
       \begin{align*}
          \|   \partial^l U  \partial^\beta \partial^2 F \partial^\gamma \nabla_{v}H_{\ell,-} \|_{\Hc^0_{\r}} 
                    &\lesssim \sup_j \| F^j \|_{H^{2m-1}_x} \| \partial^l U   \partial^\gamma \nabla_{v} H_{\ell,-} \|_{\Hc^0_{\r}} \\
                    &\lesssim (\| f \|_{\Hc^{2m-1}_{r} } + \sup_j \| F^j (0)\|_{H^{2m-1}_x} )\| \partial^l U   \partial^\gamma \nabla_{v} H_{\ell,-} \|_{\Hc^0_{\r}} \\
          &\leq \Lambda(T, M )   \| \partial^l U   \partial^\gamma \nabla_{v} H_{\ell,-}\|_{\Hc^0_{\r}}.
           \end{align*}
        Consequently, it remains to estimate  $ \| \partial^l U   \partial^\gamma \nabla_{v} H_{\ell,-} \|_{\Hc^0_{\r}} $ for 
         $  l \geq  2(m+k) - {d \over 2} - 2$ and $    l+ | \gamma |  \leq 2(m + k-\ell).$
         By using again \eqref{Lpexp} and the definition of $H_{\ell,-}$, we can expand  $ \partial^\gamma \nabla_{v} H_{\ell,-}$ as terms of  the form $J_{p} (f)$, with $p\leq 2(\ell-k)+ | \gamma |-1$.
          Since we have that
          $  2 (\ell-k) + | \gamma | - 1 \leq  1 + {d \over 2 } < 2(m+k)- {d \over 2} - 2$, we can use estimate~\eqref{est-Lij} in Lemma~\ref{lem-Lij}  again to estimate all  the terms in the expression of $J_{p} (f)$ involving $U$ 
           and its derivatives in 
           in $L^\infty$. This yields
           $$   \| \partial^l U   \partial^\gamma \nabla_{v} H_{\ell,-} \|_{\Hc^0_{\r}} \leq \Lambda( T, M) \sum_{\tilde \gamma}  \| \partial^l U  \partial^{\tilde \gamma} f \|_{\Hc^0_{\r}}$$
            with  $ |\tilde \gamma | \leq  |\gamma |+ 2 (\ell-k) - 1$. 
            Consequently, arguing as for~\eqref{before}, 
            we obtain that 
            $$ \| \partial^l U   \partial^\gamma \nabla_{v} H_{\ell,-} \|_{\Hc^0_{\r}} \leq \Lambda( T, M)  \left( \| U \|_{L^\infty}
             \| f \|_{\Hc^{2m-1}_{r}} +  \| (1 +  |v|^2)^{r}  f \|_{L^\infty_{x,v}} \|U \|_{\Hc^{2m-1}_{-r'}}\right),$$
     where we recall $r'>d/2$         and we conclude finally by invoking estimate~\eqref{est-Lij} in Lemma~\ref{lem-Lij}  and  the Sobolev embedding that
           \begin{equation}
           \label{F1k2} \|F_{1, \ell,-} \|_{L^2(0, T; \Hc^0_{\r})} \leq \Lambda(T, M), \quad 2 \leq \ell \leq (m+k)/2.
           \end{equation}
              For the case $\ell =1$ to be meaningful, $k$ must be equal to $0$. We set aside the term $ \pa^{\alpha(I,J)}_x F \cdot \na_v f $ (which appears in the formula~\eqref{even-eq}), and we thus have to study the term
              $$
 L_{i_{1}, j_{1} } \cdots L_{i_{m-1}, j_{m-1} } \Big( \partial^2_{x_{i_{m} }, x_{j_{m} } } F \cdot \nabla_{v} f \Big) -\pa^{\alpha(I,J)}_x F \cdot \na_v f .            $$
              We argue exactly as before to obtain a bound in ${L^2(0, T; \Hc^0_{\r})}$ by $\Lambda(T, M)$ (note indeed that at most $2m-1$ derivatives of $f$ and $F$ are involved).
              Gathering all pieces together, we have thus proven that
            \begin{equation}
            \label{estimationFk1}
              \|F_{1,\ell,-} \|_{L^2(0, T; \Hc^0_{\r})} \leq \Lambda(T, M).
              \end{equation}
              
              \bigskip
              
              Let us now treat the contribution of $G_{\ell,+}$, which will give rise to terms involving $2m$ up to $2(m+k)$ derivatives of $f$.
              Let $j \in \{0,\cdots, 2k\}$. Let us describe the form of the terms involving derivatives of order $2m+j$ of $f$. We note that $2m+j -1 \geq 2m -1  > m+p-1 \geq m+k-1$.
              This means that such terms are necessarily of the form 
              \begin{equation}
              \label{commu-diff}
             \Bigg( \pa^{\a^0}_x \pa^{\b^0}_v L_{i_1,j_1} \cdots  L_{i_k,j_k} \pa^{\a^k}_x \pa^{\b^k}_v\cdots  L_{i_{m+j-k},j_{m+j-k}}\pa^{\a^{m+j-k}}_x \pa^{\b^{m+j-k}}_v\Bigg) f ,
\end{equation}
              with 
              $$\sum_{k=0}^{m+j-k} |\a^k| + |\b^k| =  2k-j, \qquad \sum_{k=0}^{m+j-k} |\b^k| \neq 0.$$
              In order to have exactly $2m+j$ derivatives of $f$, this expression can be rewritten as
                $L^{K,L} \pa^\alpha_x \pa^\b_v f$,
              where $|K|=|L|= m+j-k$ and $|\a| + |\b| = 2k-j$, $|\b| \geq 1$. Indeed if derivatives
              fall on a coefficient of one of the $ L_{i_k,j_k}$, then there are less than $2m+j$ derivatives on $f$.
              
              We denote by $\gamma^{I,J,1}_{K,L,\alpha,\beta}$ the coefficient associated to such  terms.
              Remark that for $|\gamma| \leq 2k-j-1$, $\pa^\gamma_x \pa^2 F^i  \in L^2(0,T; H^{2m+j-2}_x)$.
              Since we have $2m>3d/2 +4$,  we can bound this term in $L^2(0,T; W^{d+2, \infty}_{x})$ by Sobolev embedding. 
              Likewise, for $|\gamma| \leq 2k-j-1$, since $2m+j-1-d/2> d+2$  we have $\pa^\gamma_{x,v} U  \in L^\infty(0,T; W^{d+2, \infty}_{x,v})$.   All in all, we deduce
              $$
              \| \gamma^{I,J,1}_{K,L,\alpha,\beta} \|_{L^2(0,T; W^{d+2, \infty}_{x,v})} \leq \Lambda(T,M).
              $$
          It remains to treat the other terms that all involve at most $2m-1$ derivatives are involved on $f$. If $k\geq 1$, we set aside the  term  $ \pa^{\alpha(I,J)}_x F \cdot \na_v f $ in~\eqref{even-eq},
 which corresponds to the case $\ell=1$ treated above (relevant when $k=0$). 
             
              The other terms  can be considered as remainders that are uniformly bounded in ${L^2(0, T; \Hc^0_{\r})}$, since at most $2m-1$ derivatives are involved on $f$
              and at most $2(m+k)-1$ derivatives are involved on $F$; these terms can be treated exactly as we did to treat the remainders in $G_{\ell,-}$.              

   The treatment of $F_2, F_3, F_4$ gives rise to similar terms and we omit it.  
   
   \end{proof} 
   
   \begin{proof}[Proof of Lemma~\ref{even'}] The proof is similar to the previous one.
   We shall only explain why the terms involving at least $2m$ derivatives of $f$ are indeed of the form appearing in~\eqref{even-eq-general}.

Let $k=0,\cdots,p-1$, and $r=m+j$, for $j = -k-1,\cdots, k$.
We look for the terms involving $2m+l$ derivatives of $f$, for $l=0,\cdots, k+1+j$.
   Among the operators in $L^{I,J}$ there are exactly $2m+l-(m+k+1-r)=2m+j+l-k-1$ derivatives to be applied on $f$. Since $m>p\geq k+1$, we have $2m+j+l-k-1> m+j$. This means that these derivatives must be of the form $L^{K,L} \pa^\gamma_x \pa^\delta_v $, with $|K|=|L|= m+l-k-1$ and
   $ |\gamma| +|\delta|=  j-l+k+1$ (up to commutations between the differential operators as in~\eqref{commu-diff}, which is treated like in the previous proof). In the end, the terms involving $2m+l$ derivatives of $f$ are thus necessarily of the form $L^{K,L} \pa^\gamma_x \pa^\delta_v f$, with
   $$
   |K|=|L|= m+l-k-1, \qquad |\gamma| +|\delta|= 2k+2-l,
   $$
   as appearing in~\eqref{even-eq-general}.
   
   \end{proof}

   \begin{remark}
   \label{rem-coef} An inspection of the proof reveals that the uniform regularity of the coefficients in~\eqref{coef1}, \eqref{coef2}, \eqref{coef3}, \eqref{coef4} can be improved to ${L^2(0,T; W^{p, \infty}_{x,v})} $ for all $p< 2m-2 -d/2$.
   \end{remark}
\section{Burgers equation and the semi-lagrangian approach}

{
In this section, we explain the procedure to straighten the transport operator $\Tc$, and which allows, loosely speaking, to come down to the operator $\pa_t + a(v) \cdot \na_v$. 
This relies on several changes of variables in velocity that we introduce now.}

\label{sec-Burg}
Let $\Phi(t,x,v)$  satisfy the Burgers equation
\begin{equation}\label{Burgers}
\left\{
\begin{aligned}
&\pa_t \Phi +  a(\Phi) \cdot \na_x \Phi =  F(t,x,\Phi), \\
&\Phi(0,x,v) = v.
     \end{aligned}
     \right.
\end{equation}
The interest in introducing the vector field  $\Phi$ comes from the following algebraic identity.
\begin{lem}
\label{lem-phi-alg}
Given a smooth function $g$ satisfying $\Tc g = R$, 
the function $$G(t,x,v) := g(t,x,\Phi(t,x,v))$$ solves the equation
\begin{equation}
\label{GG}
\pa_t G + a(\Phi(t,x,v)) \cdot \na_x G= R (t,x, \Phi(t,x,v)).
\end{equation}
\end{lem}

\begin{proof}[Proof of Lemma~\ref{lem-phi-alg}]We compute
\begin{align*}
\pa_t G &= (\pa_t g) (t,x,\Phi(t,x,v)) + \pa_t \Phi \cdot (\na_v g)(t,x,\Phi(t,x,v)), \\
a(\Phi) \cdot \na_x G &= a(\Phi) \cdot  (\na_x g)(t,x,\Phi(t,x,v)) + [a(\Phi) \cdot \na_x \Phi] \cdot (\na_v g)(t,x,\Phi(t,x,v))
\end{align*}
Since $\Tc g =R$, we have
\begin{multline*}
 (\pa_t g) (t,x,\Phi(t,x,v)) + a(\Phi) \cdot  (\na_x g)(t,x,\Phi(t,x,v)) \\  = -F(t,x,\Phi)\cdot (\na_v g)(t,x, \Phi(t,x,v))  +R (t,x, \Phi(t,x,v)).
\end{multline*}
From~\eqref{Burgers}, we deduce~\eqref{GG} as claimed.
\end{proof}
In other words, the change of variables in velocity $v \mapsto \Phi(t,x,v)$
allows to straighten the vector field  $\Tc$.

We  now provide a lemma concerning the existence, uniqueness and regularity of solutions of~\eqref{Burgers}. 

\begin{lem}
\label{lem-phi}
Assume that for all $j=1,\cdots, \ell$, $F^j \in L^2(0,T'; H^n_x)$ with norm bounded by $\Lambda(T',M)$.
There is $T \in (0,T']$ such that the following holds.
There exists a unique solution $\Phi(t,x,v) \in C^0(0,T; W^{k,\infty}_{x,v})$ of \eqref{Burgers} and we have
the following estimates:
\begin{align}
 \label{estim-Burgers-infini}
  \sup_{[0,T]} \sup_v  \sum_{|\a| \leq n} \| \pa^\a_{x,v} (\Phi  - v)\|_{L^2_{x,v}}  +    \sup_{[0,T]}   \| \Phi  - v\|_{W^{k,\infty}_{x,v}}  \lesssim   \Lambda(T,M),\\
 \label{estim-Burgers-infini-bis}
  \sup_{[0,T]} \sup_v  \sum_{|\a| \leq n} \| \pa^\a_{x,v} (a(\Phi)  - a(v))\|_{L^2_{x,v}}  +   \sup_{[0,T]}  \|a(\Phi)  - a(v)\|_{W^{k,\infty}_{x,v}} \lesssim    \Lambda(T,M),
\end{align}
for all $k< n - d/2$. 
\end{lem}

We shall not provide the proof of Lemma~\ref{lem-phi} as it follows closely the proof of Lemma 4.6 in \cite{HKR}. Here the source is semi-linear whereas it is linear in Lemma 4.6 of \cite{HKR}; however, the proof is essentially the same (see also \cite{HKNR} for a similar issue).

We now introduce the characteristics associated to $\Phi$, defined as the solution to    
   \begin{equation}
    \label{characteristic}
    \left\{
\begin{aligned}
     &\partial_{t} \mathrm{X}(t,s,x,v)=  a(\Phi)(t,\mathrm{X}(t,s,x,v), {v}), \\ 
     &\mathrm{X}(s,s,x,v)= x,
     \end{aligned}
     \right.
    \end{equation}
    and study the deviation of $\mathrm{X}$ from the (modified) free transport flow\footnote{Note that the $\mathrm{X}$ introduced here is not the same as the $X$ previously defined in~\eqref{charac}.}. 
    
          \begin{lemma}
   \label{straight}
   Assume that for all $j=1,\cdots, \ell$, $F^j \in L^\infty(0,T'; H^n_x)$ with norm bounded by $\Lambda(T',M)$.
There is  $T \in (0,T']$ such that the following holds.
   For every $0 \leq s , t \leq T$, we can write
    \begin{equation}
    \label{estimXtilde}
     \mathrm{X}(t,s,x,v)=  x  + (t-s) \left( a(v) +   \tilde{\mathrm{X}}(t,s,x,v)  \right)
     \end{equation}
     with $\tilde{\mathrm{X}}$ that satisfies the estimate
      \begin{equation}
      \label{LinftyXtilde}
 \begin{aligned}
   \sup_{t,s \in [0,T]}\sup_v \sum_{|\a| \leq n} \| \pa^\a_{x,v}\tilde{\mathrm{X}}(t,s,x,v) \|_{L^2_{x}} +     \sup_{t,s \in [0,T]} \| \tilde{\mathrm{X}}(t,s,x,v) \|_ {W^{k,\infty}_{x, v}}    \lesssim \Lambda(T,M),
   \end{aligned}
   \end{equation}
   for all  $k< n - d/2$.
       Moreover, 
the map $x\mapsto x+ (t-s) \tilde{\mathrm{X}}(t,s,x,v)$ is a diffeomorphism,  
and
there exists  $\Psi(t,s, x,  v)$ such that the identity
   $$ \mathrm{X}(t,s, x , \Psi(t,s,x, v)) = x +  (t-s) a(v)$$
   holds. Finally, we have the  estimate
          \begin{equation}
 \begin{aligned}
     \label{estimPsi}
     \sup_{t,s \in [0,T]} \left[  \sup_v\sum_{|\a| \leq n}\|   \pa^\a_{x,v}(\Psi(t,s,x,v)  - v )\|_{L^2_{x}} +
\|   \Psi(t,s,x,v)  - v \|_ {W^{k,\infty}_{x, v}}\right]
        \lesssim  \Lambda(T,M),
  \end{aligned}
  \end{equation}
for all  $k< n - d/2$.
     
   \end{lemma}

Again, we will not reproduce the proof of Lemma~\ref{straight} as it follows closely that of Lemma 5.1 in \cite{HKR}. 

\bigskip

 In what follows, the procedure will consist in applying derivatives on the equation~\eqref{abstract} using multiple combinations of the operators $L^{I,J}$ that were introduced and studied in the previous section. This yields systems of Vlasov equations with sources, such as~\eqref{even-eq-general}  in Lemma~\ref{even'}. This is on these equations that we will apply the change of variables $v \mapsto \Phi(t,x,v)$ in order to straighten the transport operator $\Tc$. We then integrate along characteristics, which is why the $\mathrm{X}(t,s,x,v)$ are useful, and average in velocity to obtain equations bearing on moments. In these equations, it will be crucial to apply various changes of variables based on the $\tilde{\mathrm{X}}$ and $\Psi$ introduced in Lemma~\ref{straight}.
 
 This is what we refer to as the \emph{semi-lagrangian} approach.

\section{Averaging operators}
\label{sec-K}

For $i\in \{1,\cdots, d\}$ and a smooth function $U(t,s,x,v)$,  we define  the following integral operator $ K^{(i)}_{U}$ acting on scalar functions $H(t,x)$:
 \begin{equation}
  K^{(i)}_{U}(H) (t,x) =    \int_{0}^t \int_{\R^d} (\pa_{x_i}  H) (s,  x - (t-s) a(v)) 
      U(t,s,x,v)\, dv ds.
      \end{equation} 
The integral operator $K$  can be seen as a modified version of the operator $\mathsf{K}^{(i)}_{\mathsf{U}}$
 $$ \mathsf{K}^{(i)}_{\mathsf{U}}(H) (t,x) =    \int_{0}^t \int_{\R^3} (\pa_{x_i}  H) (s,  x - (t-s) {v}) 
      \mathsf{U}(t,s,x,v)\, dv ds$$
      that was studied in \cite{HKR}. 
\subsection{The smoothing estimate}
We note that the operators $ {K}^{(i)}_U$ and $\mathsf{K}^{(i)}_{\mathsf{U}}$seem to feature a loss of derivative in $x$.
It was nevertheless proved in  \cite[Proposition 5.1 and Remark 5.1]{HKR} that for the operators $ \mathsf{K}^{(i)}_{\mathsf{U}}$, this loss is only apparent, provided that $U$ is sufficiently smooth in $x,v$ and decaying in $v$: this is the content of the following Theorem.
\begin{theorem}[\cite{HKR}]
\label{prop-K}
Let $k>1+d$ and $\sigma>d/2$. For all $H \in L^2(0, T; L^2_{x})$, and  for all $i\in \{1,\cdots, d\}$,  we have
 \begin{equation}
 \label{eq-K-droit}
 \| \mathsf{K}^{(i)}_{\mathsf{U}}(H)\|_{L^2(0, T; L^2_{x})} \lesssim \sup_{0 \leq s,\, t \leq T} \| \mathsf{U} (t,s,\cdot) \|_{\Hc^k_{\sigma}} \|H\|_{L^2(0, T; L^2_{x})}.
 \end{equation}
  \end{theorem}

Based on this result, we deduce the following smoothing estimate\footnote{A close version of this result is also stated in \cite{HKNR} for the special case $a(v) =\hat{v}$.} for the operators  $  K^{(i)}_{U}$.

\begin{proposition}
\label{proposition-K}
Let $k>1+d$ and $\sigma>d/2$. For all $H \in L^2(0, T; L^2_{x})$, and  for all $i\in \{1,\cdots, d\}$, we have
 \begin{equation}
 \label{eq-K}
 \| K^{(i)}_{U}(H)\|_{L^2(0, T; L^2_{x})} \lesssim \sup_{0 \leq s,\, t \leq T} \| U(t,s,\cdot) \|_{\Hc^k_{r_k}} \|H\|_{L^2(0, T; L^2_{x})},
 \end{equation}
     with $r_k ={\sigma + (1+\lambda) (d+k)}$. 
\end{proposition}

\begin{proof}[Proof of Proposition~\ref{proposition-K}]
To ease readability we set $\pa_x = \pa_{x_i}$ and we forget about the subscript $i$.
We come down from the modified to the straight operator by using the change of variables $w:= a(v)$. We get
\begin{align*}
K_{U}&(H) (t,x) \\
&= \int_{0}^t \int_{a(\R^d)} (\pa_{x}  H) (s,  x - (t-s) w) 
      U(t,s,x,a^{-1}(w)) |\det D a (a^{-1}(w))|^{-1} \, dw ds \\
      &=  \mathsf{K}_{ \mathsf{U} } (H)(t,x).
\end{align*}
with 
$$\mathsf{U}(t,s,x,w) := U(t,s,x,  a^{-1}(w))  |\det D a (a^{-1}(w))|^{-1} {{1}_{a(\R^d)}}.$$
Let $k>1+d$ and $\sigma>d/2$. By Theorem~\ref{prop-K}, we get
$$
 \| K_{U}(H)\|_{L^2([0, T]; L^2_{x})} =  \|  \mathsf{K}_{ \mathsf{U} } (H)\|_{L^2([0, T]; L^2_{x})}
 \lesssim \sup_{0 \leq s,\, t \leq T} \|  \mathsf{U}(t,s,\cdot) \|_{\Hc^k_{\sigma}} \|H\|_{L^2([0, T]; L^2_{x})}.
$$
By assumption on $a$, we have
$$|\pa^\alpha_w a^{-1}(w) | \lesssim (1+ |a^{-1} (w)| )^{1+ \lambda |\alpha|}.$$ 
In particular, we deduce 
$$
|\det D a (a^{-1}(w))|^{-1} \lesssim (1+ |a^{-1} (w)| )^{d(1+\lambda)}
$$
As a consequence, we have, by the Fa\`a di Bruno formula, and using the reverse change of variables $v =a^{-1}(w)$ and~\eqref{deriv-a}, that
$$
\|  \mathsf{U}(t,s,\cdot) \|_{\Hc^k_{\sigma}} \lesssim \| {U}(t,s,\cdot) \|_{\Hc^k_{\sigma + (d+k)(1+\lambda)}},
$$
hence the claimed estimate.
\end{proof}

\subsection{Intermission: a comparison to averaging lemmas}

We end this  section with a comparison of the smoothing estimate we have just exposed, say in the simple case where $a(v)= v$, which corresponds to Theorem~\ref{prop-K}, with kinetic averaging lemmas. Averaging lemmas were introduced in  \cite{GPS,Ago,GLPS} and now generically stand for various smoothing effects in average for kinetic transport type equations\footnote{Actually this can be embedded in a more general framework, see in particular \cite{Gerard,GG,GGW}.}. They proved over the years to be fundamental in several contexts of kinetic theory, as they provide compactness and regularity.
 There exist many versions of these, involving several different assumptions on the functional spaces, on the number of derivatives in $v$ or in $x$ in the source etc., see e.g. \cite{DPLM,PS,GSR,Bou,JV} and  \cite{Jabin,ASR,AM} for more recent advances.
The closest (to Theorem~\eqref{prop-K}) avatar of averaging lemmas is arguably the following  result of Perthame and Souganidis \cite{PS}.

\begin{theorem}[Perthame and Souganidis \cite{PS}]
\label{thm-PS}
Let $1<p<+\infty$. Let $f,g=(g_j)_{j =1,\cdots,d} \in L^p_{t,x,v}$ satisfy the transport equation
\begin{equation}
\label{eq-av}
\partial_t f + v \cdot \na_x f = \sum_{j=1}^d \pa_{x_j} \pa^k_v g_j,
\end{equation}
where $k$ is an arbitrary multi-index. Let $\varphi(v)$ be a $C^\infty$ compactly supported function and set
$$
\rho_\varphi (t,x) = \int_{\R^d} f (t,x,v) \varphi(v) \,dv. 
$$
Then we have, for all $\alpha \in [0, \min(1/p,1/p'))$,
\begin{equation}
\label{eq-rho}
\| \rho_\varphi\|_{L^p_{t,x}} \leq C_{d,p,\alpha,\varphi}  \| f \|_{L^p_{t,x,v}}^{1- \frac{\a}{|k|+1}} \| g \|_{L^p_{t,x,v}}^{\frac{\a}{|k|+1}}.
\end{equation}
\end{theorem}
Let us focus especially on the case $p=2$, $|k|=0$ (in which case~\eqref{eq-rho} actually also holds for $\alpha=1/2$).
Theorem~\ref{prop-K}  can also be understood as a kind of averaging lemma for the moments in $v$ of the kinetic equation~\eqref{eq-av}, in the special case where the source has the form
\begin{equation}
\label{eq-str}
 \sum_{j=1}^d \pa_{x_j} H_j(t,x)  \pa^k_v \mathcal{U}_j(t,x,v),
\end{equation}
where $\mathcal{U}_j$ is smooth in $x$ and $v$, and the initial condition is $f|_{t=0}=0$. Let $\varphi(t,x,v)$ be a smooth and decaying test function. Then by the method of characteristics, 
$$
f(t,x,v)  = \int_0^t \sum_{j=1}^d \pa_{x_j}  H_j(s,x-(t-s)v) \pa^k_v \mathcal{U}_j(s,x-(t-s)v,v) \, ds ,
$$
and thus 
$$
\begin{aligned}
\rho_\varphi(t,x) &= \int_0^t \int_{\R^d} \sum_{j=1}^d \pa_{x_j}  H_j(s,x-(t-s)v) \pa^k_v  \mathcal{U}_j(s,x-(t-s)v,v) \varphi(t,x,v) \, ds \\
&= \sum_{j=1}^d  \mathsf{K}^{(j)}_{U_j}(H_j)(t,x),
\end{aligned}
$$
setting $U_j(s,t,x,v) = \pa^k_v \mathcal{U}_j(s,x-(t-s)v,v)  \varphi(t,x,v)$. The regularity assumption of Theorem~\ref{prop-K} reads
$$
 \sup_{0 \leq s,\, t \leq T} \| U_j(t,s,\cdot) \|_{\Hc^k_{\sigma}} < +\infty
$$
for $k>1+d$, $\sigma>d/2$, and the consequence is 
\begin{equation}
\label{eq-rho-new}
\|\rho_\psi\|_{L^2(0,t; L^2_x)} \lesssim  \sup_{0 \leq s,\, t \leq T} \sum_{j} \| U_j(t,s,\cdot) \|_{\Hc^k_{\sigma}} \| H_j \|_{L^2(0,t; L^2_x)} .
\end{equation}
This estimate is not a consequence of Theorem~\ref{thm-PS}.
Indeed, note that it does not involve the $L^2$ norm of the solution $f$: somehow, this can be roughly seen as a version of Theorem~\ref{thm-PS} 
allowing $\a=1$ (whereas Theorem~\ref{thm-PS}  only allows $\a\leq 1/2$), at the expense of asking for the structure assumption~\eqref{eq-str} on the source $g$ and of considering a norm for the source that is more demanding than the $L^2$ norm of estimate~\eqref{eq-rho}.

Observe also that Theorem~\ref{prop-K} does not require the test function in $v$ to be decaying at infinity, as long as for all $j$,  $\mathcal{U}_j$ in~\eqref{eq-str} is itself decaying sufficiently fast at infinity.

\section{Proof of Theorem~\ref{thm} and of Corollaries~\ref{coro} and~\ref{cor}}
\label{sec-end}

{
We finally set up an induction argument, that relies on the machinery developed in the previous sections, and will ultimately lead to the proof of Theorem~\ref{thm}.
In order to summarize the procedure in a few words:
\begin{itemize}

\item By induction, we assume smoothness on the moments until order $n'-1$.
We can first apply Lemma~\ref{lem-Lij} to obtain the same smoothness for the coefficients of the operators $L_{i,j}$.

\item We apply Lemma~\ref{even'} or~\ref{odd'} in order to get the system of equations satisfied by $(L^{K,L}\pa^\alpha_x \pa^\beta_v f)$, which is of the abstract form
$$\Tc(\mathfrak{F}) + \mathfrak{A} \mathfrak{F}  = \mathcal{B},
$$
where $\mathfrak{A}$ is a matrix whose coefficients we control and $\mathcal{B}$ is the rest we need to control.
Loosely speaking, $\mathcal{B}$ is made either of remainders we can control thanks to the induction assumption, and terms of the form $- \pa^{\alpha(K,L) }_x F \cdot \na_v f$, for $K,L \in \{1,\cdots, d\}^{m+k}$, whose contribution is the main matter.

\item We then invert the operator $\Tc + \mathfrak{A}$ in order to solve the equation. At this stage, after integration in velocity (remind that we are interested in the regularity of moments), we use the changes of variables introduced in Lemmas~\ref{lem-phi-alg}, \ref{lem-phi} and~\ref{straight}.

\item  What is rather straightforward then is the study of the contribution of the initial data and of the remainder terms in $\mathcal{B}$. As already said, the contribution of the terms $- \pa^{\alpha(K,L) }_x F \cdot \na_v f$ is more serious and involves the study of integrals of the form
$$
 \int_{0}^t \int_{\R^d} (\pa_{x_i}  \pa^{\alpha(K,L) }_xF) (s,  x - (t-s) a(v)) 
      U(t,s,x,v)\, dv ds,
$$
which seem to feature a loss of derivative in $x$.
We recognize the integral operators introduced and studied in Section~\ref{sec-K}. This is where the smoothing estimate 
of Proposition~\ref{proposition-K} proves to be crucial.
\end{itemize}

}
\subsection{End of the proof of Theorem~\ref{thm}}


For $ \mathbf{n} \geq 2m-1$, let $\mathcal{P}(\mathbf{n})$ be the following statement:

\bigskip
\fbox{
\begin{minipage}[c]{0.9\textwidth}
 $\qquad \mathbf{\mathcal{P}(n)}:$ There is $T>0$ such that
for all test functions $$\psi(t,x,v) \in L^\infty(0,T; \mathcal{W}^{d+2+n-2m,\infty}_{-r_0}),$$ setting 
for all $|\a|=n$,
$$m_{\psi,\alpha} (t,x) = \int_{\R^d} \pa^\alpha_x f(t,x,v) \psi(t,x,v) \, dv,$$ 
there exists $\Lambda$ for which
\begin{equation}
\label{Pn-key}
\sum_{|\alpha|=n} \| m_{\psi,\alpha}\|_{L^2(0,T; L^2_x)}\lesssim \Lambda(T,M).
\end{equation}
\end{minipage}
}

\bigskip

By Proposition~\ref{lem-local}, it is clear that $\mathcal{P}(2m-1)$ is verified.

Let $\mathbf{n} \in \{2m, \cdots, 2(m+p)\}$. 
Let us  assume that $\mathbf{n}$ is even, of the form $2(m+k)$. We shall not proceed with the case where $\mathbf{n}$ is odd, as it follows by completely similar arguments. 
Assume that $\mathcal{P}(2m), \cdots, \mathcal{P}(\mathbf{n}-1)$ are satisfied and let $T>0$ be a time on which the estimates~\eqref{Pn-key} (for $2m, \cdots, \mathbf{n}-1$) are satisfied. We shall prove that $\mathcal{P}(\mathbf{n})$ is also verified. Once this will be done, we deduce by induction that $\mathcal{P}(2m), \cdots, \mathcal{P}(2(m+p))$ are satisfied; we then deduce the required estimates~\eqref{apriori}.

\bigskip

Thanks to the property  $\mathcal{P}(\mathbf{n}-1)$ applied to the $(\psi_j)_{j=1,\cdots,r}$, and~\eqref{Fj}, we first have
\begin{equation}
\label{Fjnew}
\sum_{j=1}^\ell \| F^j \|_{L^2(0,T; H^{2(m+k)-1}_x)} \leq \Lambda(T,M).
\end{equation}
We can therefore apply Lemma~\ref{lem-Lij} and obtain a possible smaller time still denoted by $T$ and operators $L_{i,j}$ with coefficients
$(\varphi_{k,l}^{i,j}, \psi_{k,l}^{i,j})_{i,j,k,l \in \{1,\cdots,d\}}$ 
belonging to $L^\infty(0,T; \Hc^{2(m+k)-2}_{-\r})$ for all $\r>d/2$, with uniform regularity 
$$
\|(\varphi_{k,l}^{i,j}, \psi_{k,l}^{i,j})_{i,j,k,l}\|_{L^\infty(0,T; \Hc^{2(m+k)-2}_{-\r})}\leq \Lambda(T,M).
$$

Let us  consider the vector (the precise ordering does not matter)
\begin{equation}
\label{def-F}\mathfrak{F}= 
\Big( L^{K,L} \pa_x^\alpha \pa_v^\beta f\Big)_{r\in\{m-k,\cdots, m+k\}, K,L \in \{1,\cdots,d\}^{r}, |\a| + |\b| =m+k-r} 
\end{equation}
By Lemma~\ref{even'}, it follows that $\mathfrak{F}$ satisfies the system
\begin{equation}
\label{eq-crucru}
\Tc(\mathfrak{F}) + \mathfrak{A} \mathfrak{F}  = \mathcal{B}  + \mathfrak{R},
\end{equation}
where $\mathfrak{A}(t,x,v)$ is a matrix with coefficients in $L^2(0,T; W^{d+2, \infty}_{x,v})$, satisfying
\begin{equation}
\label{estim-A}
\| \mathfrak{A}\|_{L^2(0,T; W^{d+2, \infty}_{x,v})} \lesssim \Lambda(T,M).
\end{equation}
 (The term $\mathfrak{A} \mathfrak{F} $  encodes the contribution of the leading order terms in the triple sum of the right-hand side of~\eqref{even-eq-general}.)
 On the other hand, $\mathfrak{R}$ is a remainder satisfying the estimate
\begin{equation}
\label{rest-fin}
\| \mathfrak{R}\|_{L^2(0,T; \Hc^0_{\tilde{r}})} \lesssim \Lambda(T,M)
\end{equation}
for all $\tilde{r}<r-d/2$
and $\mathcal{B}$ is  defined as follows: all its components  are equal to $0$ except those corresponding to the components associated to some $K,L \in  \{1,\cdots,d\}^{m+k}$, in which case it is equal to
$$
- \pa^{\alpha(K,L) }_x F \cdot \na_v f.
$$
The next step consists in using the change of variables $v \mapsto \Phi(t,x,v)$, where $\Phi$ solves~\eqref{Burgers}, in order to straighten the vector field $\Tc$, see Lemma~\ref{lem-phi-alg}. To this end, we use Lemma~\ref{lem-phi} (reduce again $T>0$ if necessary) and use the notation $\circ \Phi $ to denote the composition in $v$ with $\Phi$. Setting $\F= \mathfrak{F} \circ \Phi$, we obtain
\begin{equation}
(\pa_t + a(\Phi) \cdot \na_x )\F + (\mathfrak{A} \circ \Phi) \F = \mathcal{B}\circ \Phi  + \mathfrak{R}\circ \Phi.
\end{equation}
Let $\mathcal{A}(s,t,x,v)$ be the operator, 
whose existence is ensured by the Cauchy-Lipschitz theorem, as the solution of the following \emph{linear} ODE
$$
\pa_s \mathcal{A}(s,t,x,v) = \mathfrak{A}(s,x,\Phi(s,x,v)) A(s,t,x,v), \qquad A(t,t,x,v)= Id.
$$
Thanks to~\eqref{estim-A}, we also have
\begin{equation}
\label{estim-AA}
\| \mathcal{A}(\cdot,t,\cdot) \|_{L^\infty(0,T; W^{d+2, \infty}_{x,v})} + \| \pa_s \mathcal{A}(\cdot,t,\cdot) \|_{L^2(0,T; W^{d+2, \infty}_{x,v})}\lesssim \Lambda(T,M).
\end{equation}
By the method of characteristics we get
\begin{equation}
\begin{aligned}
\label{char-eq}
\F(t,x,v) &= \mathcal{A}(t,0,x,v) \F(0,\mathrm{X}(0,t,x,v),v) \\
&+\int_0^t  \mathcal{A}(t,s,x,v) \mathcal{B}\circ \Phi (s,\mathrm{X}(s,t,x,v),v) \,ds  \\
&+\int_0^t  \mathcal{A}(t,s,x,v) \mathfrak{R}\circ \Phi (s,\mathrm{X}(s,t,x,v),v) \,ds. 
\end{aligned}
\end{equation}
Let $\psi(t,x,v) \in L^\infty(0,T; \mathcal{W}^{d+2+2k,\infty}_{-r_0})$.  We multiply the representation formula~\eqref{char-eq} by $\psi(t,x,\Phi(t,x,v))|\det D_v \Phi(t,x,v)|$ and integrate in $v$ to obtain
\begin{equation}
\int_{\R^d} \F(t,x,v) \psi(t,x,\Phi(t,x,v)) |\det D_v \Phi(t,x,v)| \, dv  = I_0 + I_1 + I_2 
\end{equation}
with
\begin{equation}
\begin{aligned}
I_0 &= \int_{\R^d}  \mathcal{A}(t,0,x,v) \F(0,\mathrm{X}(0,t,x,v),v) \psi \circ \Phi |\det D_v \Phi(t,x,v)| \, dv,\\
I_1&=\int_0^t \int_{\R^d} \mathcal{A}(t,s,x,v)(\mathfrak{R}\circ \Phi) (s,\mathrm{X}(s,t,x,v),v) \psi \circ \Phi |\det D_v \Phi(t,x,v)| \, dv ds, \\
I_2&=\int_0^t \int_{\R^d} \mathcal{A}(t,s,x,v) (\mathcal{B}\circ \Phi) (s,\mathrm{X}(s,t,x,v),v)  \psi \circ \Phi|\det D_v \Phi(t,x,v)|\,  dv ds.
\end{aligned}
\end{equation}
By the change of variables $v \mapsto \Phi(t,x,v)$, we have
$$
\int_{\R^d} \F(t,x,v) \psi(t,x,\Phi(t,x,v)) |\det D_v \Phi(t,x,v)| \, dv = \int_{\R^d} \mathfrak{F} (t,x,v) \psi(t,x,v) \, dv.
$$
Let us first study  this term.
Since $\mathcal{P}(2m), \cdots, \mathcal{P}(2(m+k)-1)$ are satisfied, we can apply Lemma~\ref{lem-momentsLij} (the assumption~\eqref{assump-moments} is indeed verified), which yields, see~\eqref{eq-mom1} and~\eqref{eq-mom2}, that for all $I,J \in \{1,\cdots,d\}^{m+k}$,
\begin{equation*}
\int_{\R^d}  L^{I,J} f \, \psi(t,x,v) \,dv =  \int_{\R^d}\pa^{\alpha(I,J)}_x f \, \psi(t,x,v) \,dv + \mathfrak{R}_{I,J,\psi},
\end{equation*}
where $\mathfrak{R}_{I,J,\psi}$ is a remainder satisfying the estimate
\begin{equation*}
\| \mathfrak{R}_{I,J,\psi}\|_{L^2(0,T;L^2_x)} \leq \Lambda(T,M).
\end{equation*}
 Consequently, recalling the definition of $\mathfrak{F}$ in~\eqref{def-F}, if we are able to obtain the bound 
$$\|I_0 \|_{L^2(0,T; L^2_x)} + \| I_1\|_{L^2(0,T; L^2_x)} +  \|I_2\|_{L^2(0,T; L^2_x)} \leq \Lambda(T,M),$$ 
then we deduce the bound 
$$\sum_{I,J} \left\| \int_{\R^d} \pa_{x}^{\alpha(I,J)} f \psi  \, dv \right\|_{L^2(0,T;L^2_x)} \leq \Lambda(T,M),$$ 
that is we obtain the sought bound~\eqref{Pn-key} at rank $n$.

\subsubsection{Study of $I_0$} Let us begin by treating the contribution of the initial data, that corresponds to the term $I_0$. First by using estimate~\eqref{estim-Burgers-infini} in Lemma~\ref{lem-phi}, the $L^\infty$ bound for $\mathcal{A}$ in~\eqref{estim-AA}, 
and the estimate
\begin{equation}
\label{psicool}
\| (1+|v|^2)^{-r_0/2} \psi \|_{L^\infty_{x,v}} \lesssim 1, 
\end{equation}
we have for all $x \in \mathbb{T}^d$, 
    \begin{align*}
      \left|\int_{\R^d} \mathcal{A}(t,0,x,v)\F( 0,\mathrm{\mathrm{X}}(0,t,x,v),v) (1+|v|^2)^{r_0/2} |\det D_v \Phi(t,x,v)| \, dv \right|  \\
      \leq \Lambda(T, M)  \int |\F (0,\mathrm{X}(0,t,x,v),v) |  (1+|v|^2)^{r_0/2} \, dv.
       \end{align*} 
     Therefore, we get that
     $$ \| {I}_{0} \|_{L^2(0,T; L^2_{x})} \leq \Lambda(T, M)  \left\|\int_{\R^d} \|\F (0,\mathrm{X}(0,t,\cdot,v),v) \|_{L^2_{x}}  (1+|v|^2)^{r_0/2}\, dv \right\|_{L^2(0, T)}.$$
      By using the change of variable  $y= \mathrm{X}(0, t, x, v) + ta(v)=  x - t \tilde{\mathrm{X}}(0, t, x, v)$ and  Lemma \ref{straight}, we obtain that 
      $$  \|\F (0,\mathrm{X}(0,t,\cdot,v),v) \|_{L^2_{x}} \leq \Lambda(T, M) \|\F(0,\cdot - ta(v), v)\|_{L^2_x} \leq \Lambda(T, M) \| \F(0,\cdot, v) \|_{L^2_{x}}$$
       and hence, we deduce that since $r> r_0 + d/2$,  for some $r'>d/2$, there holds
       $$  \| {I}_{0} \|_{L^2(0,T; L^2_{x})} \leq \Lambda(T, M) \left(\int_{\mathbb{R}^d}    \frac{dv}{(1+|v|^2)^{r'}}\right)^{\frac{1}{2}} \| \F(0) \|_{\Hc^0_{r}}.$$
        By using the fact that at $t=0$ we have $\Phi(0,x,v)=v$ and   $L^{(K,L)}|_{t=0}= \partial_{x}^{\alpha (K,L)}$ we end up with
         $$ \| \F(0)\|_{\Hc^0_{r}} =  \|\mathfrak{F}(0)  \|_{\Hc^0_{r}} \leq \Lambda(M) \sum_{j=m-k}^{m+k} \sum_{|\a| + |\b| =m+k-j } \| \pa_x^{2j} \pa_x^\alpha \pa_v^\beta f_0\|_{\Hc^{0}_{r}}$$ 
         and
          hence we finally obtain that
      $$  \| {I}_{0} \|_{L^2(0,T; L^2_{x})} \leq  \Lambda(T, M).$$
       
  \subsubsection{Study of $I_1$}      We treat the other remainder term $I_1$ in a similar fashion. Indeed, using again estimate~\eqref{estim-Burgers-infini} in Lemma~\ref{lem-phi}, \eqref{estim-AA} and~\eqref{psicool}, we first get 
\begin{multline*}
\| {I}_{1}\|_{L^2(0, T; L^2_{x})}\leq \Lambda(T, M) \\
    \times \left \| \int_{0}^t \int_{\mathbb{R}^d}  \|\mathfrak{R} (s, \mathrm{X}(s,t,\cdot,v), \Phi(s,\mathrm{X}(s,t,\cdot,v),v)) \|_{L^2_{x}} (1+|v|^2)^{r_0/2} \,dv ds\, \right\|_{L^2(0, T)}.
\end{multline*}
Thanks to the change of variable $x\mapsto  \mathrm{X}(s,t,x,v)$ and to the estimates of Lemma \ref{straight}, it follows that
\begin{align*}
\| {I}_{1}\|_{L^2(0, T; L^2_{x})} & \leq  \Lambda(T, M) \left\|  \int_{0}^t \int_{\mathbb{R}^d} \|\mathfrak{R}  (s, \cdot, \Phi(s,\cdot,v))\|_{L^2_{x}}  (1+|v|^2)^{r_0/2}   \,dv ds \right\|_{L^2(0, T)} \\
 &  \leq  \Lambda(T, M) \left\| \int_{0}^t \| (\mathfrak{R}\circ \Phi) (s)\|_{\Hc^0_{\r}}\, ds  \right\|_{L^2(0, T)} \\
 & \leq  \Lambda(T, M) \,T\,  \| \mathfrak{R}\circ \Phi \|_{L^2(0, T; \Hc^0_{\r})},
  \end{align*}
  by choosing $\r > r_0+d/2$, which is possible since $r> r_0 + d$.
Using again the change of variables $v\mapsto \Phi(t,x,v)$,  Lemma \ref{lem-phi} and the estimate~\eqref{rest-fin}, we thus obtain
             $$  \| {I}_{1} \|_{L^2(0,T; L^2_{x})} \leq  \Lambda(T, M).$$

\subsubsection{Study of $I_2$}  The main matter thus concerns the contribution of the term $I_2$, which  features an apparently loss of derivative in $x$. This is however not the case, thanks to Proposition~\ref{proposition-K}. Let $K,L \in \{1,\cdots,d\}^{m+k}$. Writing $\pa^{\alpha(K,L)}_x =
\pa_{x} \pa^{\alpha'}_x$ with $|\alpha'| = |\alpha(K,L)| -1$,
we are led to study terms of the form (here $F_i^j$ stands for the $i$-th coordinate of $F^j$):
\begin{align*}
\sum_{j=1}^\ell \int_0^t \int_{\R^d} (\pa_x \pa^{\alpha'}_x  F^j_i)(s, \mathrm{X}(s,t,x,v))  \psi(t,\mathrm{X}(s,t,x,v),\Phi(s,\mathrm{X}(s,t,x,v),v)) \\
\times  \mathcal{A}_{K,L}^{I,J}(t,s,,\mathrm{X}(s,t,x,v), \Phi(s,\mathrm{X}(s,t,x,v),v))  A_j(\Phi(s,\mathrm{X}(s,t,x,v),v))\\
\times \pa_{v_i} f(s, \mathrm{X}(s,t,x,v), \Phi(s,\mathrm{X}(s,t,x,v),v)  |\det D_v \Phi(t,x,v)| \, dv ds,
\end{align*}
where $\| \mathcal{A}_{K,L}^{I,J} \|_{L^\infty(0,T; W^{d+2, \infty}_{x,v})}\leq \Lambda(T,M)$.

We use the change of variables $ v= \Psi(s,t,x,w)$  to rewrite this expression as 
$ {\sum_{j=1}^\ell K_{U_j} (\pa^{\alpha'}_x F^j_i)} $,
with
\begin{equation}
\begin{aligned}
U_{j}(s,t,x,v)&= A_j \Big(\Phi(s,x-(t-s)a(v),\Psi(s,t,x,v))\Big) \\
&\times  \mathcal{A}_{K,L}^{I,J}\Big(t,s,x-(t-s)a(v), \Phi(s,x-(t-s)a(v),\Psi(s,t,x,v))\Big) \\
&\times  \psi\Big(t,x-(t-s)a(v), \Phi(s,x-(t-s)a(v),\Psi(s,t,x,v))\Big) \\  
&\times \pa_{v_i} f\Big(s, x-(t-s)a(v), \Phi(s,x-(t-s)a(v),\Psi(s,t,x,v))\Big) \\
&\times  |\det D_v \Phi(t,x,\Psi(s,t,x,v))|  |\det D_v \Psi(s,t,x,v)|,
\end{aligned}
\end{equation}
where we recall the operators $K$ were introduced in Section~\ref{sec-K}.
In order to apply Proposition~\ref{proposition-K},
we have to estimate, $s,t$ being fixed, $U_{j}$ in $\Hc^{2+d}_{r'}$, with $r'> d/2 + 2 (1+\lambda)(1+d)$ 
and $r\geq r' + r_0$ (which is possible since $r>R$ as defined in~\eqref{def-NR}).
First, by~\eqref{deriv-a}, \eqref{deriv-A}, \eqref{estim-AA}, \eqref{estim-Burgers-infini} in Lemma~\ref{lem-phi} and estimate~\eqref{estimPsi} in Lemma~\ref{straight}, we can uniformly bound in $L^\infty$ all terms involving $A_j$, $\Phi$, $\Psi$ and their derivatives (since only at most $2+d$ derivatives can be involved). For what concerns $\psi$, we use
$$
\| (1+|v|^2)^{-r_0/2} \pa^\a \psi \|_{L^\infty_{x,v}} \lesssim 1, \quad \forall |\a|\leq d+2.
$$
We are therefore led to estimate integrals under the form 
       $$\begin{aligned} &I= \\
        &\left|  \int_{\mathbb{T}^d \times \mathbb{R}^d} |g(x-(t-s)a(v),{ \Phi(s,x-(t-s)v,\Psi(s,t,x,v))} )|^2 (1+ |v|^2)^{r_0 + r'} \, dv dx \right|,
        \end{aligned}$$
       where $g = \pa^\alpha f$, $|\alpha| \leq d+3$.
     To this end, we can use the change of variables  $v \mapsto w=  \Psi(s,t,x,v)$ and rely on estimate~\eqref{LinftyXtilde} in Lemma~\ref{straight} to obtain the bound 
          $$ I \leq \Lambda(T, M)  \int_{\mathbb{T}^d \times \mathbb{R}^d} |g(\mathrm{X}(s,t,x,w), { \Phi(s,\mathrm{X}(s,t,x,w),w)}) |^2  (1 + |w|^2)^{r_0+r'}\, dx dw.$$
          Next, arguing as for $I_1$, we can use successively
         the change of variable $x \mapsto y= \mathrm{X}(s,t,x,w)$ with the estimates of Lemma \ref{straight},
   and the change of variable $w \mapsto u= \Phi(s,y, w)$ with estimate~\eqref{estim-Burgers-infini} in  Lemma~\ref{lem-phi}, 
to finally obtain
           $$   I \leq \Lambda(T, M)  \|g\|_{\Hc^0_{r}}^2 \leq \Lambda(T, M)  \|f\|_{\Hc^{2m-1}_{r}}^2,$$
since $2m-1 \geq d+3$ and $r> R$. As a result we obtain the bound
\begin{equation}
\sup_{s,t} \| U_j \|_{\Hc^{2+d}_{r'}} \leq \Lambda(T, M)  \|f\|_{\Hc^{2m-1}_{r}}^2 \leq  \Lambda(T, M).
\end{equation}
We can therefore apply Proposition~\ref{proposition-K} to get the bound 
\begin{equation}
\begin{aligned}
\| K_{U_{j,i}} (F^j_i) \|_{L^2(0,T;L^2_x)} &\lesssim \sup_{s,t} \| U_j \|_{\Hc^{2+d}_{r'}} \| F^j_i\|_{L^2(0,T; H^{2(m+k)-1}_x)}  \\
&\leq  \Lambda(T, M) \| F^j_i\|_{L^2(0,T; H^{2(m+k)-1}_x)} \\
&\leq  \Lambda(T, M),
\end{aligned}
\end{equation}
thanks to estimate~\eqref{Fjnew}. 
We deduce
             $$  \| {I}_{2} \|_{L^2(0,T; L^2_{x})} \leq  \Lambda(T, M)$$
and gathering all pieces together, we therefore obtain~\eqref{Pn-key} at rank $\mathbf{n}$, and the induction argument is complete. Theorem~\ref{thm} follows.

\subsection{Proof of Corollary~\ref{coro}}

In order to prove the higher order regularity for the characteristics, we proceed as in \cite[Lemma 5.1]{HKR}.

By Theorem~\ref{thm} and the assumption~\eqref{Fj}, we have for all $j =1,\cdots, \ell$,
 $$F^j \in L^2(0,T; H^{n'}_x)$$
 and thus by Sobolev embedding, we deduce that for $k < n' -d/2$,
  \begin{equation}
  \label{Linft}
F^j \in L^2(0,T; W^{k,\infty}_x).
\end{equation}
We set 
$$
Z:=(Y,W)  := (X-tv-x, V-v).
$$
Let us first prove that $Z \in L^\infty(0,T; W^{k,\infty}_{x,v})$ for  $k < n' -d/2$.
Note that by definition of $(X,V)$, $Z$ satisfies the equation
$$
Z= \left( \int_0^t (Y+v) \, ds, \, \int_0^t \sum_{j=1}^\ell A_j(W+v) F^j (Y+x+tv) \, ds\right).
$$
By~\eqref{deriv-A} and~\eqref{Linft}, we obtain by induction (on the number of applied derivatives) that for $t\leq T$,
$$
\sup_{|\alpha|< n'-d/2}  \sup_{[0,t]} \| \pa^\alpha_{x,v} Z\|_{L^\infty_{x,v}}
\lesssim \int_0^t \lambda(s) \left(1+\sup_{|\alpha|< n'-d/2}  \sup_{[0,s]} \| \pa^\alpha_{x,v} Z\|_{L^\infty_{x,v}}\right)\, ds,
$$
where $\lambda$ is a non-negative function belonging to $L^2(0,T)$, with norm bounded by $\Lambda(T,M)$. We deduce our claim thanks to the Gronwall inequality, which yields
\begin{equation}
\label{infi}
\sup_{|\alpha|< n'-d/2}  \sup_{[0,t]} \| \pa^\alpha_{x,v} Z\|_{L^\infty_{x,v}} \leq \sqrt{t} \Lambda(T,M).
\end{equation}
We  deduce in particular from this estimate that for $T' \in (0,T]$ small enough, for all $v \in \R^d$, the map $x \mapsto X(T',0,x,v)$ is a $C^1$ diffeomorphism.

Next, let us turn to the $L^\infty_t L^\infty_v L^2_x$ estimate. 
We set
$$
\mathcal{N}(t) := \sup_{|\alpha|\leq n'}  \sup_{[0,t]} \| \pa^\alpha_{x,v} Z\|_{L^\infty_{v}L^2_x}.
$$
By an application of the Fa\`a di Bruno formula, we obtain
\begin{align*}
&\mathcal{N}(t)  \lesssim 
 \sum_{j=1}^\ell  \int_0^t  \sum_{k_1, k_2, \beta_{1}, \cdots, \beta_{k_1+k_2}}
 J^j_{k_1,k_2, \beta_{1}, \cdots, \beta_{k_1+k_2}} \, ds,
\end{align*}
with 
\begin{multline*}
J^j_{k_1,k_2, \beta_{1}, \cdots, \beta_{k_1+k_2}}:= \\
 \Big\| | (D^{k_1}_{v} A_j) \circ V(s)  (D^{k_2}_{x} F^j) \circ X(s) | 
   |\partial^{\beta_{1}}_{x,v} (X,V)| \cdots | \partial^{\beta_{k_1+k_2}}_{x,v} (X,V) | \Big\|_{L^\infty_{v} L^2_{x}},
   \end{multline*}
   and where the sum is taken only on indices such that $k_1+k_2=:  k \leq | \alpha | \leq n'$, $ \beta_{1}+ \cdots+ \beta_{k}= | \alpha |$
             with for every $j$,  $|\beta_{j}| \geq 1$ and $|\beta_{1}| \leq | \beta_{2}| \leq \cdots\leq  | \beta_{k}|.$
             
             Let us observe that  in the sum, if $k_1+k_2 = k \geq 2$, we necessarily have $|\beta_{k-1}| < n' -{d / 2}$.
              Indeed, otherwise, we would have 
              $ |\beta_{1}| + \cdots + | \beta_{k} | \geq  2n' - d  $ and thus $n' \geq 2n'-d $, which means  $ n' \leq d$. This is impossible by assumption on $n'$.  Next, 
             \begin{itemize}
             \item if  $k_2< n' -{d / 2}$ and $k_1 + k_2=k  \geq 2$ we obtain thanks to the above observation and~\eqref{infi} that for $i=1,\cdots, k-1$,
             \begin{equation}
             \label{beta}
      \| \partial^{\beta_{i}}_{x,v} (X,V)\|_{L^\infty_{x,v}} \lesssim  1 + T + \|\partial^{\beta_{i}}_{x,v} (Z) \|_{L^\infty_{x,v}}
              \lesssim   \Lambda(T, M).
\end{equation}
             Moreover, using  \eqref{deriv-A}, \eqref{Linft}  we get
           \begin{align*}
          &  J^j_{k_1,k_2, \beta_{1}, \cdots, \beta_{k_1+k_2}}  \\&  \leq  \|D^{k_1} A_j \|_{L^\infty_{x,v}}
            \|D^{k_2} F^j \|_{L^\infty_{x,v}}\, \left\|  \prod_{i=1}^{k-1} \partial^{\beta_{i}}_{x,v} (X,V) \right\|_{L^\infty_{x,v}}  \| \partial^{\beta_{k}}_{x,v} (X,V)\|_{L^\infty_{v} L^2_{x}}\\
           &  \leq  \Lambda(T, M)    \|D^{k_2} F^j \|_{L^\infty_{x,v}} (1 +  \mathcal{N}(s)).
           \end{align*}
           If $k=1$, the above estimate is clearly also  valid.
             \item  if $k_2\geq n' -{d / 2}$, we observe that  for every $i$,   $| \beta_{i}| \leq  | \beta_{k}| \leq n' - (k-1) <  {d / 2}$. In particular
              $| \beta_{i}| <   n' - {d / 2 }$  by  assumption on $n'$ and we have this time that~\eqref{beta} holds for all $i=1,\cdots, k$.
        This yields
       \begin{align*}J^j_{k_1,k_2, \beta_{1}, \cdots, \beta_{k_1+k_2}}
       &  \lesssim \left\| (D^{k_1}_{v} A_j ) \circ V \right\|_{L^\infty_{v} L^2_{x}} 
       \left\| (D^{k_2}_{x} F^j ) \circ X \right\|_{L^\infty_{v} L^2_{x}} 
       \Lambda(T, M) \\
        &  \lesssim \left\| D^{k_1}_{x,v} A_j \right\|_{L^\infty_{v}} 
       \left\| (D^{k_2}_{x} F^j ) \circ X \right\|_{L^\infty_{v} L^2_{x}} 
       \Lambda(T, M)
        \\ &  \lesssim \Lambda(T, M)   \left\| (D^{k_2}_{x} F^j ) \right\|_{L^2_{x}}.
       \end{align*}
       To get the last estimate, we restrict to $T'\leq T$ small enough so that we can use the change of variable
      $ y= X(t,0,x,v)$ when computing the $L^2_{x}$ norm of $ (D^{k_2}_{x,v} F^j ) \circ X$.
             \end{itemize}
             By combining the above estimates, we obtain that for $t\leq T'$,
           $$ \mathcal{N}(t) \leq  \sqrt{t} \Lambda(T, M) + \int_{0}^t \Lambda(T, M) \sup_j \|F^j (s)\|_{H^{n'}_x}\mathcal{N}(s)\, ds.$$
            By using again~\eqref{Linft} and the Gronwall inequality, we thus obtain that for $t\leq T'$,
            $$    \mathcal{N}(t) \lesssim \sqrt{t} \Lambda(T,M),$$
which concludes the proof of Corollary~\ref{coro}.

\subsection{Proof of Corollary~\ref{cor}}

The idea, as in \cite[Proposition 5.2]{Gerard}, consists in applying Theorem~\ref{thm} with the test function 
$$
\psi_\eta (v) = e^{- v \cdot \eta} \in {W}^{n',\infty}_{x,v}, 
$$
where $ \eta \in \R^d$ has to be seen as the Fourier variable in velocity. A close inspection of the proofs reveals that the conclusion of Theorem~\ref{thm} can be refined into
\begin{equation}
\label{eq-proof-cor}
\forall \eta \in \R^d, \qquad \left \|\int f \psi_\eta \,dv  \right\|_{L^2(0,T_0; H^{n'}_x)} \leq \Lambda (T_0,M, \| \psi_\eta\|_{{W}^{n',\infty}_{v}}), 
\end{equation}
where $\Lambda$ is a polynomial function.
Moreover, $ \| \psi_\eta\|_{{W}^{n',\infty}_{v}} \lesssim \Lambda'(|\eta|)$, where $\Lambda'$ is also a polynomial function (of degree $n'$).
Since
$$
1/(2\pi)^{d/2} \int f \psi_\eta \,dv  = \mathcal{F}_v {f}(t,x, \eta),
$$
 we deduce from~\eqref{eq-proof-cor} that for some $p>0$ taken large enough,
$$
\left \|  \widehat{f}(t,k, \eta)  (1+ |k|^2)^{n'/2} (1+ |\eta|^2)^{-p/2} \right\|_{L^2(0,T_0; L^2(\Z^d \times \R^d))} < +\infty,
$$
which means  that $f \in L^2(0,T_0; H^{n',-p}_{x,v})$.

\section{Application to classical models from physics}
\label{sec-Vla}

The goal of this section is to briefly explain why both Vlasov-Poisson and relativistic Vlasov-Maxwell enter the abstract framework, and thus why Theorem~\ref{thm} (and its corollaries) apply to these classical models.

\subsection{Vlasov-Poisson}     
\label{appli1}

The Cauchy problem for the Vlasov-Poisson system~\eqref{VP} was studied (among many other references)
 \begin{itemize}
 \item for what concerns (global) weak solutions, in \cite{Ar},
\item  for what concerns local strong solutions in \cite{UO}, and \cite{BarDeg,LP,Pfa,SchVP,Glassey,BR,Hor} for global strong solutions.
\end{itemize}

Let us check the following structural assumptions for~\eqref{VP}.

\noindent $\bullet$ \emph{Assumptions on the advection field.}
In this model, $a(v) =v$, so that all required assumptions on $a$ are straightforward properties. One can take $\lambda=0$ in~\eqref{growth-a-1}. 

\noindent $\bullet$ \emph{Assumptions on the force field.}
For what concerns the force field $F$, we can write $\ell=1$, $A_1=1$ and $F^{1} = -\na_x \phi$, where $\phi$ is computed only thanks to the moment of order $0$ of $f$ only, that is $\psi_1 =1$ (thus $r_0=0$) and
$$
m_{\psi_1}= \int_{\R^d} f \, dv.
$$
The assumption~\eqref{Fj} follows straightforwardly from the Poisson equation, as for all $n \in \N$,
there holds
$$
\forall t \geq 0, \quad \| F^1 (t) \|_{H^n_x} \lesssim \| m_{\psi_1}(t) \|_{H^{n-1}_x}.
$$
We  however do not need the smoothing effect due to the Poisson equation. It follows directly that
both estimates \eqref{Fj} and \eqref{Fj-inft} hold. 
The stability estimate~\eqref{stab} holds because of the same estimate, by linearity of the Poisson equation. It turns out that using the smoothing estimate, we can obtain a stronger version of Theorem~\ref{thm}: we embed this situation in what we refer to as transport/elliptic systems, and refer to Theorem~\ref{thm2} in Section~\ref{sec-ell}.

\bigskip

Note also that the Vlasov-Poisson system with dynamics constrained on geodesics introduced in the context of stellar dynamics in \cite{DILS} enter the abstract framework as well (and in this model there is no smoothing of the force field).

\subsection{Relativistic Vlasov-Maxwell}     
\label{appli2}
The Cauchy problem for the relativistic Vlasov-Maxwell system~\eqref{VM} was studied (among many other references)
 \begin{itemize}
 \item for what concerns (global) weak solutions, in \cite{DPL},
\item  for what concerns (local) strong solutions in \cite{Wol1,Wol2,Deg,Asano,GS86,GlaStrCMP,Glassey,SchIUMJ,BGP,KS,P,LS}.\end{itemize}

Let us check  the following structural assumptions for~\eqref{VM}.
 
 \noindent $\bullet$ \emph{Assumptions on the advection field.} In this model, $a(v) =\hat{v}$, and one can  check by a straightforward induction that
 $$
\| \pa^\a_v \hat{v} \|_{L^\infty_v} \leq C_\a, \qquad \forall \a.
 $$
 We have $a(\R^d) = B(0, c)$ and there holds the explicit formula
 $$
 \forall w \in B(0, c), \quad a^{-1}(w) = \frac{w}{\sqrt{1-|w|^2/c^2}}.
 $$
 It follows that one can take $\lambda=2$ in~\eqref{growth-a-1}. 
 
 \noindent $\bullet$ \emph{Assumptions on the force field.}
 For what concerns the force field $F$, we observe that we can take $\ell=4$ and write
\begin{equation}
\label{dec-1}
A_1 =1, \quad F^1=E,
\end{equation}
and setting $B=(B_1,B_2,B_3)$ in an orthonormal basis $(e_1,e_2,e_3)$,
\begin{equation}
\label{dec-2}
\begin{aligned}
&A_2 = \hat{v}_1, \quad F^2=  B_2 e_3- B_3 e_1, \\
&A_3 = \hat{v}_2, \quad F^3=  B_3 e_1- B_1 e_3, \\
&A_4 = \hat{v}_3, \quad F^4=  B_1 e_2- B_2 e_1.
\end{aligned}
\end{equation}
The electromagnetic field $(E,B)$ is computed only from initial data $(E_0,B_0)$ and the moments of order $0$ and $1$, that correspond to $\psi_1=1, \psi_2 =\hat{v}$ (so that $r_0=0$) and
$$
m_{\psi_1}= \int_{\R^d} f \, dv, \quad m_{\psi_2}= \int_{\R^d} f \hat{v} \, dv
$$
The assumption~\eqref{Fj} follows from classical energy estimates for Maxwell equations: we have
for all $n \in \N$ and all $t \geq 0$, 
$$\| (E,B)\|_{L^2(0,t; H^n_x)} \leq C_{n} {t}^{3/2} \sum_{i=1}^2 \| m_{\psi_i} \|_{L^2(0,t;H^n_x)}
+  \|(E,B) (0) \|_{L^2(0,t; H^n_x)},
$$
see e.g. \cite[Lemma 3.2]{HKNR}. The  estimate~\eqref{Fj-inft} is proved similarly.
The stability estimate~\eqref{stab} holds because of the same energy estimate, by linearity of the Maxwell equations.


\subsection{Remarks}
Some remarks about possible generalizations of the abstract framework are in order.

\begin{itemize}

\item It is possible to add a smooth force, of $C^k$ regularity with $k$ large enough, and still  adapt the results of Theorem~\ref{thm}, without significantly modifying the analysis. This allows for instance to consider Vlasov-Poisson with a smooth external magnetic field.

\item The so-called relativistic gravitational Vlasov-Poisson system (which may be relevant for galactic dynamics) enters the abstract framework as well, by a combination of the estimates of Section~\ref{appli1} and~\ref{appli2} (see e.g. \cite{GlaStrRVP,HR,KTZ,LMR} for some references about this system).

\item The divergence-free (in $v$) condition for $F$ is not an absolute requirement for the analysis.
It may  be dropped up to introducing more complicated formulas. In particular, it is likely that fluid/kinetic systems for sprays such as Vlasov-Stokes or Vlasov-Navier-Stokes in dimension $d=2$ enter this framework (or a slightly modified version of if) as well. We refer e.g. to \cite{Jab,BDGM,Des,BGLM} for some references about these equations. See also \cite{BarDes,MouSue} for other fluid/kinetic systems.

\item Note that the so-called non-relativistic Vlasov-Maxwell system (that is System~\eqref{VM} with $v$ replacing all occurences of $\hat{v}$) does not enter the abstract framework. Indeed the assumption~\eqref{deriv-A} is not satisfied.
However we claim that~\eqref{deriv-A} is crucial only for having a good local well-posedness theory in $\Hc^n_r$ spaces. This means that without~\eqref{deriv-A}, we can still obtain a result similar to that of Theorem~\ref{thm}, except that we have to \emph{assume} the existence of a solution of~\eqref{abstract} with the required regularity. For what concerns the non-relativistic Vlasov-Maxwell system, such solutions do exist, following Asano \cite{Asano}, which requires the introduction of Sobolev spaces with loss of integrability in velocity.

\end{itemize}

     \section{The case of transport/elliptic type Vlasov equations}
     \label{sec-ell}

       \subsection{An improvement of Theorem~\ref{thm}}
      Let us assume in this section that the following strengthened version of~\eqref{Fj} is verified:
       \begin{equation}
\label{Fj-imp}
\begin{aligned}
&\| F^j \|_{L^2(0,t; H^n_x)} \\
& \leq \Gamma^{(j)}_{n} \left(t,  \| m_{\psi_1} \|_{L^2(0,t;H^{n-1}_x)}, \cdots, \| m_{\psi_r} \|_{L^2(0,t;H^{n-1}_x)}, \sum_{j=1}^{\ell} \|F^{j} (0) \|_{H^n_x} \right).
 \end{aligned}
\end{equation}
In other words, the force is smoothed out and gains one derivative compared to the distribution function. We refer to such a situation as the transport/elliptic type case.
This  includes in particular the Vlasov-Poisson system. We then have the following version of Theorem~\ref{thm}.
This is an improved version in the sense that the higher regularity we ask for is only regularity in $x$ and not at all in $v$ (compare~\eqref{eq-thm2} below to~\eqref{eq-thm} in Theorem~\ref{thm}). 
\begin{theorem}
\label{thm2}
Let $n\geq N$ and $r>R$.
Let $n'>n$ be an integer such that
$n> \lfloor \frac{n'}{2} \rfloor+{d} + 1$. Assume that $f_0 \in \Hc^{n}_{r}$ and $F^j(0) \in H^{n'}_x$ for all $j \in \{1,\cdots, \ell\}$.
Assume furthermore that the initial data $f_0$ satisfy the following higher space regularity:
\begin{equation}
\label{eq-thm2}
\begin{aligned}
\pa^\alpha_x f_0 \in \Hc^0_r, \qquad\forall |\alpha|= n'.
\end{aligned}
\end{equation}
Then there is $T>0$ such that the following holds. There exists a unique solution $(f(t),F(t))$ with initial data $(f_0, F(0))$ to~\eqref{abstract} such that $f(t) \in C(0,T; \Hc^{n}_r)$.

Moreover, for all test functions $\psi \in L^\infty(0,T; \mathcal{W}^{n',\infty}_{-r_0})$, we have
\begin{equation}
\int f \psi \,dv  \in L^2(0,T; H^{n'}_x)
\end{equation}

\end{theorem}
As in Corollary~\ref{cor}, we may deduce as well under the assumptions of Theorem~\ref{thm2} that 
\begin{equation}
f   \in L^2(0,T; H^{n',-\infty}_{x,v}).
\end{equation}    
    
    \begin{proof}[Proof of Theorem~\ref{thm2}]
  The beginning of the proof is the same as for Theorem~\ref{thm} (of which we keep the notations).       Let us set in this context
  \begin{equation}\label{defM2}
M := \| f_0 \|_{\Hc^{2m-1}_r} +\sum_{k=0}^{2p} \sum_{|\alpha| =2m+k} \| \pa^\alpha_x f_0 \|_{\Hc^0_r}+\sum_{j=1}^\ell \|F^j(0)\| _{H^{2(m+p)}_x},
\end{equation}
  We proceed with the same  induction argument, treating all terms similarly except for\footnote{We also remark that in order to treat the term $I_2$, we do not absolutely need to use Proposition~\ref{proposition-K}: we can indeed rely on the smoothing estimate~\eqref{Fjnewnew} on the force instead and argue as we did for $I_1$. This observation will be useful later in order to treat other Vlasov models.} what concerns the treatment of the term $I_0$ for which the following improvement in Section~\ref{sec-end}. The idea will be to use integration by parts in $v$ to trade derivatives in $v$ against derivatives in $x$, allowing to obtain estimates depending on~\eqref{defM2} (compared to~\eqref{defM} for Theorem~\ref{thm}).

 First note using the smoothing estimate~\eqref{Fj-imp} that we improve~\eqref{Fjnew} to
 \begin{equation}
\label{Fjnewnew}
\sum_{j=1}^\ell \| F^j \|_{L^2(0,T; H^{2(m+k)}_x)} \leq \Lambda(T,M).
\end{equation}
We can use this improved estimate with Remark~\ref{rem-coef} to deduce that the coefficients of $\mathfrak{A}$ (as appearing in~\eqref{eq-crucru}) satisfy the improved form of~\eqref{estim-A}
\begin{equation}
\label{estim-A'}
\| \mathfrak{A}\|_{L^2(0,T; W^{p, \infty}_{x,v})} \lesssim \Lambda(T,M), \qquad \forall p < 2m-d/2-1.
\end{equation}
Therefore we deduce the improved form of~\eqref{estim-AA}:
\begin{equation}
\label{estim-AA'}
\| \mathcal{A}(\cdot,t,\cdot) \|_{L^\infty(0,T; W^{p, \infty}_{x,v})} \lesssim \Lambda(T,M),  \qquad \forall p < 2m-d/2-1.
\end{equation}

 The treatment of $I_0$ then leads to the study of terms of the general form
 $$
J=\int_{\R^d}  (\pa^{\a}_x \pa^\b_v \F)(0,\mathrm{X}(0,t,x,v),v) m(t,x,v)\, dv,
 $$
where, for $j=m-k,\cdots,m+k$, $|\a|+|\b| =  m+k+j$, $|\a|\geq 2j$, and 
$$\|m\|_{L^\infty(0,T; \Hc^{N}_{-r'- r_0})} \leq \Lambda(T,M),$$
for all $N < 2m - d/2 -1$ and all $r'>d/2$.
If $|\b|=0$, there is nothing special to do, as only derivatives in $x$ are involved, so let us assume that $|\b|\geq 1$.
We write $\pa^\b_v = \pa^{\beta'}_v \pa_v$. We have
\begin{multline*}
J = \int_{\R^d} \pa_v\Big[  (\pa^\a_x \pa^{\b'}_v \F)(0,\mathrm{X}(0,t,x,v),v) \Big]m(t,x,v)\, dv \\
- \int_{\R^d}  ( \pa^\a_x \pa^{\b'}_v)( \pa_v \mathrm{X}(0,t,x,v) \cdot \na_x  \F)(0,\mathrm{X}(0,t,x,v),v) m(t,x,v)\, dv,
\end{multline*}
  and thus by integration by parts in $v$, we get
\begin{multline*}
J =- \int_{\R^d} \Big[ (\pa^\a_x \pa^{\b'}_v \F)(0,\mathrm{X}(0,t,x,v),v) \Big] \pa_v m(t,x,v)\, dv \\
- \int_{\R^d}  ( \pa^\a_x \pa^{\b'}_v)( \pa_v \mathrm{X}(0,t,x,v) \cdot \na_x  \F)(0,\mathrm{X}(0,t,x,v),v) m(t,x,v)\, dv. 
\end{multline*}  
We therefore observe that this procedure allows to trade derivatives in $v$ against derivatives in $x$.

 Assume now that one can write, for  some $l \in \{ 1, \cdots,  |\b|\} $,
 $$
 J= \sum_{|\b'|\leq l} \sum_{|\a'| \leq |\a|+ |\b|-l }  \int_{\R^d} \Big[ ( \pa^{\a'}_x \pa^{\b'}_v \F)(0,\mathrm{X}(0,t,x,v),v) \Big]  m_{\a',\b'}(t,x,v)\, dv + R_l
 $$
 where 
 \begin{equation}
 \label{hyp-m}
\| m_{\a',\b'} \|_{L^\infty(0,T; \Hc^{N_l}_{-r'- r_0})} \leq \Lambda(T,M),
\end{equation}
for all $N_l < 2m - d/2 -1-|\b|+l$ and all $r'>d/2$,
 and $R_l$ is a remainder satisfying 
  $$\| R_l \|_{L^2(0,T;L^2_{x})} \leq \Lambda(T,M).$$
 Let us show that this property holds as well for at rank $l-1$.
   Following the same integration by parts argument as above, we may write 
   $$
    J-R_l=  J_1 +J_2+J_3 
    $$
    where
    $$
    \begin{aligned}
      J_1 =&\sum_{|\a'| \leq |\a|+ |\b|-l }  \int_{\R^d} \Big[ (\pa^{\a'}_x \F)(0,\mathrm{X}(0,t,x,v),v) \Big]  m_{\a',0}(t,x,v)\, dv, \\
   J_2 =&- \sum_{|\b'|\leq l}  \sum_{\b' = (\b'', j)}  \sum_{|\a'| \leq |\a|+ |\b|-l } \\
   &\qquad\qquad\qquad \int_{\R^d} \Big[ (\pa^{\a'}_x \pa^{\b''}_v \F)(0,\mathrm{X}(0,t,x,v),v) \Big]  \pa_{v_j} m_{\a',\b'}(t,x,v)\, dv,  \\
  J_3 =  &- \sum_{|\b'|\leq l}  \sum_{\b' = (\b'', j)}    \sum_{|\a'| \leq |\a|+ |\b|-l } \\
  &\qquad\quad \int_{\R^d}  ( \pa^{\a'}_x \pa^{\b''}_v)( \pa_{v_j} \mathrm{X}(0,t,x,v) \cdot \na_x  \F)(0,\mathrm{X}(0,t,x,v),v) m_{\a',\b'}(t,x,v)\, dv.
 \end{aligned}
 $$
   The terms $J_1$ and $J_2$ have the good form already. For $J_3$, by using Leibniz rule, we observe that  we need to study terms of the form
$$
\overline{J} =  \int_{\R^d}  \pa^\gamma_{x,v} \mathrm{X}(0,t,x,v) \pa^{\eta_1}_x \pa^{\eta_2}_v \F(0,\mathrm{X}(0,t,x,v),v) m_{\a',\b'}(t,x,v)\, dv,
$$
with $|\eta_2| \leq |\beta''| =l-1$, $1\leq |\eta_1| \leq |\a'| +1$, and  $|\gamma| =|\a'| + |\b''| -  |\eta_1|- |\eta_2|+ 2$.  

Assume first that $|\eta_1|+|\eta_2| \leq 2m-1$. If $|\eta_1|+|\eta_2| < 2m-1-d$, then by Sobolev embedding we can bound 
$$
\| (1+ |v|^2)^{r/2} (\pa^{\eta_1}_x \pa^{\eta_2}_v \F)(0,\mathrm{X}(0,t,x,v),v) \|_{L^\infty_{x,v}} \leq \|f_0\|_{\Hc^{2m-1}_r} \leq \Lambda(M).
$$
Since $0<|\gamma| \leq 2(m+k)$, we use~\eqref{Fjnewnew} and Lemma~\ref{lem-phi} to get
$$\| \pa^\gamma \mathrm{X} \|_{L^\infty(0,T; L^\infty_v L^2_{x})} \leq \Lambda(T,M).$$
(This is where the elliptic estimate~\eqref{Fj-imp} is crucially used.) 
Furthermore, since $2m-d/2 - 1 -2p>d$, we can bound
$$
\| m_{\a',\b'} \|_{L^\infty(0,T; \mathcal{W}^{0,\infty}_{-r'- r_0})} \leq \Lambda(T,M),
$$
for $r'>d/2$ such that $r> r'  + r_0+d$. 
Therefore such terms satisfy a bound
$$
\| \overline{J}\|_{L^2(0,T; L^2_{x,v})} \leq \Lambda(T,M),
$$
and thus can be put into the  remainder $R_{l-1}$.
 If $|\eta_1|+|\eta_2| \geq 2m-1-d$, 
then  $|\gamma| \leq 2k + d+1$. Since $2m-d-1 >d/2$, we can use this time  $\|\pa^\gamma \mathrm{X}  \|_{L^\infty(0,T; L^\infty_{x,v} )} \leq \Lambda(T,M)$ and again, arguing as in the treatment of $I_0$ in  the proof of Theorem~\ref{thm}, such terms are remainders.

Otherwise $|\eta_1|+|\eta_2| \geq  2m$. Then we have $|\gamma| \leq 2k $
and thus $2(m+k)-|\gamma| \geq 2m$.
We set in this case $m_{\eta_1, \eta_2} :=\pa^\gamma_{x,v}  \mathrm{X} \, m_{\a',\b'}$.
 In order to show that $\pa^\gamma_{x,v}  \mathrm{X} \, m_{\a',\b'}$ has the required regularity, we are  led to study terms of the form
$$
\tilde{J}  = \|\pa^a_{x,v} \pa^\gamma_{x,v}  \mathrm{X} \,  \pa^b_{x,v} m_{\a',\b'} \|_{L^\infty(0,T;\Hc^0_{-r'-r_0})}, \qquad |a| + |b| =N_{l-1},
$$
for all $N_l < 2m - d/2 -1-|\b|+l-1$ and all $r'>d/2$.
Assume first that $|a|<2m-d/2$, then we have $|a| + |\gamma| < 2(m+k)-d/2$ and we use estimate~\eqref{LinftyXtilde} in Lemma~\ref{lem-phi} to get $\|\pa^a_{x,v} \pa^\gamma_{x,v}  \mathrm{X}\|_{L^\infty(0,T; L^\infty_{x,v})} \leq \Lambda(T,M)$, and apply~\eqref{hyp-m}
to bound 
$$ \|  \pa^b_{x,v} m_{\a',\b'} \|_{L^\infty(0,T;\Hc^0_{-r'-r_0})}  \leq \Lambda(T,M).$$
Otherwise, $|a|\geq 2m-d/2$.
Since we have $ 2(m+k)-|\gamma| \geq N_{l-1}$ for all $N_{l-1}<  2m -2  -|\b| + l $,
and we can use estimate~\eqref{LinftyXtilde} in Lemma~\ref{lem-phi} to get
$$\sum_{|a|\leq N_{l-1}} \|\pa^a_{x,v} \pa^\gamma_{x,v}  \mathrm{X} \|_{L^\infty(0,T; L^\infty_v L^2_{x})} \leq \Lambda(T,M).$$
Since $|b| = N_{l-1} - |a| \leq N_{l-1} -2m +d/2$, we have 
$N_l-|b| \geq 2m + 1 -d/2 > d$.
As a result, by~\eqref{hyp-m} and the Sobolev embedding we get that
$$
 \|  \pa^b_{x,v} m_{\a',\b'} \|_{L^\infty(0,T; \mathcal{W}^{0,\infty}_{-r'-r_0})}  \leq \Lambda(T,M).
$$
In all cases, we have obtained
$$
\tilde{J}  \leq \Lambda(T,M).
$$
Therefore  the corresponding terms  of $J_3$ can be written under the form 
$$
 \int_{\R^d}  (\pa^{\eta_1}_x \pa^{\eta_2}_v \F)(0,\mathrm{X}(0,t,x,v),v) m_{\eta_1, \eta_2}(t,x,v)\, dv,
$$
with $\| m_{\eta_1, \eta_2} \|_{L^\infty(0,T; \Hc^{N_{l-1}}_{-r'-r_0})} \leq \Lambda(T,M)$ for all 
$N_{l-1}<  2m -1  -|\b| + (l-1) $ and $r'>d/2$.


We conclude by induction that we can write at rank $l=0$
 $$
 J= \sum_{|\a'| \leq m+k+j}  \int_{\R^d} \Big[ ( \pa^{\a'}_x  \F)(0,\mathrm{X}(0,t,x,v),v) \Big]  m_{\a',0}(t,x,v)\, dv + R
 $$
 with $\| m_{\a',\b'} \|_{L^\infty(0,T; \Hc^{N}_{-r'- r_0})} \leq \Lambda(T,M)$ 
for all $N< 2m-1 -|\b|$ and $r'>d/2$
 and $\| R\|_{L^2(0,T; L^2_{x})} \leq \Lambda(T,M)$ is a remainder. 
 
 We then note that $2m-2 -2k > d$, so that 
 $$\| m_{\a',\b'} \|_{L^\infty(0,T; \mathcal{W}^{0,\infty}_{-r'- r_0})} \leq \Lambda(T,M).$$ 
Arguing as in the previous treatment of $I_0$ in the proof of Theorem~\ref{thm},  we finally conclude that
 \begin{equation}
\| I_0\|_{L^2(0,T; L^2_{x})} \leq  \Lambda(T,M) \sum_{j=m-k}^{m+k} \sum_{|\a|  =m+k-j } \| \pa_x^\alpha f_0\|_{\Hc^{0}_{r}}.
\end{equation}
This allows to conclude the proof.
 
 \end{proof}
 
As already noted in the proof of Theorem~\ref{thm2}, we actually do not need to use Proposition~\ref{proposition-K} to treat the term $I_2$ in view of Theorem~\ref{thm2}: we can indeed rely on the smoothing estimate~\eqref{Fj-imp} on the force instead.  Furthermore,  one can obtain $L^\infty_t$ estimates
 instead of the $L^2_t$  theory that we have developped. 
 This observation implies the following fact:
 replacing~\eqref{Fj-imp} by the slightly weaker estimate (in the sense that it is implied by~\eqref{Fj-imp}):
        \begin{equation}
\label{Fj-imp2}
\begin{aligned}
&\| F^j \|_{L^2(0,t; H^n_x)} \\
& \leq \Gamma^{(j)}_{n} \left(t,  \| m_{\psi_1} \|_{L^\infty(0,t;H^{n-1}_x)}, \cdots, \| m_{\psi_r} \|_{L^\infty(0,t;H^{n-1}_x)}, \sum_{j=1}^{\ell} \|F^{j} (0) \|_{H^n_x} \right).
 \end{aligned}
\end{equation}
together with an associated stability estimate replacing~\eqref{stab} with $L^\infty_t$ norms instead of $L^2_t$ for the moments on the right-hand side,
 then Theorem~\ref{thm2} still holds.
It suffices to estimate all terms (that is to say the moments, $I_0$, $I_1$, $I_2, \cdots$) 
 in $L^\infty(0,T; L^2_x)$ instead of  $L^2(0,T; L^2_x)$ as previously.
This remark is useful in particular to treat the so-called Vlasov-Darwin model from plasma physics, that we introduce in the following paragraph.

\subsection{Vlasov-Darwin} The Vlasov-Darwin system is another model that allows to describe the dynamics of charged particles in a plasma, which stands on stage between Vlasov-Poisson and relativistic Vlasov Maxwell. Like Vlasov-Poisson, it can be derived from Vlasov-Maxwell in the non-relativistic regime, that is to say in the limit $c \to + \infty$.  
The difference is that Vlasov-Darwin happens to be a higher oder approximation than Vlasov-Poisson, see \cite{BK}; in particular it retains self-induced magnetic effects that have disappeared completely in the Vlasov-Poisson dynamics. It reads
\begin{equation}\label{VD} 
\left \{ \begin{aligned}
&\partial_t f + \hat{v} \cdot \nabla_x f + \left(E + \frac{1}{c} \hat{v} \times B\right)\cdot \nabla_v f =0,
\\
&E= -\na_x \phi -\frac{1}{c} \pa_t A, \quad B = \na_x \times A,
\\
&-\na_x \phi = \int_{\R^3} f \, dv  - \int_{\T^3 \times \R^3} f \, dv  dx, \\
&-\Delta_x A= \frac{1}{c} \mathbb{P}\int_{\R^3}\hat{v}  f \, dv  , \quad \na_x \cdot A=0,
\end{aligned}
\right.
\end{equation}  
where $c>0$ is the speed of light and $\mathbb{P}$ denotes the Leray projector. 
The Cauchy problem for the  Vlasov-Darwin system~\eqref{VM} was studied (among many other references):
 \begin{itemize}
 \item for what concerns (global) weak solutions, in \cite{Pal06},
\item  for what concerns strong solutions in \cite{Pal06,See,SAAI}.
\end{itemize}

To embed this system into the abstract framework, we need to make the additional assumption that all initial conditions $f_0$ that are considered are a.e. non-negative. By a standard property of the Vlasov equation, any associated  solution $f(t)$ is also a.e. non-negative.

\noindent $\bullet$ \emph{Assumptions on the advection field.} In this model $a(v) =\hat{v}$, which is already treated for the relativistic Vlasov-Maxwell case.

\noindent $\bullet$ \emph{Assumptions on the force field.} We have the decomposition~\eqref{dec-1}--\eqref{dec-2} as well.
Let us set
$$
E= E_L + E_T, \qquad E_L= \na_x \phi, \qquad E_T= -\frac{1}{c} \pa_t A.
$$
and introduce 
$$\psi_1=1, \psi_2 =\hat{v}, \psi_3 = \frac{\hat{v}\otimes \hat{v}}{\sqrt{1+|v|^2/c^2}}, \psi_3 = Id - m_{\psi_3},$$ 
(so that $r_0=0$) and
$$
m_{\psi_i}= \int_{\R^d} \psi_i  f \, dv,
$$
where $m_{\psi_3}$ and $m_{\psi_4}$ are symmetric matrices.
Since $E_L$ and $E_T$ derive from potentials solving a Poisson equation, we have
$$
\forall t \geq 0, \quad \| (E_L, B) (t) \|_{H^n_x} \lesssim \sum_{i=1}^2 \| m_{\psi_i}(t) \|_{H^{n-1}_x}
$$
and thus $$\| (E_L, B)  \|_{L^\infty(0,t;  H^n_x)} \lesssim \sum_{i=1}^2 \| m_{\psi_i}(t) \|_{L^\infty(0,t; H^{n-1}_x)}.$$
For what concerns $E_T$, this is a little more subtle; this is where we need that $f(t) \geq 0$ a.e. As in 
\cite[Lemma 2.10]{Pal06}, we obtain that $E_T$ satisfies the inhomogeneous elliptic equation
\begin{equation}
\label{eq-ET}
- \Delta E_T + \frac{1}{c}  m_{\psi_4} E_T = - \frac{1}{c} \left( m_{\psi_4} E_L - m_{\psi_2} \times B -  \na_x : m_{\psi_3}  \right).
\end{equation}
We fix the time $t\geq 0$ which is a parameter here (we take the $L^\infty_t$ norm in the end). Let $n >  d $. 
By  \cite[Lemma 2.10]{Pal06},  which relies on the fact  that $m_{\psi_4}$ is actually a \emph{semi-definite} symmetric matrix, it follows that~\eqref{eq-ET} has a unique solution $E_T$ in $H^1_x$, with the bound
$$
\begin{aligned}
\|E_T \|_{H^1_x}  &\lesssim \| m_{\psi_4} E_L \|_{H^{-1}_x} + \| m_{\psi_2} \times B  \|_{H^{-1}_x} + \|\na : m_{\psi_3} \|_{H^{-1}_x} \\
&\lesssim \| m_{\psi_4} E_L \|_{L^2_x} + \| m_{\psi_2} \times B  \|_{L^2_x} + \|m_{\psi_3} \|_{L^2_x} \\
&\lesssim \| m_{\psi_4}\|_{H^n_x}  \| E_L \|_{H^n_x} + \| m_{\psi_2} \|_{H^n_x}  \| B  \|_{H^n_x}  + \|m_{\psi_3}\|_{H^n_x}  \\ 
&\lesssim \left( 1+ \sum_{i=1}^2 \| m_{\psi_i}(t) \|_{H^{n}_x}\right)\left(
 \| m_{\psi_4}\|_{H^n_x} +  \| m_{\psi_2} \|_{H^n_x}   + \|m_{\psi_3}\|_{H^n_x}\right). 
\end{aligned} 
$$
Then assume by induction that we have a bound of the form
\begin{equation}
\label{eq-vd-ind}
\forall k=1,\cdots, N, \qquad  \|E_T \|_{H^k_x} \lesssim \Gamma_{k} \left(\| m_{\psi_1} \|_{H^n_x}, \cdots, \| m_{\psi_4} \|_{H^n_x} \right),
\end{equation}
for $N\leq n$, where $\Gamma_k$ is a polynomial function.
Assume first that $N < n-d/2$.
Let $|\a|= N$. We note that $\pa^\a_x E_T$ satisfies
$$
- \Delta  \pa^\a_xE_T + \frac{1}{c}  m_{\psi_4} \pa^\a_xE_T = - \frac{1}{c} \pa^\a_x \left( m_{\psi_4} E_L - m_{\psi_2} \times B -  \na_x : m_{\psi_3}  \right)- [\pa^\a_x, m_{\psi_4}] E_T.
$$
We have by standard tame Sobolev estimates
\begin{equation}
\label{eq-vd-1}
\begin{aligned}
&\| \pa^\a_x \left( m_{\psi_4} E_L - m_{\psi_2} \times B -  \na_x : m_{\psi_3}  \right) \|_{H^{-1}_x} \lesssim \\
&\qquad\qquad\qquad \left( 1+ \sum_{i=1}^2 \| m_{\psi_i}(t) \|_{H^{n}_x}\right)\left(
 \| m_{\psi_4}\|_{H^n_x} +  \| m_{\psi_2} \|_{H^n_x}   + \|m_{\psi_3}\|_{H^n_x}\right).
 \end{aligned}
\end{equation}
Since  $N<n -d/2$, we  can use the Sobolev embedding to obtain
$$ 
\begin{aligned}
\| [\pa^\a_x, m_{\psi_4}] E_T\|_{H^{-1}_x} &\lesssim \| m_{\psi_4} \|_{W^{N, \infty}_x} \| E_T\|_{H^{N}_x} \\
& \lesssim    \| m_{\psi_4} \|_{H^n_x} \Gamma_{k} \left(\| m_{\psi_1} \|_{H^n_x}, \cdots, \| m_{\psi_4} \|_{H^n_x} \right).
 \end{aligned}
$$
We apply again the $H^1_x$ estimate of \cite[Lemma 2.10]{Pal06} to obtain a bound of the form
$$
\|E_T \|_{H^{N+1}_x} \lesssim \Gamma_{N+1} \left(\| m_{\psi_1} \|_{ H^n_x}, \cdots, \| m_{\psi_4} \|_{H^n_x} \right).
$$
We deduce by induction that for all $N<n-d/2$,
$$
\|E_T \|_{H^{N+1}_x} \lesssim \Gamma_{N+1} \left(\| m_{\psi_1} \|_{H^n_x}, \cdots, \| m_{\psi_4} \|_{ H^n_x} \right).
$$
In particular, since $n>d$, we deduce in particular
\begin{equation}
\label{eq-vd-inf}
\|E_T \|_{L^\infty_x} \lesssim \Gamma \left(\| m_{\psi_1} \|_{H^n_x}, \cdots, \| m_{\psi_4} \|_{ H^n_x} \right).
\end{equation}
Now assume we have~\eqref{eq-vd-ind} for some $N \leq n$. We have the tame Sobolev estimate
$$ 
\begin{aligned}
\| [\pa^\a_x, m_{\psi_4}] E_T\|_{H^{-1}_x} &\lesssim \| m_{\psi_4} \|_{H^n} (\| E_T\|_{H^{N}_x} + \|E_T\|_{L^{\infty}_x}) \\
& \lesssim    \| m_{\psi_4} \|_{H^n_x} \Gamma_{N} \left(\| m_{\psi_1} \|_{H^n_x}, \cdots, \| m_{\psi_4} \|_{H^n_x} \right),
 \end{aligned}
$$
by  \eqref{eq-vd-ind} at rank $N$ and  \eqref{eq-vd-inf}.
Thus using the $H^1_x$ estimate of \cite[Lemma 2.10]{Pal06}, we obtain~\eqref{eq-vd-ind} at rank $N+1$.
By induction, we conclude that
$$
\|E_T \|_{L^\infty(0,T; H^{n+1}_x)} \lesssim \Gamma_{n+1} \left(\| m_{\psi_1} \|_{L^\infty(0,T; H^n_x)}, \cdots, \| m_{\psi_4} \|_{L^\infty(0,T; H^n_x)} \right),
$$
which is an estimate of the requested form~\eqref{Fj-imp2}.  A stability estimate of the same form also holds because of similar considerations.

 \section{On the regularity assumptions of Theorem~\ref{thm}}
 \label{sec-rk}
 
 The goal of this short last section is to discuss  the type of regularity assumptions which could be maybe conceivable for proving propagation of higher reguliarity.
 
 \bigskip
 
\noindent{\bf Example 1.}
Consider the free transport equation
\begin{equation}
\label{eq-ft}
\pa_t f + v  \pa_x f = 0,
\end{equation}
set in $\R \times \R$ to simplify the discussion. Let $\varphi(v)$ be a $C^\infty$  function, with compact support in $[-1/2,1/2]$ and such that $\int_{\R} \varphi \, dv =0$. Let $g$ be the piecewise continuous function defined by $g(x)= 1$ for $x \in [-1,1]$ and $0$ elsewhere.
Observe that in the sense of distributions, we have $g'(x) = \delta_{x=-1} - \delta_{x=-1}$, where $\delta$ stands for the Dirac measure. 
We consider the initial condition 
$$f|_{t=0} = g(x) \varphi(v) \in L^2_{x,v},$$ 
and the solution to~\eqref{eq-ft} reads
$$
f(t,x,v)  = g(x-tv) \varphi(v).
$$
It follows by explicit computations that $\rho(t,x) := \int_{\R} f \, dv$ satisfies
$$
\begin{aligned}
\pa_x \rho (t,x) &= \varphi\left( \frac{x+1}{t} \right) -  \varphi\left( \frac{x-1}{t} \right), \\
\pa_x^k \rho(t,x) &= \frac{1}{t^{k-1}} \left(  \varphi^{(k-1)}\left( \frac{x+1}{t} \right)  -   \varphi^{(k-1)}\left( \frac{x-1}{t} \right)\right), \qquad \forall k \in \N^*.
\end{aligned}
$$
We have for $t<4$,
$$
\| \pa_x^k \rho(t)  \|_{L^2_{x}}^2 =  \frac{1}{t^{2(k-1)}}  \left( \left\| \varphi^{(k-1)}\left( \frac{x+1}{t} \right)\right\|_{L^2_{x}}^2 +  \left\| \varphi^{(k-1)}\left( \frac{x-1}{t} \right)\right\|_{L^2_{x}}^2  \right),
$$
since $\varphi$ is compactly supported in $[-1/2,1/2]$, and thus
$$
\| \pa_x^k \rho(t) \|_{L^2_{x}}^2  =  \frac{2}{t^{2(k-1)-1}}  \| \varphi^{(k-1)} \|_{L^2_x}^2.
$$
We deduce that for any $T>0$, $\rho \notin L^2(0,T; H^2_x)$. However $\rho(0,x) =0 \in H^k_x$ for all $k \in \N$. 

This example shows that regularity of moments at initial time may not be propagated, and a  more precise information such as~\eqref{eq-thm} is somehow required to obtain higher regularity for moments.

\bigskip

\noindent{\bf Example 2.} Consider the equation 
\begin{equation}
\label{eq-ex2}
\pa_t f + v \pa_x f + F(t,x) \pa_v f = 0,
\end{equation}
on $\T \times \R$, with 
$$F (t,x) = \int_{\R} \psi(v) f(t,x,v) \, dv,$$ 
where $\psi \in C^\infty_c(\R^d)$ with compact support in $[-1/2,1/2]$.  It is clear that~\eqref{eq-ex2} enters the abstract framework of this work.

We consider the initial condition 
$$f|_{t=0} =f_0^{(1)} + f_0^{(2)},$$ 
where $f_0^{(2)} $ is a smooth non-negative  function, with support in $\T \times [-1/2,1/2]$ and $f_0^{(1)} $ is a smooth non-negative  function, with support in $\T\times [1,2]$.

Consider $f^{(1)}$ the solution of~\eqref{eq-ex2} associated to the initial condition $f_0^{(1)} $, and assume that it is defined on an interval $[0,T]$, for $T>0$ small enough.
Now define $f^{(2)}$ as the solution on $[0,T]$  of the \emph{linear} kinetic transport equation
$$
\pa_t f + v \pa_x f + \left(\int_{\R} \psi(v) f^{(1)} \, dv\right) \pa_v f = 0,
$$
with initial condition $f_0^{(2)}$.

Because of the form of the force $F$ (notably because $\psi$ is localised in $[-1/2,1/2]$), we observe that up to reducing $T>0$, the solution $f$ on $[0,T]$ of~\eqref{eq-ex2} can be written as 
$$
f = f^{(1)} + f^{(2)},
$$ 
since $T>0$ can be chosen small enough so that the support in velocity of $f^{(2)}(t)$ is disjoint from that of $\psi$, and thus 
$$ \int_{\R} \psi(v) f^{(2)}(t) \, dv =0.$$
Now let $k \in \N$ and  assume that there is $(x_0,v_0) \in \T \times (1,2)$ such that $f|_{t=0}(x_0,v_0) \neq 0$ and is locally $H^k$ around this point. Because of the assumptions on the supports, this is equivalent to ask that  $f_0^{(2)}(x_0,v_0) \neq 0$ and is locally $H^k$ around this point. However we can choose  (independently of $f_0^{(2)}$)  $f_0^{(1)}$ so that $\int_{\R} \psi(v) f^{(1)} \, dv$ is not $H^k$, in such a way that 
$f^{(2)}(t)$ (and thus $f(t)$) is not locally $H^k$ around points of the form $(X(0,t,x_0,v_0),V(0,t,x_0,v_0))$, where $(X,V)$ denote the characteristics associated to $F$, as defined in~\eqref{charac}.

This example shows that local regularity may not be propagated (along characteristics), contrary to what happens for the class of PDEs considered in \cite{Bony}. This is due to the ``non-locality'' in velocity. Therefore a  global regularity assumption is required
in order to obtain propagation of higher regularity.  

This example can (also) be slightly modified, 
in order to prove that a local version of~\eqref{eq-thm} cannot either be propagated into higher local regularity of moments,
see the next (and last) example.

\bigskip

\noindent{\bf Example 3.} Consider the equation 
\begin{equation}
\label{eq-ex3}
\pa_t f + v \pa_x f + F(t,x+1/4) \pa_v f = 0,
\end{equation}
on $\T \times \R$ (here we identify $\T$ with $[0,1)$ with periodic boundary conditions). 
Let us consider as in the previous example
$$F (t,x) = \int_{\R} \psi(v) f(t,x,v) \, dv.$$ 
We consider the initial condition 
$$f|_{t=0} =f_0^{(1)} + f_0^{(2)},$$ 
where $f_0^{(1)} $ is a non-negative  function, with compact support in $[0,1/8] \times \R$ and $f_0^{(2)} $ is a non-negative  function, with compact support in $[1/4,3/8]  \times  \R$.

Observe that because of the shift in the argument of the force, by looking at the supports in $x$, the solution $f^{(2)}$ associated to the initial condition $f_0^{(2)} $ is equal to $f_0^{(2)}(t,x-tv,v) $ on $[0,T]$, for $T>0$ small enough. Moreover, we have
$$
 \left(\int_{\R} \psi(v) f^{(2)}(t,x+1/4,v) \, dv\right) \pa_v f^{(2)} =0.
$$
Now define $f^{(1)}$ as the solution on $[0,T]$  of the \emph{linear} kinetic transport equation
$$
\pa_t f + v \pa_x f + \left(\int_{\R} \psi(v) f^{(2)}_0(x+1/4-tv,v) \, dv\right) \pa_v f = 0,
$$
with initial condition $f_0^{(2)}$.

We observe that up to reducing $T>0$, the solution $f$ on $[0,T]$ of~\eqref{eq-ex2} can be written as 
$$
f = f^{(1)} + f^{(2)}.
$$ 
Indeed, by looking at the supports in $x$, we can impose $T>0$ small enough so that
$$
\begin{aligned}
 &\left(\int_{\R} \psi(v) f^{(1)}(t,x+1/4,v) \, dv\right) \pa_v f^{(2)}= 0, \\
 &\left(\int_{\R} \psi(v) f^{(1)}(t,x+1/4,v) \, dv\right) \pa_v f^{(1)} = 0.
 \end{aligned}
$$
Now let $k \in \N$ and  assume that there is $x_0\in (0,1/8)$ such that $\int_{\R} f|_{t=0}(x_0,v) \,dv \neq 0$ and $f|_{t=0}$ is locally $H^k_x$ around this point. This is equivalent to ask  for the fact that  $\int_{\R} f_0^{(1)}(x_0,v) \,dv \neq 0$ and $f_0^{(1)}$ is locally $H^k_x$ around this point. This corresponds to a local analogue of~\eqref{eq-thm}. However we can choose (independently of $f_0^{(1)}$) $f_0^{(2)}$ so that $\int_{\R} \psi(v) f_0^{(2)}(x-tv,v) \, dv$ is not locally $H^k$, in such a way that 
the moments in velocity of $f^{(1)}(t)$ (and thus of $f(t)$) are not locally $H^k_x$ around points of the form $X(0,t,x_0,v_0)$, for some $v_0 \in \R$.

\bigskip

\noindent {\bf Acknowledgements.} The author  would like to thank Fr\'ed\'eric Rousset, as this work wouldn't have been possible without \cite{HKR}, and J\'er\'emie Szeftel for a stimulating discussion about the second order operators introduced in \cite{HKR} and used in this work as well.

\bibliographystyle{plain}
\bibliography{space-vlasov}

\end{document}